\documentclass[11pt,letterpaper]{amsart}
\usepackage{geometry}                
\usepackage{graphicx}
\usepackage{amssymb}
\usepackage{epstopdf}
\newtheorem{theorem}{Theorem}[section]
\newtheorem{lemma}{Lemma}[section]
\newtheorem{corollary}{Corollary}[section]
\newtheorem{proposition}{Proposition}[section]
\newtheorem{conjecture}{Conjecture}[section]
\newtheorem{definition}{Definition}[section]
\newtheorem{example}{Example}[section]
\newtheorem{remark}{Remark}[section]

\numberwithin{equation}{section}

\def\Z{\mathbb Z}

\def\R{\mathbb R}

\def\N{\mathbb N}

\def\Om{\Omega}

\def\om{\omega}
\def\d{\partial}
\def\s{\sigma}
\def\e{\epsilon}
\def\a{\alpha}
\def\b{\beta}
\def\g{\gamma}
\def\l{\lambda}

\title{Recovering contact forms from boundary data}

 \author[G.~Katz]{Gabriel Katz}
\address{MIT, Department of Mathematics, 77 Massachusetts Ave., Cambridge, MA 02139, U.S.A.}
\email{gabkatz@gmail.com}

\begin{document}

\maketitle 

\begin{abstract} Let $X$ be a compact smooth manifold with boundary. The paper deals with contact $1$-forms $\beta$ on $X$, whose Reeb vector fields $v_\beta$ admit Lyapunov functions $f$. 

We tackle the question: how to recover $X$ and $\beta$ from the appropriate data along the boundary $\partial X$?  We describe such boundary data and prove that they allow for a reconstruction of the pair $(X, \beta)$, up to a diffeomorphism of $X$. We use the term ``holography" for the reconstruction. We say that objects or structures inside $X$ are {\it holographic}, if they can be reconstructed from their $v_\beta$-flow induced ``shadows" on the boundary $\partial X$.  

We also introduce numerical invariants that measure how ``wrinkled" the boundary $\partial X$ is with respect to the $v_\beta$-flow and study their holographic properties under the contact forms preserving embeddings of equidimensional contact manifolds with boundary.  We get some ``non-squeezing results" about such contact embedding, which are reminiscent of Gromov's non-squeezing theorem in symplectic geometry.
\end{abstract}


\section{Introduction}

In this paper, we study contact $1$-forms $\b$ 
on compact connected smooth $(2n+1)$-manifolds $X$ with boundary. We focus on residual structures along the boundary $\d X$ that allow to reconstruct $X$ and the contact form $\b$ on $X$, up to a smooth diffeomorphism. When possible, such a reconstruction deserves the name of ``{\sf holography}".  

In fact, modulo the Holography Theorem from \cite{K4}, our techniques are quite pedestrian, but the questions we tackle here seem to be novel. Some ideas of this paper can be traced back to \cite{K2} and \cite{K4}. \smallskip

Let us describe one special case in which our holography results admit considerable simplifications. Consider  a compact connected smooth manifold $X$ of dimension $2n+1$, equipped with a contact $1$-form $\b$. Recall that a form $\b$ is {\sf contact}, if the top form $\b \wedge (d\b)^n$ is a volume form on $X$.  The contact form $\b$ generates a unique nonsingular vector field $v_\b$ which is called a {\sf Reeb field}. It is characterized by two properties: $\b(v_\b) = 1$ and $d\b(v_\b \wedge \sim) \equiv 0$. In this paper we deal mostly with the Reeb fields that admit Lyapunov functions $f: X \to \R$, i.e., $df(v_\b) > 0$.

The Reeb field $v_\b$ (as well as any non-vanishing in the vicinity of $\d X$ vector field) divides the boundary $\d X$ into two loci:  $\d_1^+X$ and $\d_1^-X$.
Along  $\d_1^+X$, $v_\b$ points in the interior of $X$ or is tangent to $\d X$; along  $\d_1^-X$, $v_\b$ points  outside of $X$ or is tangent to $\d X$. For a generic $v_\b$, the loci $\d_1^+X$, $\d_1^-X$ are compact manifolds that share a common boundary $\d_2X$. Temporarily, for the sake of simplicity, let us assume that, in the vicinity of $\d_2X$, the boundary $\d X$ is {\sf convex} with respect to the $v_\b$-flow. This  crudely means that the behavior of $v_\b$ in the transversal to $\d_2X$ plane resembles the behavior of the vector field $\d_y$ relative to the domain $\{x \geq y^2\} \subset \R^2$. 

The Lyapunov function $f$ forces any $v_\b$-trajectory that originates at a point $x \in  \d_1^+X$ to reach the locus $\d_1^-X$ at some point $y$. This correspondence $x \leadsto y$ produces a map $C_{v_\b}:  \d_1^+X \to \d_1^-X$. We call $C_{v_\b}$ the {\sf causality map}. In the special case of $\d X$ being $v_\b$-convex, $C_{v_\b}$ is continuous. 

The Holography Theorem from \cite{K4} claims crudely that the knowledge of the map $C_{v_\b}$ is sufficient for a reconstruction of $X$ and the un-parametrized dynamics of the $v_\b$-flow, up to a diffeomorphism of $X$ that is the identity on $\d X$. In this paper, we refine this claim. Imagine that, in addition,  we are given the restriction $\b^\d$ of the contact form $\b$ to the boundary $\d X$ and the restriction $f^\d$ to $\d X$ of a Lyapunov function $f: X \to\R$ such that $df(v_\b) =1$. Then our main result, Theorem \ref{th.main},  claims that the triple $(X, \b, f)$ can be reconstructed from the triple $(C_{v_\b}, \b^\d, f^\d)$, up to a diffeomorphism that is that is the identity on $\d X$.

In this paper, we are also preoccupied with problems that are not explicitly related to holography. Let us describe   some of them. 

Recall that a vector field $u$ is called is called a {\sf contact field} for a given contact $1$-form $\b$ on $X$, if $\mathcal L_u \b = \l \cdot \b$ for some smooth function $\l: X \to \R$. The Reeb vector field $v_\b$ is contact since $\mathcal L_{v_\b} \b = 0$. We strive to understand which contact vector fields generate a flow that preserves the manifold $X$, in particular its boundary. Proposition \ref{prop.X_is_FI(u)-invariant} provides an answer in terms of an arbitrary smooth function $h: X \to \R$ such that $0$ is its regular value and $h^{-1}(0) = \d X$.

We also are tackling the following natural question: How to discriminate between Reeb vector fields and general nonsingular vector fields on a given odd-dimensional manifold $X$ with boundary? We are far away from answering this question in its full generality. 
However, in Proposition \ref{cor.Morse-Reeb}, we show that,  for any gradient-like vector field $v$ of a Morse function $f$ on any compact manifold $Y$ of dimension $\geq 7$, the restriction of $v$ to the compliment $X$ of a convex regular neighborhood of the $f$-singular set is {\it not} a Reeb field for any choice of a contact form $\b$ on $X$. 

We suspect that the following notion of a {\sf contact shell} may be relevant to answering the general question about the properties that distinguish between Reeb and generic vector fields. 
 Let $v$ be a smooth non-vanishing vector field on a $(2n +1)$-dimensional compact manifold $X$, and let $\b$ be a $1$-form such that $\b(v)=1$ and $v\, \rfloor \, d\b= 0$. Note that here the property $\b \wedge (d\b)^n > 0$ may be violated somewhere in $X$. Let $S \subset X$ be a connected  oriented \emph{closed} smooth manifold of a dimension $2k$, $k \leq n$, transversal to the $v$-flow. 
  A  {\sf contact $(2k+1)$-shell} $Sh(S, v, \b)$ is a smooth compact submanifold of $X$, diffeomorphic to the product $S \times [0, 1]$ and such that each oriented segment $s \times [0, 1]$, $s \in S$ is mapped diffeomorphically to an oriented segment of a $v$-trajectory contained in $Sh(S, v, \b)$. Most importantly, we require that the restriction of $\b \wedge (d\b)^k$ to the shell $Sh(S, v, \b)$ be a positive volume form. 

In Proposition \ref{prop.NO_shells} we prove that, if $v_\b$ is the Reeb vector field of a contact form $\b$ on a compact $(2n+1)$-dimensional manifold $X$, then $X$ does not contain any contact $(2k+1)$-shells $Sh(S, v_\b, \b)$ for all $k \in [1, n]$.  In particular, the sets $\d_1^\pm X(v_\b)$ have no \emph{closed} connected components. We conjecture that the converse is true: if,  for a $1$-form $\b$ and a vector field $v$ such that $\b(v)=1$ and $v\, \rfloor \, d\b= 0$, the contact shells are absent,  then $\b$ is a contact form.

In the spirit of the Santal\'{o} Formula in Integral Geometry \cite{S}, in Theorem \ref{th.isoperimetric}, we prove the following inequality for Reeb fields $v_\b$ that admit a Lyapunov function:
$$
\int_X \b \wedge (d\b)^n \leq \mathsf{diam}_\mathcal R(\b) \cdot \int_{\d_2X(v_\b)} \b \wedge (d\b)^{n-1}. 
$$
Here $\mathsf{diam}_\mathcal R(\b)$, the {\sf Reeb diameter}, equals $\sup_{\g} \big\{\big |\int_\g \b \big |\big\}$, where the $\sup$ is taken over all the Reeb trajectories $\g$.

Next, let us describe the ``contact non-squeezing" Theorem \ref{th.non-squeezing}. It imposes constraints on contact embeddings $\Psi$ of compact contact $(2n+1)$-manifolds $X$ with boundary into a given contact  $(2n+1)$-manifold $Y$  (see Fig. \ref{fig.3_contact_2}).  Here a ``contact embedding" means that $\b_X= \Psi^\ast(\b_Y)$, where $\b_X$, $\b_Y$ are given contact forms on $X$ and $Y$. We assume that the  Reeb vector field $v_{\b_Y}$ admits a Lyapunov function. 

This kinds of problems have a glorious history:  non-squeezing theorem of Gromov's \cite{Gr2} in Symplectic Geometry, non-squeezing theorem of Eliashberg,  Kim,  and Polterovich \cite{EKP} in Contact Geometry, and its extension by Chiu \cite{Chi} and more recent work by Fraser, Sandon, and Zhang \cite{FSZ} to mention just a few. 

 Theorem \ref{th.non-squeezing} describes in a new  phenomenon:  the $(2n-1)$-volume of the equatorial locus $\d_2Y(v_{\b_Y})$ in the target restricts the $(2n-1)$-volume of the equatorial locus $\d_2X(v_{\b_X})$ in the source $X$, provided that a priory we restrict from above half of the number of times any $v_{\b_Y}$-trajectory can hit the hypersurface  $\Psi(\d X)$. We denote by $c^\bullet(\Psi, v_{\b_Y})$ this upper boundary. The restriction  imposed on the contact embedding  $\Psi$ is described by the formula  
 $$\int_{\d_2X(v_{\b_X})} \b_X \wedge (d\b_X)^{n-1} \leq c^\bullet(\Psi, v_{\b_Y}) \cdot \int_{\d_2Y(v_{\b_Y})} \b_Y \wedge (d\b_Y)^{n-1}.$$
 Example \ref{ex.ellipsoid} contains a translation of what our restrictions mean for the contact embeddings of odd-dimensional ellipsoids in a given target space $(Y, \b_Y)$.\smallskip
 
Then we introduce new numerical invariants of the Reeb flows $v_\b$ on manifolds with boundary. They are based on signed volumes of the $v_\b$-generated Morse strata (see (\ref{eq.2.1})) of the boundary, which measure how ``wrinkled" the boundary is with respect to the Reeb flow. We study the  holographic properties of these invariants under contact embeddings $\Psi$ of equidimensional manifolds (Proposition \ref{prop.t-independent_kappas}, Theorem \ref{th.detecting_volumes}).
\smallskip

Finally, using similar ``holographic" approach, we striving to reconstruct Legendrian submanifolds $L$ of $(X, \b)$ from their shadows $L^\dagger$ of on the screen $\d_1^+X(v_\b)$ (see Fig. \ref{fig.3_contact_3}). 
As before, the Reeb vector field $v_\b$ is traversing. Proposition  \ref{prop.holography_for_Legendrians} describes the  structures on the screen $\d_1^+X(v_\b)$, sufficient for reconstructing $L$ from $L^\dagger$.  
In Corollary \ref{cor.sub_Legendrian_links}, for sub-Legendrian links,  we give a boundary-localized criterion when the $v_\b$-flow through $L$ interacts with a $v_\b$-concave portion of the boundary $\d X$.
\smallskip

%
%
%
%

We tried  to make this text self-contained. 

\section{Traversing, boundary generic, \& traversally generic vector fields}

Let $X$ be a compact connected smooth $(n+1)$-dimensional manifold with boundary. 

\begin{definition}\label{def.2.1_traversing}
A non-vanishing vector field $v$ on $X$ is called {\sf traversing} if each $v$-trajectory is homeomorphic to ether a closed interval, or to a singleton. \smallskip

We denote by $\mathcal V_{\mathsf{trav}}(X)$ the space of all traversing vector fields on $X$. 
\hfill $\diamondsuit$
\end{definition}

The $1$-dimensional singular oriented foliations $\mathcal F(v)$ on manifolds $X$ with boundary, generated by such traversing vector fields $v$, were originally studied in \cite{EG}. The singularities of such foliations are confined to the boundary $\d X$.  In fact, by adding a collar to $X$ to form a larger open manifold $\hat X$ and extending the vector field $v$ to the collar appropriately, we may think of the singular foliation $\mathcal F(v)$ as the restriction to $X$ of a nonsingular foliation $\mathcal F(\hat v)$, defined on $\hat X$. The book \cite{K3} may serve as a general reference source for the properties of such foliations $\mathcal F(v)$.  \smallskip

It follows that any traversing vector field is nonsingular and of the {\sf gradient type} \cite{K1}, i.e., there exists a smooth Lyapunov function $f:X \to \R$ such that $df(v) > 0$ in $X$. Moreover, the converse is true: any non-vanishing gradient-type vector field $v$ is traversing. 
\smallskip

$\diamondsuit$ For the rest of the paper, we assume that the field $v$ on $X$ extends to a  vector field $\hat v$ on some open manifold $\hat X$ which properly contains $X$. We treat the extension $(\hat X, \hat v)$ as a {\sf germ} $Op(X, v)$ that contains $(X, v)$. 
In fact, the treatment of $(X, v)$ will not depend on the germ of extension $(\hat X, \hat v)$, but many constructions are simplified by introducing an extension.
\smallskip

Any smooth vector field $v$ on $X$, which does not vanish along the boundary $\d X$, gives rise to a partition $\d_1^+X(v) \cup \d_1^-X(v)$ of the boundary $\d_1 X =: \d X$ into  two sets: the locus $\d_1^+X(v)$, where the field is directed {\sf inward} of $X$ or is tangent to $\d X$, and  the locus $\d_1^-X(v)$, where it is directed {\sf outward} of $X$ or is tangent to $\d X$. 

We assume that $v|_{\d X}$, viewed as a section of the quotient  line bundle $T(X)/T(\d X)$ over $\d X$, is transversal to its zero section. This assumption implies that both sets, $\d^+_1 X(v)$ and $\d^-_1 X(v)$, are compact manifolds which share a common smooth boundary $\d_2X(v) =: \d(\d_1^+X(v)) = \d(\d_1^-X(v))$. Evidently, $\d_2X(v)$ is the locus where $v$ is \emph{tangent} to the boundary $\d X$.

In his groundbreaking work \cite{M}, M. Morse made an important observation: for a generic vector field $v$, the tangent locus $\d_2X(v)$ inherits a similar structure in connection to $\d_1^+X(v)$, as $\d_1 X$ has in connection to $X$. That is, $v$ gives rise to a partition $\d_2^+X(v) \cup \d_2^-X(v)$ of  $\d_2X(v)$ into  two sets: the locus $\d_2^+X(v)$, where the vector field $v$ is directed inward of $\d_1^+X(v)$ or is tangent to $\d_2X(v)$, and  $\d_2^-X(v)$, where it is directed outward of $\d_1^+X(v)$ or is tangent to $\d_2X(v)$. Again, we assume that the restriction $v|_{\d_2X(v)}$, viewed as a section of the quotient  line bundle $T(\d X)/T(\d_2X(v))$ over $\d_2X(v)$, is transversal to its zero section.

For generic fields, this structure replicates itself: the {\sf cuspidal locus} $\d_3X(v)$ is defined as the locus where $v$ is tangent to $\d_2X(v)$; $\d_3X(v)$ is divided into two manifolds, $\d_3^+X(v)$ and $\d_3^-X(v)$. In  $\d_3^+X(v)$, the field is directed inward of $\d_2^+X(v)$ or is tangent to its boundary, in $\d_3^-X(v)$, outward of $\d_2^+X(v)$ or is tangent to its boundary. We can repeat this construction until we reach the zero-dimensional stratum $\d_{n+1}X(v) = \d_{n+1}^+X(v) \cup  \d_{n+1}^-X(v)$.  


Thus, a  generic vector field $v$  on $X$  gives rise to two  stratifications: 
\begin{eqnarray}\label{eq.2.1}
\d X =: \d_1X \supset \d_2X(v) \supset \dots \supset \d_{n +1}X(v), \nonumber \\ 
X =: \d_0^+ X \supset \d_1^+X(v) \supset \d_2^+X(v) \supset \dots \supset \d_{n +1}^+X(v), 
\end{eqnarray}
the first one by closed submanifolds, the second one---by compact ones.  Here $\dim(\d_jX(v)) = \dim(\d_j^+X(v)) = n +1 - j$. \smallskip

$\diamondsuit$ We will often use the abbreviated notation ``$\d_j^\pm X$" instead of ``$\d_j^\pm X(v)$" when the vector field $v$ is fixed or its choice is obvious. \smallskip

These considerations motivate a more formal definition (\cite{K1}, \cite{K3}).

\begin{definition}\label{def.2.2_boundary-generic} 
Let $X$ be a compact smooth $(n+1)$-dimensional manifold with boundary $\d X \neq \emptyset$, and $v$ a smooth vector field on $X$.  

We say that  $v$  is {\sf boundary generic} if the restriction $v|_{\d X}$ does not vanish and produces a filtrations of $X$ as in (\ref{eq.2.1}). Its strata $\{\d_j^+X \subset \d_jX\}_{1 \leq j \leq n+1}$  are defined inductively in $j$ as follows:

\begin{itemize}
\item $\d_0X =: \d X$, $\d_1X =: \d X$ (hence $\d_0X$ and $\d_1X$---the base of induction---do not depend on $v$).
\item  for each $k \in [1, j]$, the $v$-generated stratum $\d_kX$ is a closed smooth submanifold of  $\d_{k-1}X$,
\item  the field  $v$, viewed as section of the quotient 1-bundle  $$T_k^\nu =: T(\d_{k-1}X)/ T(\d_kX) \longrightarrow \d_kX,$$ is transversal to the zero section of $T_k^\nu \to \d_kX$ for all $k \leq j$. 
\item the stratum $\d_{j+1}X$ is the zero set of the section $v \in T_j^\nu$. 
\item the stratum $\d^+_{j+1}X \subset \d_{j+1}X$ is the locus where $v$ points inside of $\d_j^+X$.
\end{itemize}

We denote the space of boundary generic vector fields on $X$ by the symbol $\mathcal B^\dagger(X)$. \hfill $\diamondsuit$
\end{definition}

The trajectories $\g$ of a boundary generic vector field $v$ on $X$ interact with the boundary $\d_1 X$ so that each point $a \in \g \cap \d_1 X$ acquires a {\sf multiplicity} $m(a) \in \N$, the order of tangency of $\g$ to $\d_1 X$ at $a$. The multiplicity $m(a)$ may be defined as follows. As before, we embed $X$ into an open equi-dimensional $\hat X$ and extend  $v$ to a nonsingular vector field $\hat v$ on $\hat X$. Then we consider a smooth function $z: \hat X \to \R$ such that; {\sf(1)} $0$ is a regular value of $z$, {\sf(2)}  $f^{-1}(0) =\d X$, and {\sf(3)} $f^{-1}((-\infty, 0]) = X$. With these items in place, the multiplicity $m(a)$ is the integer $ j\geq 1$ such that the $k$-iterated $\hat v$-directional derivatives $\{\mathcal L_{\hat v}^k\}_k$ satisfy the equations 
\begin{eqnarray}\label{eq.MULTIPLICITY}
(\mathcal L_{\hat v}z)(a) =0,\,  (\mathcal L_{\hat v}^2z)(a) =0,\, \ldots ,\, (\mathcal L_{\hat v}^{j-1}z)(a) = 0, \text{ but }(\mathcal L_{\hat v}^{j}z)(a) \neq 0.
\end{eqnarray}

We associate a {\sf divisor} $D_\g = \sum_{a\, \in\, \g \cap \d_1 X} m(a)\cdot a$ with each $v$-trajectory $\g$. In fact, for any boundary generic $v$, the multiplicity $m(a) \leq \dim(X)$ and the support of $D_\g$ is finite (\cite{K2}, Lemma 3.1).  
\smallskip

Therefore, we may associate also a finite {\sf ordered} sequence $\om(\g) = (\om_1, \om_2, \dots, \om_q)$ of {\sf multiplicities} with each $v$-trajectory $\g$. The multiplicity $\om_i$ is the order of tangency between the curve $\g$ and the hypersurface $\d_1 X$ at the $i^{th}$ point of the finite set $\g \cap \d_1 X$. The linear order in $\g \cap \d_1 X$ is determined by $v$.   

For each trajectory $\g$ of a boundary generic and traversing $v$, we introduce two important quantities:
\begin{eqnarray}\label{eq.2.4}
m(\g) =: \sum_{a \in \g \cap \d_1X }\; m(a), \; \text{ and } \; m'(\g) =: \sum_{a \in \g \cap \d_1X }\; (m(a) -1),
\end{eqnarray} 
the {\sf multiplicity} and the {\sf reduced multiplicity} of the trajectory $\g$.
 
Similarly, for a sequence $\omega = (\omega_1, \omega_2, \, \dots \, ,  \omega_q)$, we introduce its {\sf norm} and its {\sf reduced norm} by the formulas: 

\begin{eqnarray}\label{eq.2.5}
|\omega| =: \sum_i\; \omega_i \quad \text{and} \quad |\omega|'  =: \sum_i\; (\omega_i -1).
\end{eqnarray}
Note that $q$, the cardinality of the {\sf support} of $\omega$, is equal to $|\omega| - |\omega|'$.
\smallskip

We consider also an important subclass of traversing and boundary generic fields $v$, which we call {\sf traversally  generic} (see Definition \ref{def.2.5_traversally_generic_fields}  below and Definition 3.2 from \cite{K2}). Such vector fields admit $v$-flow-adjusted semi-local coordinate systems, in which the boundary is given by a quite special polynomial equations, 
and all the trajectories are parallel to the preferred coordinate axis (see \cite{K2}, Lemma 3.4). More importantly, for such vector fields, the {\sf smooth topological type} of the flow in the vicinity of each trajectory of a combinatorial type $\om$ is determined by $\om$ (see \cite{K1})! \smallskip

Let us define the notion of a traversally  generic vector field. Given a boundary generic and traversing vector field $v$, for each trajectory $\g$, consider the finite set $\g \cap \d_1X = \{a_i\}_i$ and the collection of tangent spaces $\{T_{a_i}(\d_{j_i}X^\circ)\}_i$ to the pure strata $\{\d_{j_i}X^\circ\}_i$. Note that each point $a_i$ belongs to a unique stratum with the {\it maximal possible} index $j_i =: j(a_i)$. In fact, each space $T_{a_i}(\d_{j_i}X^\circ)$ is transversal to the curve $\g$ \cite{K2}.  
 \smallskip 

Let $S$ be a local transversal section of the $\hat v$-flow, an extension of $v$, at a point $a_\star \in \g$, and let $\mathsf T_\star$ be the space tangent to the section $S$ at $a_\star$. Each space $T_{a_i}(\d_jX^\circ)$, with the help of the $\hat v$-flow, determines a vector subspace $\mathsf T_i = \mathsf T_i(\g)$ in  $\mathsf T_\star$. It is the image of  the tangent space $T_{a_i}(\d_jX^\circ)$ under the composition of two maps: 

(1) the differential of the $v$-flow-generated diffeomorphism that maps $a_i$ to $a_\star$, and

(2) the linear  projection $T_{a_{\star}}(X) \to \mathsf T_\star$ whose kernel is generated by $v(a_\star)$. 
\smallskip

The configuration $\{\mathsf T_i\}$ of {\sf affine} subspaces  $\mathsf T_i \subset \mathsf T_\star$ is called {\sf generic} (or {\sf stable}) when all the multiple intersections of spaces from the configuration have the least possible dimensions, consistent with the dimensions of $\{\mathsf T_i\}$. In other words, $$\textup{codim}(\bigcap_{s} \mathsf T_{i_s},  \mathsf T_\star) = \sum_s \textup{codim}(\mathsf T_{i_s},  \mathsf T_\star)$$ for any subcollection $\{\mathsf T_{i_s}\}$ of spaces from the list $\{\mathsf T_i\}$.

Consider the case when $\{\mathsf T_i\}$ are vector subspaces of $\mathsf T_\star$.  If we interpret each subspace $\mathsf T_i$ as the kernel of a linear epimorphism $\Phi_i:  \mathsf T_\star \to \R^{n_i}$, then the property of  $\{\mathsf T_i\}$ being generic can be reformulated as the property of  the direct product map $\prod_i \Phi_i:  \mathsf T_\star \to \prod_i  \R^{n_i}$ being an epimorphism.  In particular, for a generic configuration of affine subspaces, if a point  belongs to several  $\mathsf T_i$'s, then the sum of their codimensions $n_i$ does not exceed the dimension of the ambient space $\mathsf T_\star$. 
\smallskip

The definition below resembles and is inspired by the ``{\sf Condition NC}" imposed on, so called, {\sf Boardman maps} (see \cite{Bo}) between smooth manifolds (see \cite{GG}, page 157, for the relevant definitions). In fact, for a generic traversing vector fields $v$, the $\hat v$-flow delivers germs of Boardman maps  $p(v, \g): \d_1X \to \R^n$, available in the vicinity of each trajectory $\g$ \cite{K2}. Here $\R^n$ is identified with a transversal section $S_\star$ of the $\hat v$-flow in the vicinity of a point $\star \in \g$.

\begin{definition}\label{def.2.5_traversally_generic_fields}  A traversing vector field $v$ on $X$ is called {\sf traversally generic} if: 
\begin{itemize}
\item  the field is boundary generic in the sense of Definition \ref{def.2.2_boundary-generic} ,
\item for each $v$-trajectory $\g \subset X$,
the collection of  subspaces $\{\mathsf T_i(\g)\}_i$  is generic in $\mathsf T_\star$; that is, the obvious quotient map $\mathsf T_\star \to \prod_i \big(\mathsf T_\star/ \mathsf T_i(\g)\big)$ is surjective (this property does not depend on the choice of the point $\star \in \g$ and of the space $\mathsf T_\star$, transversal to $\g$).
\end{itemize}

We denote by $\mathcal V^\ddagger(X)$ the space of all traversally generic vector  fields on $X$. \hfill $\diamondsuit$
\end{definition}

In fact, $\mathcal V^\ddagger(X)$ is an open and dense in the space of all traversing vector fields on $X$ \cite{K1}. 

For a \emph{traversally  generic} vector field $v$ on a $(n+1)$-dimensional $X$, the trajectory space $\mathcal T(v)$ is stratified by subspaces, labeled with the elements $\om$ of a universal  poset $\mathbf\Om^\bullet_{|\sim|' \leq n} \subset \mathbf \Om^\bullet$ that is defined by the constraint $|\om|' \leq n$. It depends only on $\dim(X) = n+1$ (see \cite{K3} for the definition and properties of $\mathbf\Om^\bullet_{|\sim|' \leq n}$).

 
\section{Holography on manifolds with boundary and causality maps}

Let $X$ be a compact connected smooth $(2n+1)$-dimensional manifold with boundary and $v$ a smooth {\sf traversing} vector field. Such $v$ admits a {\sf Lyapunov function} $f: X \to \R$ so that $df(v) > 0$.  We assume that $v$ is boundary generic in the sense of Definition \ref{def.2.2_boundary-generic}.   

Let $\mathcal F(v)$ be the $1$-dimensional singular {\sf oriented foliation}, generated by the traversing $v$-flow. 

We denote by $\g_x$ the $v$-trajectory through $x \in X$. Since $v$ is traversing, 
each $\g_x$ is homeomorphic either a closed segment, or to a singleton.\smallskip

A traversing vector field $v$ on $X$ induces a structure of a {\sf partially-ordered set} $(\d_1 X, \succ_v)$ on the boundary $\d_1 X$: for $x, y \in \d X$, we write $y \succ x$ if the two points lie on the same $v$-trajectory $\g$ and  $y$ is reachable from $x$ by moving in the $v$-direction. \smallskip
  
We denote by $\mathcal T(v)$ the {\sf trajectory space} of $v$ and by $\Gamma: X \to \mathcal T(v)$ the obvious projection. For a traversing $v$, $\mathcal T(v)$ is a compact space in the topology induced by $\Gamma$. 
Since for a traversing $v$, any trajectory intersects the boundary, we get that $\mathcal T(v)$ is a quotient of $\d_1 X$ modulo the relation $ \succ_v$. 

A traversing and boundary generic $v$ gives rise to the {\sf causality (scattering) map} 
\begin{eqnarray}\label{def.causal_map}
C_v: \d_1^+X(v) \to \d_1^-X(v)
\end{eqnarray}
(see \cite{K2} or \cite{K4}) that takes any point $x \in \d_1^+X(v)$ to the unique consecutive point $y \in \g_x \cap \d_1^-X(v)$, $y \neq x$, which can be reached from $x$ in the $v$-direction. If no such $y$ is available, we put $C_v(x) = x$.  We stress that typically $C_v$ is a {\it discontinuous} map. 

We notice that, for any smooth positive function $\l: X \to \R_+$, we have $C_{\l\cdot v} = C_v$. Thus, the causality map depends only on the {\sf conformal class} of a traversing vector field $v$.
\smallskip

In the paper, we will discuss two kinds of related holography problems in the context of contact forms and traversing Reeb vector flows. To avoid technicalities, we prefer to state them first in rather general and nebulous way, to be adapted later to the context of the Reeb flows that admit Lyapunov functions. 

The first kind of problem amounts to the question: 

{\sf To what extent given boundary data are sufficient for reconstructing the unknown bulk $X$ and the traversing $v$-flow on it?} \smallskip

This question may be represented symbolically by the following two diagrams.
 \begin{eqnarray}\label{3_Reconstruction Holo} 
\bullet  \text{\sf Holographic Reconstruction Problem}\nonumber \\ 
(\d_1X, \; \succ_v, )\; & \stackrel{\mathbf{??}}{\longrightarrow} & \;
 (X,\; \mathcal F(v)),\\
 (\d_1X, \; \succ_v, \; f^\d)\; & \stackrel{\mathbf{??}}{\longrightarrow} & \;
 (X,\; \mathcal F(v),\; f),
 \end{eqnarray}
 where $\succ_{v}$ denotes the partial order on boundary, defined by the causality map $C_v$, and the symbol ``$\stackrel{\mathbf{??}}{\longrightarrow}$" indicates the unknown ingredients of the diagram. Here $f$ denotes a Lyapunov function for the vector field $v$,  $\mathcal F(v)$ denotes the $1$-dimensional oriented foliation, produced by $v$, and the function $f^\d =: f|_{\d X}$. 
\smallskip

The second kind of problem is stated as follows.\smallskip 

{\sf Given two manifolds, equipped with traversing flows, and a transformation (a smooth diffeomorphism) $\Phi^\d$ of their boundaries, respecting the relevant boundary data (such as the partial order $\succ_v$ on the boundary $\d X$ or/and the function $f^\d$), is it possible to extend $\Phi^\d$ to a transformation $\Phi$ of the manifolds that respects the corresponding flows-generated structures in their interior?} \smallskip

 This problem is may be represented by the commutative diagrams:
\begin{eqnarray}\label{3_Extension Holo_A}
\bullet  \text{\sf Holographic Extension Problem \quad} \nonumber \\ 
 (\d_1X_1, \; \succ_{v_1})\; \stackrel{\mathsf{inc}}{\longrightarrow} \; (X_1,\; \mathcal F(v_1))\nonumber \\ 
 \quad \downarrow \; \Phi^\d  \quad \quad  \quad \quad \quad  \quad \downarrow \; ?? \;\; \Phi \quad 
 \\
 (\d_1X_2, \; \succ_{v_2})\; \stackrel{\mathsf{inc}}{\longrightarrow} \; (X_2,\; \mathcal F(v_2)) \nonumber \\
 \nonumber \\
(\d_1X_1, \; \succ_{v_1},\; f_1^\d)\; \stackrel{\mathsf{inc}}{\longrightarrow} \; (X_1,\; \mathcal F(v_1), f_1)\nonumber \\ 
 \quad \downarrow \; \Phi^\d  \quad \quad  \quad \quad \quad  \quad \downarrow \; ?? \;\; \Phi \quad 
 \\
 (\d_1X_2, \; \succ_{v_2}, \; f_2^\d)\; \stackrel{\mathsf{inc}}{\longrightarrow} \; (X_2,\; \mathcal F(v_2),\; f_2).\nonumber 
 \end{eqnarray}
 Again, $f_i$ is a Lyapunov function for the traversing vector field $v_i$, and $f_i^\d$ stands for the restriction of $f_i$ to the boundary of $X_i$, $i= 1,2$.
 
 These two types of problems come in a big variety of flavors, depending on the more or less rich boundary data and on the anticipated quality of the transformations $\Phi$ (homeomorphisms, $\mathsf{PD}$-homeomorphisms, H\"{o}lder homeomorphisms with some control of the H\"{o}lder exponent, and diffeomorphisms) (see \cite{K2}, \cite{K4}, \cite{K5}).
 
 
 \section{On contact forms and contact vector fields on manifolds with boundary}
For the reader convenience, we present a few basic properties of contact forms and contact vector fields we will rely upon.\smallskip

We adopt the notation $Op(Y)$ for the vicinity of a topological subspace $Y$ in the ambient space $X$.  \smallskip

Given a vector $v$ and an exterior $k$-form $\a$, in what follows, we denote by  $v \,\rfloor \, \a$ the $(k-1)$-form $\a(v \wedge \sim)$. \smallskip

$\bullet$ Let $\a$ be a differential $k$-form on a smooth manifold $X$ and let $Z \subset X$ be a smooth submanifold. We use somewhat unconventional notations: we write ``$\a|_Z$" for the restriction of the section $\a: X \to \bigwedge^k T^\ast X$ to the section $\a|: Z \to \bigwedge^k T^\ast Z$, and we write  ``$\a ||_Z$" for the restriction of $\a$ to the  section of the bundle  $\bigwedge^k T^\ast X|_Z \to Z$ ($\a ||_Z$ not to be confused with the restriction of $\a$ to the submanifold $Z \subset X$).

\begin{definition}\label{def.contact}
A {\sf coorientable contact structure} on a $(2n+1)$-dimensional smooth orientable manifold $X$ is a $2n$-dimensional distribution $\xi$ of cooriented hyperplanes in the tangent bundle $TX$, generated by the kernels of a differential {\sf contact $1$-form} $\b$ on $X$ such that $\b \wedge (d\b)^n > 0$ everywhere (equivalently, the restriction $(d\b)^n|_\xi > 0$ everywhere). 
\smallskip

An {\sf coorientable almost contact structure} $\xi$ on a $(2n+1)$-dimensional smooth manifold $X$ is a $2n$-dimensional distribution $\eta$ of cooriented hyperplanes in the tangent bundle $TX$, generated by the kernels of a differential $1$-form $\b$ on $X$, a nonvanishing vector field $v$ such that $\b(v) =1$  (so, $v$ is transversal to the distribution $\eta$), and a nondegenerate skew-symmetric $2$-form $\om$ on $\eta$ (equivalently, a complex structure $J: \eta \to \eta$ such that $J^2 = - \mathsf{id}$ and $\om(J(u), J(w)) = \om(u, w)$ for all $u, w \in \eta$; the choice of such $J$ is unique up to a homotopy).  
\hfill $\diamondsuit$
\end{definition}
\smallskip
 
 If $\l: X \to \R_+$ is a positive smooth function, then by a direct computation, $$(\l \b) \wedge (d(\l \b))^n = \l^{n+1} \b \wedge \Big(d\b + \frac{d\l}{\l}\wedge \b\Big)^n = \l^{n+1}[\b \wedge (d\b)^n]  > 0.$$ 


\begin{definition}\label{def.Reeb}
A vector field $v_\b$ on $X$ is called the {\sf Reeb field} for a contact $1$-form $\b$, if $\b(v_\b) =1$ and $v_\b \in \ker(d\b)$.\hfill $\diamondsuit$
\end{definition}

The property $\b \wedge (d\b)^n \neq 0$ implies that the kernel $\ker(d\b)$ is $1$-dimensional. 
%
%
%
%
Thus, any contact form $\b$ generates a unique Reeb vector field $v_\b$. \smallskip

 %

Consider a cooriented contact structure $\xi$, generated by a contact form $\b$. Then any other contact form $\b'$ that generates the same $\xi$ is proportional to $\b$ with the coefficient of proportionality being a smooth positive function. The Reeb vector fields $v_\b$ and $v_{\b'}$, although both being transversal to $\xi$, may have fundamentally different dynamics.  

We will say that such Reeb vector fields are {\sf amenable} to $\xi$. \smallskip

Recall the following well-known fact.

\begin{lemma} For a contact structure $\xi_\b$, generated by a contact $1$-form $\b$,  there is a $1$-to-$1$ correspondence between the space $C^\infty(X)$ of smooth functions $f$ and the space $\mathcal Reeb(X, \xi_\b)$ of Reeb vector fields amenable to $\xi_\b$. 

The correspondence is given by the formula $f \longrightarrow v_{e^f\cdot \b}.$  
\end{lemma}

\begin{proof} Let us fix a contact form $\b$ such that $\ker(\b) = \xi_\b$. Then any other contact form $\b'$, producing $\xi_\b$ can be written as $e^f \b$ for some smooth function $f$. Let us compare two contact forms, $e^f \b$ and $e^g \b$.
By the definition of Reeb fields, we get $e^f \b(v_{e^f \b}) = 1 =  e^g \b(v_{e^g \b})$.  If  $v_{e^f \b} = v_{e^g \b}$, this is equivalent to the equation $e^{f} = e^{g}$, or to the equation $f= g$. Therefore, the correspondence $f \longrightarrow v_{e^f\cdot \b}$ is a continuous bijection.
\end{proof}

Let $X$ be a compact connected smooth $(2n+1)$-dimensional  manifold $X$ with boundary. 


\begin{definition}\label{contact_field} A smooth vector field $u$ on $X$ is called  {\sf contact field} for a given contact $1$-form $\b$, if $\mathcal L_u \b = \l \cdot \b$ for some smooth function $\l: X \to \R$. \hfill $\diamondsuit$
\end{definition}

Since $\mathcal L_u(\mu\b)= (d\mu(u) + \mu\l)\b$, where $\l, \mu$ are smooth functions and $\mathcal L_u(\b) = \l\b$, the notion of contact vector field depends only on the contact structure $\xi_\b$, generated by a contact form $\b$. 
\smallskip




To derive the next Proposition \ref{prop.X_is_FI(u)-invariant} we will rely on a couple of lemmas.

\begin{lemma}\label{lem.beta_w_invariant} For a vector field $w \in \ker(\b)$,  $\mathcal L_w \b = 0$ if and only if $w = 0$. 
\end{lemma}

\begin{proof} By the formula $\mathcal L_w \b = w\rfloor d\b + d(w\rfloor \b)  = w\rfloor d\b$, we get $\mathcal L_w \b = 0$ if and only if $w\rfloor d\b = 0$. Since the symplectic form  $d\b|_{\ker\b}$ is non-degenerated, it follows that $w = 0$.
\hfill
\end{proof}

\begin{lemma}\label{lem.boundary}
Let $v$ be a traversing vector field on a compact connected oriented smooth $(2n+1)$-dimensional  manifold $X$ with boundary. 
Assume that two cooriented contact structures $\xi_1$ and $\xi_2$ on $X$ share the same Reeb vector field $v$ of their contact forms $\b_1$ and $\b_2$.

If the distributions-subbundles  $\xi_1^\d \subset T_\ast(X)\,|_{\d_1X}  \text{ and \,} \xi_2^\d \subset T_\ast(X)\,|_{\d_1X}$ 
 coincide, then $\xi_1= \xi_2$.
\end{lemma}

\begin{proof} Let $\b_i$ ($i = 1, 2$) be a $1$-form whose kernel is the hyperplane distribution $\xi_i$. By the definition of  Reeb vector field $v$, we have $v \rfloor d\b_i = 0$ and $v \rfloor \b_i =1$. The Cartan's identity $\mathcal L_{v}\b_i =  v \rfloor d\b_i + d(v \rfloor \b_i)$ implies that the directional derivative $\mathcal L_{v}\b_i =0$, i.e., $\b_i$ is $v$-invariant. Thus, the $1$-form $\b_i^\d$, viewed as a section of the bundle $T^\ast X|_{\d_1^+X(v)}$, spreads by  the $v$-flow uniquely along each $v$-trajectory to produce the invariant $1$-form $\b_i$. 
Using that $v$ is traversing, we conclude that $\b_i^\d$ determines $\b_i$ everywhere in $X$. 
As a result, the distributions $\xi_1, \xi_2$ are $v$-invariant. Thus, $\xi_1$ is determined by $\xi_1^\d$ and $\xi_2$ is determined by $\xi_2^\d$. Therefore, if  $\xi_1^\d= \xi_2^\d$, then $\xi_1= \xi_2$. 
%
%
\end{proof}

Given a contact form $\b$, any vector field on $X$ can be written as $u = h\cdot v_\b + w$, where $v_\b$ is the Reeb vector field, $h: X \to \R$ is a smooth function, and the vector field $w \in \xi_\b$.
Then $\mathcal L_u\b = \l\cdot \b$ for some smooth function $\l: X \to \R$ if and only if
$$\mathcal L_u\b \equiv \mathcal L_{h\cdot v_\b + w}\b = 
h \cdot \mathcal L_{v_\b}\b  + (v_\b\,  \rfloor \, \b)\, dh  + \mathcal L_w\b = \l \cdot \b.$$
Using that  $\mathcal L_{v_\b}\b \equiv 0$ and $v_\b\,  \rfloor \, \b = 1$, this is equivalent to the condition 
$1 \cdot dh  + \mathcal L_w\b = \l \cdot \b,$
or, by Cartan's formula and using that $w\,  \rfloor \, \b = 0$, to the condition
\begin{eqnarray}\label{eq.contact_field}
w \,  \rfloor \, d\b = - dh + \l \cdot \b.
\end{eqnarray}
Since $v_\b \,  \rfloor \, d\b = \mathbf 0$, (\ref{eq.contact_field}) implies that 
$$0= - w \,  \rfloor \, (v_\b \,  \rfloor \, d\b) = v_\b \,  \rfloor \, (w \,  \rfloor \, d\b) = v_\b \,  \rfloor \, (- dh + \l \cdot \b) = -dh(v_\b) + \l.$$
Thus, $\l = dh(v_\b)$, and the equation (\ref{eq.contact_field}) transforms into the equation  
\begin{eqnarray}\label{eq.contact_field_A}
w \,  \rfloor \, d\b = - dh + dh(v_\b) \cdot \b
\end{eqnarray}
with respect to $w  \in \xi_\b$.
Restricting (\ref{eq.contact_field_A}) to the hypersurface $\{h = c\}$ gives the equations
 \begin{eqnarray} \label{eq.contact_field_C}
 (w \,  \rfloor \, d\b)\,|_{h^{-1}(c)} & = & - dh(v_\b) \cdot\b\,|_{h^{-1}(c)},\\
 (w \,  \rfloor \, \b)\,|_{h^{-1}(c)} & = & 0. \nonumber 
 \end{eqnarray}
 
\begin{lemma}\label{lem.unique_w} Let $\b$ be a contact form. For any choice of the smooth function $h$, there exists a unique vector field $w \in \ker\b$ that satisfies the equation  (\ref{eq.contact_field_A}).
\end{lemma}

\begin{proof} 
Since $d\b|_{\ker(\b)}$ is a non-degenerated skew-symmetric bilinear form, for any $1$-form $\a$, viewed as a section of the dual bundle $(\ker(\b))^\ast \to X$, there exists a section $w(\a)$ of the bundle $\ker(\b) \to X$ such that $w(\a) \,  \rfloor \, d\b = \a$. Therefore, the restriction of equation (\ref{eq.contact_field_A}) to $\ker(\b)$ gives the solvable equation
\begin{eqnarray}\label{eq.contact_field_B} 
(w \,  \rfloor \, d\b) |_{\ker(\b)} = - (dh) |_{\ker(\b)}
\end{eqnarray}
with respect to the vector field $w \in \ker \b$.
By the same token, the solution $w$ of (\ref{eq.contact_field_B}) is unique.  

If $w \in \ker \b$ is the solution  of (\ref{eq.contact_field_B}), then it is automatically a solution of (\ref{eq.contact_field_A}), since $v_\b \,  \rfloor \,(w \,  \rfloor \, d\b) = v_\b \,  \rfloor \,(- dh + dh(v_\b)) \cdot \b$ produces the identity $0=0$.
%
\end{proof}

Recall that the Gray Stability Theorem \cite{Gray} claims that, for any smooth family $\{\xi_t\}_{t \in [0, 1]}$ of contact structures on a \emph{closed} smooth manifold $X$, there exists a smooth isotopy $\{\phi_t\}_{t \in [0, 1]}$ such that $(\phi_t)^\ast(\xi_t) = \xi_0$ for all $t \in [0, 1]$. A version of the Gray Stability Theorem is available for manifolds with boundary; among the multitude of other places, its formulation and proof may be found  in \cite{EM}, page 95. 
\smallskip

%


We can generalize a bit the {\it relative} Gray Stability Theorem 
by assuming that $\dot \b_t ||_{\d_1X} = dh_t ||_{\d_1X}$, where $t \in [0,1]$, $\dot\b_t =: \frac{d}{dt} \b_t$, and the functions $h_t$ satisfy the three properties, described in the next lemma. Since the argument validating  this lemma is a slight variation of the standard Moser trick \cite{Mo}, we omit lemma's proof.


\begin{lemma}\label{lem.isotopy_of_xi} 
 Let $X$ be a compact connected smooth $(2n+1)$-dimensional  manifold $X$ with boundary. Let $\{h_t: \hat X \to \R\}_{t \in [0,1]}$ be a family of smooth functions, each of which {\sf (1)} has $0$ as its regular value, {\sf (2)} $h_t^{-1}(0) = \d_1X$, and {\sf (3)} $h_t^{-1}((-\infty, 0]) = X$.  

Consider a family $\{\xi_t\}_{t \in [0, 1]}$ of cooriented contact structures on $X$, generated by a smooth family of contact forms $\{\b_t\}_{t \in [0, 1]}$, such that the $t$-derivative $\dot \b_t = dh_t$ along $\d_1X$.\footnote{The property $\dot \b_t ||_{\d_1X} = dh_t ||_{\d_1X}$ implies that $\b_t |_{\d_1X}$ is constant.} 
 \smallskip

 Then there is a smooth $t$-family of diffeomorphism $\{\phi_t: X \to X\}_{t \in [0,1]}$, 
 such that  $(\phi_t)^\ast(\xi_t) = \xi_0$; in particular, $\phi_1$ is a contactomorphism that takes $ \xi_1$ to $\xi_0$. \hfill $\diamondsuit$
\end{lemma}

\begin{lemma}\label{lem.tangent_w} Let $\b$ be a contact form. For any choice of smooth function $h$, the vector field $w \in \ker\b$ that satisfies (\ref{eq.contact_field_A}) is tangent to any hypersurface $\{h = c\}_{c \in \R}$. 
\end{lemma}

\begin{proof} It follows from (\ref{eq.contact_field_A}) that $0= w \,  \rfloor \,(w \,  \rfloor \, d\b) =  w \,  \rfloor \, ( - dh + dh(v_\b)\cdot \b)$. Since $w \,  \rfloor \, \b = 0$, we get $w \,  \rfloor \, dh = 0$. Thus, $w$ is tangent to any hypersurface $h^{-1}(c) \subset X$, where $c \in \R$.
\end{proof}

For $h$ such that $h^{-1}(0) = \d_1 X$, by Lemma \ref{lem.tangent_w}, the vector field  $u = h\cdot v_\b + w$ is tangent to $\d_1X$,  since $h|_{\d_1X} \equiv 0$. Thus, we getting the following claim.

\begin{proposition}\label{prop.X_is_FI(u)-invariant} Let $\b$ be a contact form on $X$. For any smooth function $h: \hat X \to \R$ which has the three properties listed in (\ref{eq.function for dX}), 
and for any vector field $w \in \ker(\b)$ that satisfies the equation $$w \,  \rfloor \, d\b = - dh + dh(v_\b) \cdot \b,$$ 
the one-parameter family diffeomorphisms $\{\phi_t(w): X \to X\}_t$ that integrates the field $w$ is well-defined for all $t \in \R$. 

Also, the one-parameter family diffeomorphisms $\{\psi_t(u): X \to X\}_t$ that integrates the contact vector field $u = h\cdot v_\b + w$, which satisfies the equation $\mathcal L_u \b = dh(v_\b)\cdot \b,$ is well-defined for all $t \in \R$. \hfill $\diamondsuit$
\end{proposition}



\noindent {\bf Example 4.1.} 

 $\bullet$ Let $X = \{x^2 + y^2 +z^2 \leq 1\} \subset \R^3$ and $\b = dz + x\, dy$. The Reeb vector field $v_\b$ is $\d_z$. 
 It is tangent to the $2$-sphere $\d_1 X$ along its equator $E = \{z=0\} \cap X$. 
 Thus, the restriction of $\b$ to $E$ is equal to the restriction of the form $x\, dy$ to $E$. 
If $\theta$ is the natural parameter on the circle $E$, then $\b|_E = cos^2(\theta)\, d\theta$. It is non-negative everywhere on $E$ and vanishes at $\theta = \pm \pi/2$, or at $(x, y, z) = (0, \pm 1, 0)$.    
 \smallskip
 
 If in (\ref{eq.contact_field_A}), we pick $h = (x^2 + y^2 +z^2 -1)/2$  then the equation (\ref{eq.contact_field_A}) reduces to 
 $$w \,  \rfloor \, (dx\wedge dy) = -(x\, dx + y\, dy + z\, dz) + z(dz + x\, dy).$$ 
 Coupling this equation with the equation $w \,  \rfloor \, (dz + x\, dy) = 0$, we get a formula 
 $$w = (xz-y)\,\d_x + x\, \d_y -x^2\, \d_z\; $$
 for the  vector field $w \in \xi_\b$ which is tangent to the concentric spheres. Note that the fixed point set of the $w$-flow is the $z$-axis. Finally, put $$u = h\cdot v_\b + w= (xz-y)\d_x + x\, \d_y + \frac{1}{2}\big(-x^2 +y^2 +z^2 - 1\big)\d_z.$$ The $u$-flow preserves the distribution $\xi_\b$ and the sphere $\d_1X$. It has exactly two fixed points, $(0, 0, \pm 1)$.  \smallskip \smallskip
 
 $\bullet$ We pick the shell $\{1 \leq x^2 + y^2 +z^2 \leq 4\}$ for the role of $X$ and the same contact form $\b = dz + x\, dy$. Then $\d_2^+X(v_\b)$ is the equator $E_1 = \{x^2 + y^2 +z^2 =1\} \cap \{z=0\}$ and $\d_2^-X(v_\b)$ is the equator $E_2 = \{x^2 + y^2 +z^2 =4\} \cap \{z=0\}$. The same computation leads to $\b|_{E_1} \geq 0$ and $\b |_{E_2} \leq 0$. In other words, $\pm \b |_{\d_2^\pm X(v_\b)} \geq 0$.
\smallskip

Now we consider the function $h = (x^2 + y^2 +z^2 - 4)(x^2 + y^2 + z^2 -1)$. We leave to the reader to solve the equation 
$w \,  \rfloor \, d\b = - dh + dh(v_\b) \cdot \b$
for $w \in \xi_\b$.\smallskip

$\bullet$ The next example is also given by $\b = dz + x\, dy$, but $h = z$. The choice of $X \subset \R^3$ is unimportant. In this case, $v_\b = \d_z$ and  (\ref{eq.contact_field_A}) reduces to
$$w \,  \rfloor \, (dx\wedge dy) = - dz + (dz + x\, dy) = x\, dy.$$
Thus, $w = x\,\d_x \; \in \xi_\b$ and $u = x\, \d_x + z \d_z$. The vector field $w$ is tangent to the planes $\{z = c\}_c$ and the vector field $u$ is tangent to the plane $\{z = 0\}$. The $u$-flow keeps the form $\b$ invariant. The plane $\{x = 0\}$ is the fixed point set of  the $w$-flow and the line $\{x= 0, z= 0\}$ is the fixed point set of  the $u$-flow.
\hfill $\diamondsuit$
\smallskip

\smallskip

\noindent {\bf Example 4.2.} 
\smallskip

$\bullet$ Let $X = \{x_1^2 + y_1^2 + x_2^2 + y_2^2+z^2 \leq 1\} \subset \R^5$ and $\b = dz + x_1 dy_1 + x_2 dy_2$. We pick $h = (x_1^2 + y_1^2 + x_2^2 + y_2^2+z^2 -1)/2$. The Reeb vector field $v_\b = \d_z$ and the $3$-form 
  $$\b \wedge d\b = dz \wedge dx_1 \wedge dy_1 +  dz \wedge dx_2 \wedge dy_2 
  + x_1\, dy_1 \wedge dx_2  \wedge dy_2 + x_2\, dy_2 \wedge dx_1  \wedge dy_1.$$
 Let $E$ be the equator $\{z=0\} \cap \{x_1^2 + y_1^2 + x_2^2 + y_2^2 =1\}$, a $3$-sphere $\d_2^-X(v_\b) \subset \d_1X = S^4$. Thus $\b \wedge d\b|_E = (x_1\, dy_1 \wedge dx_2  \wedge dy_2 + x_2\, dy_2 \wedge dx_1  \wedge dy_1)|_E$.
 
 A somewhat lengthy computation in the spherical coordinates $\psi, \theta \in [0, \pi], \phi \in [0, 2\pi]$ on the $3$-sphere $E$ leads to the formula 
 $$\b \wedge d\b|_E = \big[(sin^2\psi\, cos^2\,\psi+ sin^4\psi\, sin^2\theta\, cos^2\phi)\, sin\theta\big]\; d\psi \wedge d\theta \wedge d\phi \; \geq 0.$$    
 
Therefore, $(\b \wedge d\b)|_{\d_2^- X(v_\b)} \geq \;  0.$ Note that $(\b \wedge d\b)||_{\d_2^- X(v_\b)} = 0$ along the circle $\{x_1= 0\} \cap  \{x_2= 0\} \cap E$ in the $(y_1, y_2)$-plane. \hfill $\diamondsuit$

\smallskip
 
 These examples motivate the following conjecture about the Morse stratifications (see (\ref{eq.2.1})) defined by Reeb vector fields $v_\b$.
\begin{conjecture}\label{conj.positive_on_Morse_strata}
Let $\b$ be a smooth contact form on a compact connected manifold $X$ and $v_\b$ its boundary generic Reeb vector field. Then: 
\begin{itemize}
\item For an even $j$,\; $\b|_{\mathsf{int}(\d_j^+X(v_\b))}$ is a contact form away from a set of  measure zero. Moreover, 
$$\mp\big[\b\wedge (d\b)^{\frac{2n-j}{2}}\big] \big |_{\d_{j}^\pm X(v_\b)} \geq 0, \text{ and }
\b\wedge (d\b)^{\frac{2n-j}{2}} \big |\big |_{\d_{j+1}X(v_\b)} \equiv 0\; .$$

\item For an odd $j$, $d\b|_{\mathsf{int}(\d_j^+X(v_\b))}$ is a symplectic form away from a set of measure zero. Moreover,  
$$\pm (d\b)^{\frac{2n+1-j}{2}} \big |_{\d_{j}^\pm X(v_\b)} \geq 0, \text{ and }
(d\b)^{\frac{2n+1-j}{2}} \big |\big |_{\d_{j+1}X(v_\b)} \equiv 0.\quad \diamondsuit$$
\end{itemize}
\end{conjecture}





\bigskip









\section{How to discriminate between generic and  Reeb vector fields?}

The next few lemmas are dealing with the way a contact form $\b$ and its Reeb vector field $v_\b$ interact with the boundary $\d_1 X$.

\begin{lemma}\label{lem.positive_on_d_+}
Let $\b$ be a contact form on a $(2n+1)$-dimensional manifold $X$ and $v_\b$ its  boundary generic Reeb vector field. \smallskip

Then $(d\b)^n |_{\mathsf{int}(\d_1^+X(v_\b))} > 0$, while $(d\b)^n |_{\mathsf{int}(\d_1^-X(v_\b))} < 0$. Thus, the restriction \hfill\break $d\b|_{\mathsf{int}(\d_1^+X(v_\b))}$ is a symplectic $2$-form.
\end{lemma}

\begin{proof} Since $\b \wedge (d\b)^n$ is the volume form on $X$, by the definition of $\d_1^+X(v_\b)$, we conclude that $v_\b\, \rfloor \, (\b \wedge (d\b)^n)|_{\d^+X(v_\b)} >  0$ on $\mathsf{int}(\d_1^+X(v_\b))$. On the other hand, since $v_\b \,\rfloor \, (d\b)^n =0$,  $$v_\b \,\rfloor \, (\b \wedge (d\b)^n) = (v_\b \,\rfloor \, \b) (d\b)^n  - \b \wedge (v_\b \,\rfloor \, (d\b)^n) = (d\b)^n.$$ Therefore, $(d\b)^n|_{\mathsf{int}(\d_1^+X(v_\b))} >0$.  As a result, $d\b|_{\mathsf{int}(\d_1^+X(v_\b))}$ is a symplectic form.
\end{proof}


\begin{corollary} Let $\b$ be a contact form on $(2n+1)$-dimensional $X$, whose Reeb vector field $v_\b$ is boundary generic. Then $\int_{\d_2X(v_\b)} \b \wedge (d\b)^{n-1} > 0$.
\end{corollary}

\begin{proof} Since, by Lemma \ref{lem.positive_on_d_+}, $(d\b|_{\mathsf{int}(\d_1^+X(v_\b))})^n > 0$,  by Stokes' Theorem, $$\int_{\d_2X(v_\b)} \b \wedge (d\b)^{n-1} = \int_{\d_1^+X(v_\b)} (d\b)^n > 0. $$
\end{proof}

Consider a vector field $v$ and a $1$-form $\b$ such that $\b(v) = 1$ and $v\,\rfloor\, d\b = 0$.
If we would know that $\b \wedge (d\b)^n > 0$ on $X$, we would conclude that $\b$ is a contact form and $v$ is its Reeb vector field. However, the property $\b \wedge (d\b)^n > 0$ does not hold automatically. The next proposition spells out some hypotheses that insure this desired property.

\begin{proposition}\label{prop.when_beta_contact} Let $v$ be a traversing boundary generic vector field on a connected compact smooth $(2n+1)$-manifold with boundary. 
Consider any $1$-form $\b$ such that $\b(v) =1$ and $d\b(v) =0$. (By Lemma \ref{lem.df(v)=1}, 
 such a form $\b$ exists.) 
%
%
Assume that $\d_2^-X(v) = \emptyset$.\footnote{Speaking informally, this happens when all the $v$-trajectories that are tangent to $\d_1X$ are tangent ``from the interior of $X$" so that the boundary $\d_1X$ is {\it concave} with respect to the $v$-flow.} \smallskip

Then  $\b$ is a contact form and $v$ is its Reeb vector field
%
if and only if  $(d\b)^n \big |_{\mathsf{int}(\d_1^+X(v))} > 0.$ 
 
\end{proposition}

\begin{proof} We notice that $v \,\rfloor \, (d\b) =0$ and $v\,\rfloor \,\b =1$  imply that $\mathcal L_v \b = 0$. Thus, $\b$ is $v$-invariant. 

By the same two properties of $\b$, we get 
$$v \,\rfloor \, (\b \wedge (d\b)^n) = (v \,\rfloor \, \b) (d\b)^n  - \b \wedge (v \,\rfloor \, (d\b)^n) = (d\b)^n.$$ 

Hence, if $(d\b)^n |_{\mathsf{int}(\d_1^+X(v))} > 0$, then $\big(v \,\rfloor \, (\b \wedge (d\b)^n)\big)|_{\mathsf{int}(\d_1^+X(v))} > 0$. 

Let $w_1, w_2, \ldots w_{2n}$ be a frame in the tangent space $T_x \d_1^+X(v)$ whose orientation agrees with the orientation of $\d_1^+X(v)$, induced by the preferred orientation of $X$. Then $v \wedge w_1 \wedge w_2 \wedge \ldots \wedge w_{2n}$ is a positive volume.  Thus, $(\b \wedge (d\b)^n)(v \wedge w_1 \wedge w_2 \wedge \ldots \wedge w_{2n}) > 0$ along $\mathsf{int}(\d_1^+X(v))$.  Using that $\b$ is $v$-invariant, we conclude that $\b \wedge (d\b)^n$ is a positive volume form along each $v$-trajectory $\g$ such that $\g \cap \mathsf{int}(\d_1^+X(v)) \neq \emptyset$. Thus, in $X \setminus \d_2^-X(v)$, the form $\b \wedge (d\b)^n$ is a positive.  If $\d_2^-X(v) =\emptyset$ (i.e., the boundary is concave with respect to $v$), we conclude that $\b \wedge (d\b)^n > 0$ everywhere in $X$. 
As a result, $\b$ is a contact form with the Reeb vector field $v_\b = v$. 

Conversely, if $\b$, subject to the properties $\b(v) =1$ and $d\b(v) =0$, is such that $\b \wedge (d\b)^n > 0$ in $X$, then $v$ is the Reeb field for $\b$, and, by Lemma \ref{lem.positive_on_d_+},  $(d\b)^n \big |_{\mathsf{int}(\d_1^+X(v))} > 0$.
\end{proof}
\begin{remark} 
\emph{
Proposition \ref{prop.when_beta_contact} suggests that should be a geometric connection between cohomology classes with the values in the groups $H^{\ast+1}(X; \pi_\ast(\mathsf{SO}(2n)/\mathsf{U}(n))$ that obstruct the existence of complex (symplectic) structure on a transversal to $v$ distribution $\xi(v)$ and the geometry of the loci in $\d_1^+X(v)$ where $(d\b)^n |_{\d_1^+X(v)}$ flips sign. Note that this change in sign can be expressed as the \emph{global} property $\int_{{\d_1^+X(v)}} |(d\b)^n|  > \int_{{\d_1^+X(v)}} (d\b)^n$.  
\hfill $\diamondsuit$
}
\end{remark}
\smallskip

The next simple example testifies that not any \emph{traversing} vector field $v$ on a $(2n+1)$ -dimensional manifold $X$ is a Reeb vector field for some contact form $\b$, even when a transversal to $v$ $(2n)$-dimensional distribution $\xi$ admits a complex/symplectic structure. 
\smallskip

\begin{example}\label{ex.NO_Eliash}
\emph{
Consider the $3$-fold $X = S \times [0, 1]$ where $S$ is an oriented closed surface. Let $v$ be a constant vector on $X$, tangent to the fibers of the projection $S \times [0, 1] \to S$. Then Proposition \ref{prop.when_beta_contact} tells us that such a $v$ is not a Reeb vector field of any contact form $\b$ on $X$. Indeed, any transversal to $v$ plane field $\xi$ is orientable and thus admits a complex/symplectic structure. However, $d\b$ cannot be positive on $\d^+_1X(v) = S \times \{0\}$ since $\int_S d\b =0$ by Stokes' Theorem. 
}
\hfill $\diamondsuit$
\end{example}
\smallskip

Let us generalize  Example \ref{ex.NO_Eliash}.

\begin{definition}\label{def.3._SHELL} Let $v$ be a smooth non-vanishing vector field on a $(2n +1)$-dimensional compact manifold $X$, and let $\b$ be a $1$-form such that $\b(v)=1$ and $v\, \rfloor \, d\b= 0$. 
Let $S \subset X$ be a connected  oriented \emph{closed} smooth manifold of a dimension $2k$, $k \leq n$, transversal to the $v$-flow.  
\smallskip

A  {\sf contact $(2k+1)$-shell} $Sh(S, v, \b)$ is a smooth compact submanifold of $X$, diffeomorphic to the product $S \times [0, 1]$ so that each oriented segment $s \times [0, 1]$, $s \in S$ is mapped diffeomorphically to an oriented segment of a $v$-trajectory contained in $Sh(S, v, \b)$. Moreover, we require that the restriction of $\b \wedge (d\b)^k$ to the shell $Sh(S, v, \b)$ be a positive volume form. 

\hfill $\diamondsuit$ 
\end{definition}





\begin{proposition}\label{prop.NO_shells} 
If $v_\b$ is the Reeb vector field of a contact form $\b$ on a compact $(2n+1)$-dimensional manifold $X$, then $X$ does not contain any contact $(2k+1)$-shells $Sh(S, v_\b, \b)$ for all $k \in [1, n]$. 

In particular, the sets $\d_1^\pm X(v_\b)$ have no \emph{closed} connected components.
\end{proposition}

\begin{proof} Let $\b$ be a contact form whose Reeb field is $v_\b$. 
Assume that $X$ contains a contact $(2k+1)$-shell $Sh(S, v_\b, \b)$.  

Consider the form $\b \wedge (d\b)^k$ on $Sh(S, v_\b, \b)$. By Definition \ref{def.3._SHELL}, the form $\b \wedge (d\b)^k$ must be a volume form on $Sh(S, v_\b, \b)$. Thus, $\b^\bullet =: \b|_{Sh(S, v_\b, \b)}$ is a contact form on  $Sh(S, v_\b, \b)$ and $v^\bullet_\b =: v_\b|_{Sh(S, v_\b, \b)}$ is the Reeb vector field for $\b^\bullet$. Moreover, $\d_1^+\big(Sh(S, v_\b, \b)\big)(v^\bullet) \approx S \times \{0\}$ and  $\d_1^-\big(Sh(S, v_\b, \b)\big)(v^\bullet) \approx S \times \{1\}$.

On the other hand, Proposition \ref{prop.when_beta_contact} tells us that such $v^\bullet_\b$ cannot be the Reeb vector field of the contact form $\b^\bullet$. Indeed, the form 
 $(d\b^\bullet)^k$, being exact, cannot be positive on the closed manifold $\d^+_1(S \times [0, 1])(v^\bullet_\b) = S \times \{0\}$ since, by Stokes' Theorem, $\int_S (d\b^\bullet)^k =  \int_S\, d\big(\b^\bullet \wedge (d\b^\bullet)^{k-1}\big) = 0$. 
 This contradiction proves the proposition.
\hfill
\end{proof}

\begin{corollary} Let $\b$ be a contact form on a 
smooth $(2n+1)$-manifold $X$. Let the Reeb vector field $v_\b$ be transversal to a smooth  oriented hypersurface $Y \subset X$, where $Y$ is an open $(2n)$-manifold.  \smallskip

Then $Y$ does not contain any \emph{closed} smooth  $(2k)$-dimensional,  $k \in [1, n]$, submanifold $S$, such that $d\b|_S$ is symplectic form. 

In particular, this claim is valid for $Y = \mathsf{int}(\d_1^+X(v_\b))$.
\end{corollary}

\begin{proof}  Let $U$ be a smooth regular neighborhood of $Y$ in $X$. Consider the $(-v_\b)$ induced projection $Q^\b: U \to Y$. We denote its differential by $Q^\b_\ast$. Since $v_\b$ is transversal to $Y$, the isomorphism $Q^\b_\ast |: \xi_\b |_Y \to T(Y)$ is a symplectomorphism with respect to the  forms $d\b|_{\xi_\b}$ and $d\b|_Y$. 

Assume that a closed $2k$-dimensional $S \subset Y$ is such that $d\b|_S$ is symplectic form. Then $(d\b)^k |_S > 0$.
It is on the level of definitions to check that $(Q^\b)^{-1}(S)$ is a contact shell $Sh(S, v_\b, \b)$ residing in $U \subset X$. This contradicts to Proposition \ref{prop.NO_shells}. Thus, the corollary follows.
\hfill
\end{proof}

Motivated by these results and feeling a bit adventurous, let us state the following  conjecture.

\begin{conjecture} Let $X$ be a compact smooth $(2n+1)$-dimensional manifold. 
Let $v$ be a non-vanishing (traversing?) vector field on $X$, and $\b$ a  $1$-form 
such that $\b(v)=1$ and $v\, \rfloor \, d\b = 0$. 

Then $v$ is a Reeb vector field of the contact form $\b$ on $X$ if and only if $X$ does not admit  contact $(2k+1)$-shells $Sh(S, v, \b)$ for all $k \in [1, n]$. \hfill $\diamondsuit$
\end{conjecture}


\begin{corollary}
Let $X$ be a compact smooth $(2n+1)$-dimensional manifold. We denote by $\mathcal W(X)$  the space of all pairs $(v, \b)$,  where a smooth vector field $v$ and a $1$-form $\b$ satisfy the relations $\b(v) =1$ and $v \, \rfloor \, d\b = 0$.  Let  $k \in [1, n]$.\smallskip

Then the subset $\mathcal W_k^{\mathsf{Sh}}(X)$ of the space $\mathcal W(X)$, consisting of pairs $(v, \b)$ that admit contact $(2k+1)$-shells, is open in the $C^\infty$-topology of $\mathcal W(X)$. 
\end{corollary}

\begin{proof} Assume that a $(2k+1)$-shell $Sh(S, v, \b) \subset X$ exists. Thus, $\b \wedge (d\b)^k|_{Sh(S, v, \b)} > 0$. If a vector field $v'$ is $C^\infty$-close to $v$, then the shell $Sh(S, v', \b)$ is $C^\infty$-close to the shell $Sh(S, v, \b)$ thanks to the smooth dependence of solutions of ODEs that share the same initial data $S$ on the nonsingular fields transversal to $S$. In particular, if a point $x$ on a $v$-trajectory through $s \in S$ and  a point $x'$ on a $v'$-trajectory through $s \in S$ are sufficiently close in $X$, then the tangent  spaces $T_{x'}(Sh(S, v', \b))$ and $T_x(Sh(S, v, \b))$ are sufficiently close. Thus, if $\b \wedge (d\b)^k|_{T_{x}(Sh(S, v, \b))} > 0$, then $\b \wedge (d\b)^k|_{T_{x'}(Sh(S, v', \b))} > 0$ as well. It  follows that $\b \wedge (d\b)^k|_{T(Sh(S, v', \b))} > 0$ globally for such choice of $v'$. For a fixed $v'$, the property $\b \wedge (d\b)^k|_{T(Sh(S, v', \b))} > 0$ is open with respect to small smooth deformations $\b \leadsto \b'$ of the form $\b$.  Therefore, for a fixed closed $S$, the constraint $\b \wedge (d\b)^k|_{T_{x}(Sh(S, v, \b))} > 0$ defines an open set $\mathcal O(S)$ in the space $\mathcal V(X) \times \Om^1(X)$, where $\mathcal V(X)$ is the space of nonsingular vector fields and $\Om^1(X)$ the space of $1$-forms on $X$. 

As a result, the intersection $\mathcal W_k^{\mathsf{Sh}}(X)$ of $\mathcal O(S)$ with the subspace $\mathcal W(X)$, defined by $\{\b(v) =1, \; v \, \rfloor \, d\b = 0\}$,  must be open in the subspace topology.
\hfill
\end{proof}

Let $Z$ be a compact smooth $(2n+1)$-dimensional manifold and $f: Z \to \R$ a Morse function. Let $\tilde v$ be a $f$-gradient-like vector field on $Z$. We denote by $X$ the complement in $Z$ to a disjoint union of small open $(2n+1)$-balls $\{B_z\}$ that are centered on the critical points $\{z\}$ of $f$ and are convex in the local Euclidean metric $g_z$, adjusted to the local Morse coordinates at $z$.
A similar construction makes sense for any closed $1$-form $\a$ on $Z$ with the Morse type singularities. In their vicinity, $\a = df$, where $f$ has the standard Morse form in the Morse coordinates at each singularity. Recall that a vector field $v$ is of the gradient type for $\a$ on $X = Z \setminus \bigcup_{z \in \mathsf{sing}(\a)} B_z$ if $\a(v) > 0$ on $X$.

\begin{corollary} If $Z$ is a closed odd-dimensional manifold, then the gradient-like vector field $v$ on the manifold $X = Z \setminus \bigcup_{z \in \mathsf{crit}(f)} B_z$ is not a Reeb field of any contact form.
\end{corollary}

\begin{proof} Since any Morse function $f: Z \to \R$ on a closed $Z$ has maxima and minima, the spheres $S_z = \d B_z$ surrounding such critical points give rise to contact $(2n+1)$-shells. By Proposition  \ref{prop.NO_shells}, $v$ cannot be a Reeb vector field. 
\hfill
\end{proof}

As Corollary \ref{cor.Morse-Reeb} testifies, not any traversing vector field $v$ is the Reeb field for some contact form $\b$! Some obstructions for finding such contact forms $\b$ are of the homotopy-theoretical nature, others (see Proposition \ref{prop.NO_shells} below) are more of the differential topology nature.\smallskip

\begin{corollary}\label{cor.Morse-Reeb}
Let $Z$ be a \emph{compact} smooth manifold of an odd dimension that exceeds $7$. Consider a Morse function $f: Z \to \R$ such that $\mathsf{crit}(f) \neq \emptyset$ and its gradient like vector field $v$. 

Then $v$ is not a Reeb vector field of any contact form on  $X = Z \setminus \bigcup_{z \in \mathsf{crit}(f)} B_z.$

Moreover, for any $1$-form $\b$ on $X$ such that $\b(v) = 1$ and $v \,\rfloor \, d\b = 0$, the form $(d\b)^n \big |_{\mathsf{int}(\d_1^+X(v))}$ must flip sign.
\end{corollary}

\begin{proof}
For any $f$-critical point $z$, consider the sphere $S^{2n}_z = \d B_z$.
Let $a \in H^{2n}(S^{2n}_z; \Z)$ be a generator. Let $\tilde K(S^{2n})$ denote the reduced complex $K$-theory of the sphere $S^{2n}$. 

For each complex vector bundle $\eta$ over $S^{2n}$, the $n$-th Chern class $c_n(\eta)$ is a multiple of $(n-1)!\, a$. Moreover, for each $m$ such that $m \equiv 0 \mod (n-1)!$, there is a unique $[\eta] \in \tilde K(S^{2n})$ such that $c_n(\eta) = m a$ (see \cite{Huse}, 9.8 Corollary).  

On the other hand, for any Morse singularity $z$, the Euler class of a $(2n)$-dimensional  distribution $\xi(v)|_{S^{2n}_z}$, tangent to the hypersurfaces of $f$-constant level, can be computed by the formula $\chi(\xi(v)) =   (-1)^{\mathsf{ind}_z(f)}\, 2$, where $\mathsf{ind}_z(f)$ is the Morse index of $z$.  This can be seen by applying the Gaussian map to $v$ along $S^{2n}_z$, while considering the orthogonal to $v$ hyperplane field $v^\perp \approx \xi(v)|_{S^{2n}_z}$. Therefore, if the real $2n$-bundle $\xi(v)$ has a complex structure, then for each $f$-critical point $z$, we must have $(-1)^{\mathsf{ind}_z(f)}\, 2  \equiv 0 \mod (n-1)!$. 
As a result, a necessary condition for this construction to lead to a contact structure on the $(2n+1)$-dimensional $X$ is 
$$(-1)^{\mathsf{ind}_z(f)}\; 2 \equiv 0 \mod (n-1)! \text{ for each } z\in \mathsf{crit}(f).$$
Unfortunately, this arithmetic condition can be satisfied only for $n \leq 3$. Therefore, no gradient vector field $v$ of a Morse function $f: Z \to \R$ is the Reeb vector field of any contact form $\b$ on $X = Z \setminus (\bigcup_{z \in \mathsf{crit}(f)}\; D_z)$ such that $v_\b = v|_X$ in all dimensions $\dim X \geq 9$. \smallskip
 %
 
 If we identify a small vector $u \in T_zZ$ with the point $exp_{g(z)}(u) \in B_z$ in the Euclidean metric $g(z)$ in the vicinity of $z$, in terms of the Hessian of $f$ at $z$, we get
$$\d_1^+X(v) = \coprod_{z \in \mathsf{crit}(f),\; u \in T_zZ} \{Hess(f)_z(u, u) \geq 0\} \cap \d B_z,$$ and $$\d_2X(v) = \d_2^+X(v) = \coprod_{z \in \mathsf{crit}(f),\; u \in T_zZ} \{Hess(f)_z(u, u) =0 \} \cap \d B_z.$$

Therefore, $v$ is concave in relation to $\d_1X$ (i.e., $\d_2^-X(v) = \emptyset$), and Proposition \ref{prop.when_beta_contact} is applicable. Hence, by the argument above, for any $1$-form $\b$ such that $\b(v) = 1$ and $d\b(v) = 0$, the form $(d\b)^n \big |_{\mathsf{int}(\d_1^+X(v))}$ must flip sign, provided $\dim X \geq 9$. 
\smallskip
\hfill
\end{proof}


Corollary \ref{cor.Morse-Reeb} admits the following straightforward  generalization. 

\begin{corollary} Let $\tilde v$ be a vector field on an odd-dimensional manifold $Y$ such that the zeros of $\tilde v$ form a finite set $\mathcal Z_{\tilde v}$. Let $U(\mathcal Z_{\tilde v})$ be an open  regular neighborhood of $\mathcal Z_{\tilde v}$ in $Y$. Put $X = Z \setminus U(\mathcal Z_{\tilde v})$ and let $v = \tilde v|_X$. 

Then a necessary condition for $v$ being a Reeb vector field of a contact form $\b$ on $X$ is 
$$(-1)^{\mathsf{ind}_z(v)}\, 2  \equiv\; 0 \mod (n-1)! \text{ for each } z \in \mathcal Z_{\tilde v},$$ 
where $\mathsf{ind}_z(v)$ denotes the local index of the vector field $v$ at $z$.
\hfill $\diamondsuit$
\end{corollary}

Since a connected manifold with boundary admits a handle decomposition without handles of the top dimension $2n+1$, there is no homotopy-theoretical obstruction for extending any {\sf almost contact structure} (see Definition \ref{def.contact}) from the base of each handle to its interior (so that the preconditions for applying the $h$-principle are valid). Therefore, by  Gromov's Theorem \cite{Gr1} (see also its enforcement, the  fundamental Theorem 1.1 of Borman, Eliashberg, and Murphy \cite{BEM}), the cooriented contact structure $\xi$ on $X$ exists. However, we do not know whether any connected oriented manifold with boundary admits a contact form $\b$ whose Reeb vector field is {\it traversing}. The next conjecture suggests the positive answer.


\begin{conjecture}\label{th.existence_of_traversing_Reeb} 
%
Let  $X$ be a compact connected oriented smooth manifold  with  boundary, $\dim(X)= 2n+1$. Let $\b$ be a contact form on $X$ such that its Reeb vector field $v_\b$ is transversal to some $\hat v_\b$-invariant codimension one foliation $\mathcal G$ on an open $(2n+1)$-manifold $\hat X$ properly containing $X$.  Here the vector field $\hat v_\b$ on $\hat X$ is such that $\hat v_\b |_X = v_\b$.\smallskip

Then there is a smooth  isotopy  $\{\phi_t: X \to X\}_{t \in [0,1]}$ such that  the Reeb vector field $v_{\b_1}$ of the contact form $\b_1 = \phi_1^\ast(\b)$ is traversing (see Definition \ref{def.2.1_traversing}) on $\phi_1(X) \subset X$. 
\hfill $\diamondsuit$
\end{conjecture}


\section{Holography: recovering contact geometry from  the boundary data}

We start with a lemma that is not related to Contact Geometry.

\begin{lemma}\label{lem.df(v)=1} Let $v$ be a smooth traversing vector field and $\b$ a $1$-form on a compact  connected manifold $X$ such that $\b(v) = 1$.  Then $v$ admits a Lyapunov function $f: X \to \R$ such that $df(v) = 1$. 
In particular,  letting $\b' = df$, we get an exact Lyapunov $1$-form $\b'$ such that 
$\b'(v) = 1$ and $v \, \rfloor\, d\b' = 0$.

Moreover, for any $v$-trajectory $\g$ and any closed interval $[a, b] \subset \g$, we have $\int_{[a, b]} df = \int_{[a, b]} \b$.
\end{lemma}

\begin{proof} 
Let $\dim X = d$. By definition of a traversing vector field, there exists a Lyapunov function $h: X \to \R$ such that $dh(v) > 0$. We embed properly $X$ in a $d$-dimensional open manifold $\hat X$ so that $v$ extends to $\hat v$ and $h$ to $\hat h$, subject to the constraint $d\hat h(\hat v) > 0$ in $\hat X$. By shrinking $\hat X \supset X$ if necessary, we  may also extend $\b$ to $\hat \b$ in $\hat X$ so that $\hat\b(\hat v) =1$ in the vicinity of $X$ in $\hat X$.

For a traversing (and thus nonsingular) $v$-flow on a compact manifold $X$, the trajectory space $\mathcal T(v)$ is a compact Hausdorff space; in fact, for a traversally generic $v$, $\mathcal T(v)$ is $CW$-complex \cite{K1}, \cite{K3}. 

Using that each $v$-trajectory $\g$ extends to a closed segment of the unique $\hat v$-trajectory $\hat\g$, we  construct  a tubular cover $\{U_\a \subset \hat X\}_\a$ of $X$ such that there are diffeomorphisms $\phi_\a: D^{d-1} \times D^1 \to U_\a$ that map each segment $a \times D^1$, $a \in D^{d-1}$, to a segment of a $\hat v$-trajectory. Moreover, $\phi_\a(D^{d-1} \times \d D^1) \subset \hat X \setminus X$. Using compactness of $X$, we choose a finite subcover of the cover $\{U_\a\}_\a$.  Abusing notations, we still denote it  $\{U_\a\}_\a$. 

Let $g_\a: U_\a \to \R$ be defined by the formula 
\begin{eqnarray}\label{eq.adjusting_f}
g_\a(\phi_\a(a \times t))  = \int_{\phi_\a([a \times 0),\; (a \times t])} \hat\b \; =:  \int_{\phi_\a(a \times 0)}^{\phi_\a(a \times t)} \hat\b, 
\end{eqnarray}
where $a \in D^{d-1}$. Note that $dg_\a |_{\g} = \hat\b |_{\g}$, where $\g = \{\phi_\a(a \times t)\}_t$.
 Since $\hat\b(\hat v) =1$, we get $dg_\a(\hat v) =1$. The finite cover $\{U_\a\}_\a$ gives rise to a finite cover $\{V_\a\}_\a$ of the compact trajectory space $\mathcal T(v)$. 
 
 A function $\psi: \mathcal T(v) \to \R$ is defined to be smooth, if its pull-back $\Gamma^\ast(\psi): X \to \R$ is smooth. Thanks to compactness of the Hausdorff space $\mathcal T(v)$, we may use a smooth partition of unity $\{\psi_\a: V_\a \to [0,1]\}_\a$ to introduce a smooth function $f =: \sum_\a \psi_\a \cdot g_\a.$ By its construction, $\psi_\a$ has the property $d\psi_\a(v) =0$. Therefore, $$df(v) = \sum_\a (\psi_\a \cdot dg_\a(v) + g_\a \cdot d\psi_\a(v)) = \sum_\a \psi_\a =1.$$
 By a similar formula, any closed interval $[a, b] \subset \g$, we get 
  \begin{eqnarray}\label{eq.var_of_beta}
  \int_{[a,b]}df = \sum_\a \Big(\int_{[a,b]}\psi_\a \cdot dg_\a +  \int_{[a,b]} g_\a \cdot d\psi_\a\Big) = \sum_\a \Big(\psi_\a\int_{[a,b]} dg_\a\Big)\nonumber \\ = \sum_\a \psi_\a \cdot \int_{[a,b]} \b = \int_{[a,b]} \b.
  \end{eqnarray}
\end{proof}

\begin{corollary}\label{cor.Lyap=1} Any traversing vector field $v$ admits a Lyapunov function $g$ such that $dg(v) = 1$.
\end{corollary}

\begin{proof} If $f: X \to \R$ is a Lyapunov function for $v$, then the $1$-form $\b = \frac{df}{df(v)}$ satisfies the hypotheses of Lemma \ref{lem.df(v)=1}. By the lemma, there exists a Lypunov function $g: X \to \R$ such that $dg(v) = 1$.
\end{proof}

For a traversing vector field $v$ on $X$, every $v$-trajectory $\g \subset X$ has a neighborhood $\hat U(\g) \subset \hat X$ and a $t$-parametric family of diffeomorphisms $\{\phi^t: \hat U(\g) \to \hat U(\g)\}_{t \in [-\e,\e]}$ (where $\e >0$ depends on $\g$) such that $\phi^0 = \mathsf{id}$, and $\frac{d}{dt}\phi^t(a)\big |_{t=0} = v(a)$ for all $a \in \hat U(\g) \cap X$.  \smallskip

\begin{definition}\label{def.flow_generated_diff} Let $v$ be a vector field on $X$.
We say that a contact form $\b$ is $v$-{\sf invariant} if  $(\phi^t)^\ast(\b) = \b$ in $\hat U(\g) \cap X \cap (\phi^t)(X)$ for all sufficiently small $t$. 
Equivalently, if the directional derivative $\mathcal L_v\b = 0$. 

Similarly, the contact structure $\xi$ is $v$-{\sf invariant}, if $(\phi^t)^\ast(\b) = \l_t \cdot \b$, where $\l_t: X \to \R_+$ is a smooth positive function (i.e., $(\phi^t)^\ast(\xi) = \xi$ as distributions).
\hfill $\diamondsuit$ 
\end{definition}

\begin{definition}\label{def.R-equiv} 
$\bullet$ We call two contact forms $\b_1$ and $\b_2$ on $X$ {\sf Reeb-equivalent} ($\mathcal R$-{\sf equivalent} for short), if they share the same Reeb vector field $v$. Similarly, we call contact structures $\xi_1$ and $\xi_2$ $\mathcal R$-{\sf equivalent} if they are are generated by $\mathcal R$-equivalent contact forms.
 \smallskip
 
$\bullet$ We call contact forms $\b_1$ and $\b_2$  {\sf Reeb-conformally-equivalent} ($\mathcal{RC}$-{\sf equivalent} for short), if their Reeb vector fields $v_1$ and $v_2$ are proportional: i.e., $v_2 = \l\cdot v_1$, where $\l: X \to \R_+$ is a smooth positive function. Similarly, we call contact structures $\xi_1$ and $\xi_2$ $\mathcal{RC}$-{\sf equivalent} if they are generated by $\mathcal{RC}$-equivalent contact forms.
 \hfill $\diamondsuit$

\end{definition}

 \begin{remark} 
 \emph{
 Evidently, $\mathcal{RC}$-equivalent contact forms $\b_1$ and $\b_2$ \emph{share the dynamics} of their Reeb flows. In particular, if $v_{\b_1}$ is traversing, or boundary generic, or admits a Lyapunov $1$-form/function, then so is/does  $v_{\b_2}$.
 }
 \hfill $\diamondsuit$
 \end{remark}

\begin{lemma}\label{lem.beta_equiv_beta1}    Let $\b$ be a contact form on a 
smooth manifold $X$ and $v_\b$ its Reeb vector field. Consider a $1$-form $\b_1 = \b + d \eta$, where the smooth function $\eta: X \to \R$ is such that  $d\eta(v_\b) > -1$.\footnote{i.e., $\eta$ is not ``negatively very  steep" along the Reeb flow trajectories.} Such functions $\eta$ form an open convex set in $C^\infty(X)$. 

Then the form $\b_1$ is a contact form on $X$.  Moreover,  $\b_1$ is 
 $\mathcal{RC}$-equivalent to $\b$.
 \end{lemma}
 
 \begin{proof} Since $\b_1 = \b + d \eta$, we get $$\b_1\wedge (d\b_1)^n = \b_1 \wedge (d\b)^n = \b \wedge (d\b)^n + d\eta \wedge (d\b)^n.$$ 

We claim that the $(2n+1)$-form $\b_1\wedge (d\b_1)^n$ is positive on $X$, provided $d\eta(v_\b) > -1$. Indeed, the $(2n+1)$-form $d\eta \wedge (d\b)^n$ is proportional to $(2n+1)$-form $\b \wedge (d\b)^n$ with some functional coefficient $h: X \to \R$. Thus $\b_1\wedge (d\b_1)^n > 0$ is equivalent to the inequality $1+ h > 0$ on $X$. 

To detect $h$, it suffices to compare the contractions of the two forms,  $d\eta\wedge (d\b)^n = h (\b \wedge (d\b)^n)$ and $\b \wedge (d\b)^n$, with a non-vanishing Reeb vector field $v_\b$. Recall that $v_\b \rfloor\b =1$ and $v_\b \rfloor d\b \equiv 0$.
Using these properties of $v_\b$, we get $$v_\b \rfloor (\b \wedge (d\b)^n) = (d\b)^n, \text{ while } v_\b \rfloor (d\eta \wedge (d\b)^n) = d\eta(v_\b) (d\b)^n = h\;(d\b)^n.$$ Thus, when $d\eta(v_\b) > -1$, $v_\b \rfloor(\b_1\wedge (d\b_1)^n)$ is positively proportional to $(d\b)^n$ with the functional coefficient $h$. Therefore, $\b_1$ is a contact form, provided $d\eta(v_\b) > -1$.
\smallskip

Put $v_{\b_1} =: (1 + d\eta(v_\b))^{-1} v_\b$. Let us check that $v_{\b_1}$ is the Reeb vector field for the form $\b_1$. Indeed, $\b_1(v_{\b_1}) = \b_1( (1 + d\eta(v_\b))^{-1} v_\b) = (\b +d\eta) \big((1 + d\eta(v_\b))^{-1} v_\b\big) =1$ since $\b(v_\b) = 1$, and $d\b_1(v_{\b_1}) = d\b( (1 + d\eta(v_\b))^{-1} v_\b) = (1 + d\eta(v_\b))^{-1} d\b(v_\b) \equiv 0.$ 
Thus, $\b_1$ and $\b$ are $\mathcal{RC}$-equivalent. As a result, if $v_\b$ admits a Lyapunov function/form, then so does $v_{\b_1}$.\hfill
\end{proof}


 
\begin{lemma}\label{lem.4.4_X} For a contact form $\b$ on $X$, any smooth diffeomorphism $\Phi: X \to X$, that preserves the $\mathcal{RC}$-equivalence class of the contact form $\b$, 
preserves the Morse stratifications $\{\d_jX(v_\b)\}_j, \{\d_j^\pm X(v_\b)\}_j$ (see (\ref{eq.2.1})) of the boundary. 
\end{lemma}

\begin{proof} Recall that $\b$ uniquely defines the Reeb field $v_\b$, and the conformal class of $v_\b$ determines the $v_\b$-generated Morse stratification of the boundary \cite{K1}. 
\smallskip

If $\Phi^\ast(\b) = \b_1$ and $v_{\b_1} = \l \cdot v_{\b}$ for some smooth function $\l: X \to \R_{>0}$ (i.e., $\b_1, \b$ are $\mathcal{RC}$-equivalent), then the two Morse stratifications coincide, since they are defined in terms of applying iteratively the Lie derivatives $\mathcal L_{v_{\b}}$ and $\mathcal L_{v_{\b_1}} =: \mathcal L_{\l \cdot v_{\b}} = \l \cdot \mathcal L_{v_{\b}}$ to the same auxiliary function $h: \hat X \to \R$ that determines the boundary $\d_1 X$ \cite{K1}. By definition, such a function $h$ has the three properties:
\begin{eqnarray}\label{eq.function for dX} 
\bullet\; & 0 \text{ is a regular value of } h, \nonumber \\
\bullet\; & h^{-1}(0) = \d_1 X \nonumber \\ 
\bullet\; & h^{-1}((-\infty, 0]) = X.
\end{eqnarray}
For $\l >0$, by an inductive argument in $j$ (organized in an ``upper-triangular" pattern), the loci $$\d_jX(v_\b) = \big\{x \in \d X\big | \; h(x) = 0,\, (\mathcal L_{v_\b} h)(x) = 0,\, \ldots ,\, (\mathcal L^{j-1}_{v_\b} h)(x) = 0\big\}$$
$$\d_jX(\l \cdot v_\b) = \big\{x \in \d X\big | \; h(x) = 0,\, (\mathcal L_{\l \cdot v_\b} h)(x) = 0,\, \ldots ,\, (\mathcal L^{j-1}_{\l \cdot v_\b} h)(x) = 0\big\}$$
coincide.
\end{proof}
 \smallskip
 
 \begin{definition}\label{def.property_A} 
We say that a boundary generic and traversing vector field $v$ possesses  {\sf  Property $\mathsf A$}, if each $v$-trajectory $\g$ either is transversal to $\d_1 X$ at \emph{some} point of the set $\g \cap \d_1 X$, or $\g \cap \d_1 X$ is a singleton $x$ and $\g$ is quadratically tangent to $\d_1 X$ at $x$. \hfill $\diamondsuit$
\end{definition}
 \smallskip

We are ready to formulate the main result of this section. It reflects the scheme in (\ref{3_Extension Holo_A}).

\begin{theorem}{\bf(Conjugate Holography for Contact Forms and their Reeb Fields)}\label{th.main} \smallskip

Let $X_1, X_2$ be compact connected oriented smooth $(2n+1)$-dimensional manifolds with boundaries, equipped with contact forms $\b_1, \b_2$. We assume that their Reeb vector fields $v_1 =: v_{\b_1},\, v_2=: v_{\b_2}$ are traversing and boundary generic. We denote by $f_1^\bullet, f_2^\bullet$ the Lyapunov functions such that $df_1^\bullet(v_1) = 1 = df_2^\bullet(v_2)$ (their existence is guarantied by Lemma \ref{lem.df(v)=1}).
In addition, we assume that  $v_2$ possesses Property $\mathsf A$ from Definition \ref{def.property_A}.\footnote{Note that the property $\d_3X_2(v_2) = \emptyset$ implies Property $\mathsf A$ \cite{K1}, \cite{K4}.} 
\smallskip

Assume that a smooth orientation-preserving diffeomorphism $\Phi^\d: \d_1X_1 \to \d_1X_2$ commutes with the two causality maps:
$$C_{v_2} \circ \Phi^\d = \Phi^\d \circ C_{v_1}.$$ 

$\bullet$ Then $\Phi^\d$ extends to a smooth orientation-preserving diffeomorphism $\Phi: X_1 \to X_2$ such that $\Phi$ maps the oriented foliation $\mathcal F(v_1)$ to the oriented foliation $\mathcal F(v_2)$, and the pull-back contact form $\b^\dagger_1 = \Phi^\ast(\b_2)$  is $\mathcal{RC}$-equivalent (in the sense of Definition \ref{def.R-equiv}) to  $\b_1$.\smallskip

$\bullet$  If, in addition, $(\Phi^\d)^\ast(f_2^\bullet)^\d = (f_1^\bullet)^\d$  
and $(\Phi^\d)^\ast(\b_2 |_{\d_1X_2}) = \b_1 |_{\d_1X_1},$
 then  there exists a smooth diffeomorphism $\Phi: X_1 \to X_2$ which extends $\Phi^\d$ and such that $(\Phi)^\ast(\b_2) = \b_1$. 
\end{theorem}

\begin{proof} Our arguments rely heavily on the Holography Theorem from \cite{K4}. Assuming Property {\sf A}, it claims the existence of an orientation-preserving diffeomorphism $\Phi: X_1 \to X_2$ that extends $\Phi^\d$ and maps the $v_1$-oriented foliation $\mathcal F(v_1)$ to the $v_2$-oriented foliation $\mathcal F(v_2)$. In the absence of  Property {\sf A} we only may claim that the extension $\Phi$ is a {\it homeomorphism} whose restriction to the trajectories is a diffeomorphism.

Let us outline the spirit of the Holography Theorem proof. The reader interested in the technicalities may consult with \cite{K4}. 

First, using that $v_2$ is traversing, we construct a Lyapunov function $f_2: X_2 \to \R$ for it.  
Then we pull-back, via the diffeomorphism $\Phi^\d$, the restriction $f_2^\d =: f_2|_{\d_1 X_2}$. Since  $\Phi^\d$ commutes with the two causality maps, the pull back $f_1^\d =: (\Phi^\d)^\ast(f^\d_2)$ has the property $f_1^\d(y) > f_1^\d(x)$ for any pair $y \succ x$ on the same $v_1$-trajectory, the order of points being defined by the $v_1$-flow. Equivalently, we get $f_1^\d(C_{v_1}(x)) > f_1^\d(x)$ for any $x \in \d_1^+X(v_1)$ such that $C_{v_1}(x) \neq x$. As the key step, we prove in \cite{K4} that such $f_1^\d$ extends to a smooth function $f_1: X_1 \to \R$ that has the property $df_1(v_1) > 0$. Hence, $f_1$ is a Lyapunov function for $v_1$. 

Let $i= 1, 2$. Recall that each causality map $C_{v_i}$  allows to view the $v_i$-trajectory space $\mathcal T(v_i)$ as the quotient space $(\d_1X_i)\big/ \{C_{v_i}(x) \sim x\}$, where $x \in \d_1^+X_i(v_i)$; the topology in $\mathcal T(v_i)$ is defined as the quotient topology. Using that $\Phi^\d$ commutes with the causality maps $C_{v_1}$ and $C_{v_2}$, we conclude that $\Phi^\d$ induces a homeomorphism $\Phi^\mathcal T: \mathcal T(v_1) \to \mathcal T(v_2)$ of the trajectory spaces, which preserves their natural stratifications.

For a traversing vector field $v_i$, the manifold $X_i$ carries two mutually transversal  foliations: the oriented $1$-dimensional $\mathcal F(v_i)$, generated by the $v_i$-flow (or rather by $\hat v_i$ to insure that the leaves of $\mathcal F(\hat v_i)$ are non-singular in the ambient $\hat X$), and the foliation $\mathcal G(f_i)$, generated by the constant level hypersurfaces of the Lyapunov function $f_i$ (or rather by constant level hypersurfaces of $\hat f_i: \hat X_i \to \R$ to insure that the leaves are non-singular in the ambient $\hat X$). Note that the leaves of $\mathcal G(f_i)$ may be disconnected, while the leaves of $\mathcal F(v_i)$, the $v_i$-trajectories, are connected.  The two foliations, $\mathcal F(v_i)$ and $\mathcal G(f_i)$, may serve as a ``coordinate grid" on $X_i$: every point $x \in X_i$ belongs to a \emph{unique} pair of leaves $\g_x \subset \mathcal F(v_i)$ and $L_x =: f_i^{-1}(f_i(x)) \subset \mathcal G(f_i)$. 

Conversely, using the traversing nature of the vector field $v_i$, any pair $(y,\,  t)$, where $y \in \g_x \cap \d_1X_i$ and $t \in [f_i^\d(\g_x \cap \d_1X_i)] \subset \R$ determines a \emph{unique} point $x \in X_i$. Here $[f_i^\d(\g_x \cap \d_1X_i)]$ denotes the minimal closed interval that contains the finite set $f_i^\d(\g_x \cap \d_1X_i)$. Note that some pairs of leaves $L$ and $\g$ may have an empty intersection, and some components of leaves $L$ may have an empty intersection with the boundary $\d_1X_i$. However, if $f_i(\d_1X_i) = f_i(X_i)$ as sets, then $f_i^{-1}(c) \cap \d_1X_i \neq \emptyset$ for any $c \in f_i(X_i)$. In particular, if $\d_1X_i$ is connected, then $f_i(\d_1X_i) = f_i(X_i)$ and $f_i^{-1}(c) \cap \d_1X_i \neq \emptyset$ for any $c \in f_i(X_i)$

In fact, using that $f_i$ is a Lyapunov function, the hyprsurface $L = f_i^{-1}(c)$ intersects with a trajectory $\g$ if and only if $c \in [f_i^\d(\g \cap \d_1X_i)]$. 

Consider a pair of points $y, z \in \d_1X_i$. Since the two smooth leaves, $\hat\g_y$ and $\hat f_i^{-1}(f_i(z))$, depend smoothly on the points $y, z$ and are transversal in $\hat X_i$, their intersection point in $X_i$ depends smoothly on $(y, z) \in (\d_1X_i) \times (\d_1X_i)$, as long as $f_i^\d(z) \in [f_i^\d(\g_y \cap \d_1X_i)]$. 
Note that pairs $(y, z)$ with the property $f_i^\d(z) \in f_i^\d(\g_y \cap \d_1X_i)$ give rise to the intersection points that belong to $\d_1X_i$.


Now we are ready to extend the diffeomorphism $\Phi^\d$ to a homeomorphism $\Phi: X_1 \to X_2$. 
In the process, following (\ref{3_Extension Holo_A}), \emph{we assume  implicitly the existence of the foliations $\mathcal F(v_i)$ and $\mathcal G(f_i)$ and of the Lyapunov functions $f_i$ on $X_i$, $i= 1, 2$}. 

Take any $x \in X_1$. It belongs to a unique pair of leaves $L_x  \in \mathcal G(f_1)$ and $\g_x \in \mathcal F(v_1)$. We define $\Phi(x) = x' \in X_2$, where $x'$ is the unique point that belongs to the intersection of $f_2^{-1}(f_1(x)) \in \mathcal G(f_2)$ and the $v_2$-trajectory $\g' = \Gamma_2^{-1}(\Phi^\mathcal T(\g_x))$. 
By its construction, $\Phi |_{\d_1X_1} = \Phi^\d$. Therefore, $\Phi$ induces the same homeomorphism $\Phi^{\mathcal T}: \mathcal T(v_1) \to \mathcal T(v_2)$ as $\Phi^\d$ does. 

The leaf-hypersurface $\hat f_2^{-1}(f_1(x))$ depends smoothly on $x$, but the leaf-trajectory $\hat \g' = \Gamma_2^{-1}(\Phi^\mathcal T(\hat \g_x))$ may not!
 Although the homeomorphism $\Phi$ is a diffeomorphism along the $v_1$-trajectories, it is not clear that it is a diffeomorphism on $X_1$ (a priori, $\Phi$ is just a H\"{o}lder map with a H\"{o}lder exponent $\a = 1/m$, where $m$ is the maximal tangency order of $\g$'s to $\d_1X$). Presently, for proving that $\Phi$ is a diffeomorphism, we need  Property {\sf A} \cite{K4}. Assuming its validity, we use the transversality of $\g_x$ \emph{somewhere} to $\d_1X$ to claim the smooth dependence of the trajectory $\Gamma_2^{-1}(\Phi^\mathcal T(\hat \g_x))$ on $x$. Since the smooth foliations $\mathcal F(\hat v_i)$ and $\mathcal G(\hat f_i)$ are transversal, it follows that $x' = \Phi(x)$ depends smoothly on $x$. 
 Conjecturally, Property {\sf A} is unnecessary for establishing that $\Phi$ is a diffeomorphism.

It remains to sort out what happens to the contact forms under this diffeomorphism $\Phi$ or its modifications. 

Since $\mathcal L_{v_i}\b_i = 0$, the contact form $\b_i$ is $v_i$-invariant, and thus is determined by its boundary data (section) $\b_i ||_{\d X_i}: \d X_i \to T^\ast(X_i)|_{ \d X_i}$, 
 \emph{provided that the $v_i$-flow is known} (as in the scheme (\ref{3_Extension Holo_A})).   

 We observe that $f_1 = \Phi^\ast(f_2)$ by the very construction of $\Phi$. Let $v_1^\dagger =: \Phi_\ast^{-1}(v_2)$. Since $\Phi$ maps $\mathcal F(v_1)$ to $\mathcal F(v_2)$, we conclude that $v_1^\dagger$ must be positively-proportional to $v_1$. For the vector field $v_1^\dagger$,  we get $df_1(v_1^\dagger) = \Phi^\ast(df_2)(v_1^\dagger) =  df_2(v_2) \;  =1$.  At the same time, $\b_1(v_1) =1$.

 Consider the $1$-form $\b_1^\dagger =: \Phi^\ast(\b_2)$ and denote by $\xi_1^\dagger$ its kernel. 
Using that $\Phi$ preserves the orientations, 
we see that $\Phi^\ast(\b_2) \wedge \Phi^\ast(d\b_2)^n = \Phi^\ast(\b_2 \wedge (d\b_2)^n) > 0$ in $X_1$. Therefore, the hyperplane distribution $\xi_1^\dagger$ is a cooriented contact structure with the contact form $\b_1^\dagger $. 

In fact, $v_1^\dagger = \Phi^{-1}_\ast(v_2)$ is the Reeb vector field for $\b_1^\dagger$. Indeed, by naturality, $\Phi^\ast(\b_2)(v_1^\dagger) = \b_2(v_2) =1$ and $\Phi^\ast(d\b_2)(v_1^\dagger \wedge \sim) = d\b_2(v_2 \wedge \sim) =0$. Thus, $\b_1^\dagger$ is $v_1^\dagger$-invariant. Since $v_1^\dagger = \l \cdot v_1$, where $\l: X_1 \to \R_+$ is a positive smooth function,  
$\b_1$ and $\b_1^\dagger$ are $\mathcal{RC}$-equivalent contact forms, and  $\xi_1$ and $\xi_1^\dagger$ are $\mathcal{RC}$-equivalent contact structures on $X_1$ (see Definition \ref{def.R-equiv}). Equivalently, $\xi_2^\dagger =: \Phi_\ast(\xi_1)$ is $\mathcal{RC}$-equivalent to $\xi_2$. This proves the first assertion of the theorem.
 
 Now we turn to the second assertion. The special Lypunov functions $f_1^\bullet, f_2^\bullet$ are essential for constructing the desired extension $\Phi$ of $\Phi^\d$.

We can build the extension $\Phi$ of $\Phi^\d$ as before, using the map $\Phi^{\mathcal T}: \mathcal T(v_1) \to \mathcal T(v_2)$ (induced by $\Phi^\d$) and the restrictions $(f_1^\bullet)^\d = (\Phi^\d)^\ast((f_2^\bullet)^\d)$ and $(f_2^\bullet)^\d$ of the Lyapunov functions $f_1^\bullet$ and $f_2^\bullet$ to the corresponding boundaries. By the very construction of $\Phi$, we get $\Phi^\ast(f_2^\bullet) = f_1^\bullet$. Since  $df_i^\bullet(v_i) = 1$ and $\Phi$ maps $v_1$-trajectories to $v_2$-trajectories, it follows that $\Phi_\ast(v_1) = v_2$.
By the hypotheses, $(\Phi^\d)^\ast(\b_2|_{\d_1X_2}) = \b_1|_{\d_1X_1}$ and, since $\b_1(v_1) = 1= \b_2(v_2)$, we conclude that $(\Phi^\d)^\ast(\b_2||_{\d_1X_2}) = \b_1||_{\d_1X_1}$, except, perhaps, along the locus $\d_2X_2$ where $v_2$ is tangent to $\d_1X_2$. However, since $\Phi^\d$ is a smooth diffeomorphism, by continuity, we get  that $(\Phi^\d)^\ast(\b_2||_{\d_1X_2}) = \b_1||_{\d_1X_1}$ without exception.  Thus, using that $\b_i$ is $v_i$-invariant,  we get that $(\Phi)^\ast(\b_2) = \b_1$.
\end{proof}
Note that the construction of the extension $\Phi$ (from the first assertion of Theorem \ref{th.main}) is implicit and far from being canonical. For example, it depends on the choice of extension of $f_1^\d =: (\Phi^\d)^\ast(f_2^\d)$ to a smooth function $f_1: X_1 \to \R$, which is strictly monotone along the $v_1$-trajectories. 

However, as the Corollary \ref{cor.main_reconstruction}  testifies, the uniqueness (rigidity) of the extension $\Phi$ may be achieved, if we assume the \emph{full a priori knowledge} of the manifolds $X_i$, equipped with the foliation grids $\mathcal F(v_i), \mathcal G(f_i)$ and the Lyapunov functions $f_i$ or $f_i^\bullet$.  
Assuming Property $\mathsf A$, this corollary, an instant implication of Theorem \ref{th.main}, claims that  the \emph{smooth topological type} of the triple $\{X,\, \mathcal F(v_\b),\, \mathcal G(f^\bullet)\}$ and of the contact form $\b$ may be reconstructed from the appropriate boundary-confined data. \smallskip
\begin{corollary}{\bf(Reeb Flow Holography)}\label{cor.main_reconstruction} 

Let $X$ be a compact connected smooth $(2n+1)$-dimensional manifold with boundary, equipped with a contact form $\b$. We assume that the Reeb vector field  $v = v_\b$ is traversing,  boundary generic, and possesses  Property $\mathsf A$.  By Lemma \ref{lem.df(v)=1}, using $\b$, we construct a Lyapunov function $f^\bullet: X \to \R$ such that $df^\bullet(v) = 1$.
\smallskip

Then the following boundary-confined data: 
\begin{itemize}
\item the causality map $C_v: \d_1^+X(v) \to \d_1^-X(v)$,
\item  the restriction $(f^\bullet)^\d: \d_1 X \to \R$ of the Lyapunov function $f^\bullet: X \to \R$, 
\item the restriction $\b^\d: \d_1X \to T^\ast X|_{\d_1X}$ of the contact form $\b$, 
\end{itemize}
are sufficient for a reconstruction of the triple $(X, f^\bullet, \b)$ (and thus, of the Reeb vector field $v_\b$), 
up to  diffeomorphisms $\Phi: X \to X$ which are the identity map on $\d_1 X$. 
\end{corollary}

\begin{proof} 
Assume that we have two sets of structures, $\{X,\, \mathcal F(v_{\b_1}),\, \mathcal G(f^\bullet_1),\b_1\}$ and $\{X,\, \mathcal F(v_{\b_2}),\,\hfill \break \mathcal G(f^\bullet_2),\b_2\}$, which agree along $\d_1X$. 
By Theorem \ref{th.main}, being applied to $\Phi^\d = \mathsf{id_{\d_1X}}$,  the following data 
\begin{eqnarray}\label{eq.boundary_data} 
\{\d_1X,\; \mathcal F(v_{\b_1}) \cap \d_1X,\; \; (f^\bullet_1)^\d,\; \b_1^\d\} = \{\d_1X,\; \mathcal F(v_{\b_2}) \cap \d_1X,\; \; (f^\bullet_2)^\d,\; \b^\d_2\}
\end{eqnarray}
along the boundary $\d_1X$ allow an extension of $\Phi^\d$ to a diffeomorphism $\Phi: X \to X$ which maps $\{X,\, \mathcal F(v_{\b_1}),\, \mathcal G(f^\bullet_1),\b_1\}$ to $\{X,\, \mathcal F(v_{\b_2}),\, \mathcal G(f^\bullet_2),\b_2\}$.  
Therefore, we have reconstructed the triple $(X,\, f^\bullet,\, \b)$, up to a diffeomorphism $\Phi$, which is the identity on  $\d_1X$. Note that the reconstruction of the contact form $\b$ implies the reconstruction of the Reeb field $v_\b$ and of the contact structure $\xi_\b$. 
%
\hfill 
\end{proof}

\subsection{Boundary data of scattering maps $C_{v_\b}$ and the group-theoretical sections of $\b$-contactomorphisms}
\hfill\break
Although the construction of an extension $\Phi: X \to X$ of a given diffeomorphism $\Phi^\d: \d_1X \to \d_1X$, subject to the hypotheses of Corollary \ref{cor.main_reconstruction}, is not canonical, we will argue that  fixing one such extension $\Phi_\star$ makes it possible to extend any other diffeomorphism  $\tilde\Phi^\d: \d_1 X \to \d_1 X$, subject to the hypotheses of Corollary \ref{cor.main_reconstruction}, in a \emph{rigid way}, so that our reconstruction works well for \emph{families} of such diffeomorphisms $\tilde\Phi^\d$. 
\smallskip

We assume now that the quadruple $(X, \mathcal F(v_\b),  f^\bullet, \b)$ is \emph{known and fixed}. We follow the scheme  in (\ref{3_Extension Holo_A}). 
Then, given \emph{any}  diffeomorphism $\Phi^\d: \d_1X \to \d_1X$ that respects the boundary data (\ref{eq.boundary_data}), we can extend it \emph{canonically} to a diffeomorphism $\Phi: X \to X$ that keeps the triple $\{\mathcal F(v), f^\bullet, \b\}$ invariant. Indeed, any point $x \in X$ is determined by the unique pair $\big(\g_x \cap \d_1X,\; f^\bullet(x)\big)$, subject to the requirement $f^\bullet(x) \in [f^\bullet(\g_x \cap \d_1X)]$. 
Now put $\Phi(x) = x'$, where the point $x' \in X$ is uniquely defined by the pair 
$$\Big(\Phi^\d(\g_x \cap \d_1X),\; \Phi^\d\big((f^\bullet)^{-1}(f^\bullet(x)\big)\; \cap\, \d_1X) \Big).$$
By the properties of $\Phi^\d$, the latter pair is of the form 
$$\big(\g_{x'} \cap \d_1X,\; \big((f^\bullet)^{-1}(f^\bullet(x')\big) \cap \d_1X) \big),$$
where $f^\bullet(x') \in f^\bullet([\g_{x'} \cap \d_1X])$. 

 As in Theorem \ref{th.main}, assuming  Property $\mathsf A$, we get that $\Phi$ is a smooth diffeomorphism.
\smallskip

Since $\Phi^\d$ preserves the boundary data $\{(f^\bullet)^\d, \b^\d\}$ and the contact form $\b$ is $v_\b$-invariant, we get that $\Phi$ must preserve both $\b$ and its Reeb vector field $v_\b$. 
Thus, the extension $\Phi$ is a contactomorphism.
\smallskip

Now, consider the topological groups $\mathsf{Diff}^\infty_+(X)$ and $\mathsf{Diff}^\infty_+(\d_1 X)$ of smooth orientation-preserving diffeomorphisms and the natural restriction homomorphism $\mathcal R: \mathsf{Diff}^\infty_+(X) \to \mathsf{Diff}^\infty_+(\d_1 X).$ Similarly, we may consider the groups  $^\circ\mathsf{Diff}^\infty_+(X)$ and $^\circ\mathsf{Diff}^\infty_+(\d_1 X)$  that are the connected components of  $\mathsf{Diff}^\infty_+(X)$ and $\mathsf{Diff}^\infty_+(\d_1 X)$ that contain the identity.

Recall that the famous Cerf Theorem \cite{Cerf} claims that any orientation-preserving diffeomorphism of $S^3$ is isotopic to the identity (i.e., $\mathsf{Diff}^\infty_+(S^3) \approx\; ^\circ\mathsf{Diff}^\infty_+(S^3)$), and thus, extends to a diffeomorphism of the $4$-ball $D^4$. Therefore, the homomorphism $\mathcal R: \mathsf{Diff}^\infty_+(D^4) \to \mathsf{Diff}^\infty_+(S^3)$ is onto. 

However, remarkably, for any smooth compact connected manifold $X$ with boundary, the restriction map $\mathcal R:\; ^\circ\mathsf{Diff}^\infty_+(X) \to \;^\circ\mathsf{Diff}^\infty_+(\d_1 X)$ does not admit a group-theoretical 
section (\cite{ChMa}, Corollary 1.8)!\smallskip

Let $\mathsf{Diff}^\infty_+(X; \b, f^\bullet)$ denote the subgroup of the group $\mathsf{Diff}^\infty_+(X)$ of smooth orientation-preserving diffeomorphisms,  whose elements preserve the contact form $\b$ and 
the Lyapunov function $f^\bullet: X \to \R$ such that $df^\bullet(v_\b) = 1$.

Let $\mathsf{Diff}^\infty_+(\d_1 X;\, C_{v_\b}, \b^\d,  (f^\bullet)^\d)$ denote the subgroup of $\mathsf{Diff}^\infty_+(\d_1 X)$ whose elements commute with the causality map $C_{v_\b}: \d_1^+X(v_\b) \to \d_1^-X(v_\b)$, preserve the restrictions to the boundary $(f^\bullet)^\d: \d_1 X \to \R$ of the Lyapunov function $f^\bullet$ and the restriction $\b^\d: \d_1X \to T^\ast(X) |_{\d_1X}$ of $\b: X \to T^\ast(X)$. 
\smallskip

In contrast with the fact that $\mathcal R: \mathsf{Diff}^\infty_+(X) \to \mathsf{Diff}^\infty_+(\d_1 X)$ does not admit a group-theoretical section \cite{ChMa}, we get the following claim.

\begin{theorem}\label{th.canonical_ext}  Under the hypotheses of Corollary \ref{cor.main_reconstruction}, 
the restriction homomorphism \hfill \break  $\mathcal R: \mathsf{Diff}^\infty_+(X) \to \mathsf{Diff}^\infty_+(\d_1 X)$ admits a group-theoretical continuous section $\s$ over the subgroup $\mathsf{Diff}^\infty_+(\d_1 X;\, C_{v_\b}, \b^\d,  (f^\bullet)^\d) \subset  \mathsf{Diff}^\infty_+(\d_1 X)$. 

Moreover, the image $\s(\mathsf{Diff}^\infty_+(\d_1 X;\, C_{v_\b}, \b^\d,  (f^\bullet)^\d))$ belongs to the subgroup $\mathsf{Diff}^\infty_+(X; \b, f^\bullet)$ of $\mathsf{Diff}^\infty_+(X)$.
\end{theorem}

\begin{proof} Follows instantly from the previous arguments and the proof of Theorem \ref{th.main}. 
\end{proof}

\section{Non-squeezing of $2n$-symplectic volumes, Morse wrinkles of Reeb fields, and holography}

Let us introduce now some new numerical invariants of contact forms on manifolds with boundary. 

\begin{definition}\label{def.Reeb_diam}
Let $v_\b$ be a traversing Reeb vector field for a contact $1$-form $\b$ on $X$. 
 
The {\sf Reeb diameter} of $(X, \b)$ is defined by $$\mathsf{diam}_\mathcal R(\b) \;=:\;\sup_{\g \subset X} \Big\{\int_\g \b \Big\},$$ where $\g$ runs over all  $v_\b$-trajectories.  \hfill $\diamondsuit$
\end{definition}


In what follows, we consider the immersions $\Psi: X \to Y$ of compact manifolds $X, Y$ of equal dimensions (in such case, the notions of an immersion and of a submersion coincide). The restrictions of an immersing map $\Psi$ to the boundary $\d X$ is an immersion in the interior of $Y$. One may think of $\Psi$ as the restriction to $X$ of an immersion $\hat\Psi: \hat X \to \hat Y$, where $\hat X$ is an open manifold that contains $X$, and $\hat Y$ is an open manifold that contains $Y$.

\begin{lemma} Consider an orientation-preserving immersion $\Psi: X \to Y$ of two compact equidimentional manifolds. Let $\b_Y$ be a contact form on $Y$ whose Reeb vector field $v_Y$ is traversing. We assume that the hypersurface $\Psi(\d_1X)$ is boundary generic relative to $v_Y$. \smallskip

Consider the contact form $\b_X =: \Psi^\ast(\b_Y)$ on $X$. Its Reeb vector field $v_X$ is traversing with the help of the Lyapunov function $f_X = \Psi^\ast(f_Y)$, where $df_Y(v_Y) > 0$. 
Moreover, the Reeb diameters satisfy the inequality
$$\mathsf{diam}_\mathcal R(\b_X) \leq \mathsf{diam}_\mathcal R(\b_Y).$$

If an immersion $\Psi: X \to Y$ is onto (i.e., $\Psi$ is a covering map with the base $Y$), then 
$$\mathsf{diam}_\mathcal R(\b_X) =  \mathsf{diam}_\mathcal R(\b_Y).$$
\end{lemma}

\begin{proof} Since $\Psi$ is submersion and the image $\Psi(\d_1X)$ is boundary generic relative to $v_Y$, 
for any $v_Y$-trajectory $\g$ (which is a closed interval or a singleton) the set $\Psi^{-1}(\g)$ is a disjoint union of finitely many closed intervals $[\tilde\g]$, some of which may be singletons. Moreover, $\Psi: [\tilde\g] \hookrightarrow \g$ is a diffeomorphism. Thus, $\int_{[\tilde\g]} \b_X = \int_{[\tilde\g]} \Psi^\ast (\b_Y) = \int_{\Psi([\tilde\g])} \b_Y \leq \int_\g \b_Y$.  When $\Psi$ is a covering map, we get that $\Psi: \tilde\g \to \g$ is a diffeomorphism. In such a case, $\int_{\tilde\g} \b_X = \int_\g \b_Y$. Since any $v_X$-trajectory $\tilde\g$ is obtained from some $v_Y$-trajectory $\g$ as described above, the claims about the Reeb diameters  follow. 
\end{proof}

\begin{definition}\label{def.vol_of_T(beta)} 
For a traversing Reeb field, we introduce the {\sf volume of the space of Reeb trajectories} $\mathcal T(v_\b)$ by the formula
\begin{eqnarray}\label{eq.vol_of_T(beta)}
vol_{(d\b)^n}(\mathcal T(v_\b)) =:\; \int_{\d_1^+X(v)} (d\b)^n.
\end{eqnarray}
\hfill $\diamondsuit$
\end{definition}

Among the contact forms $\b$ that represent a given contact structure $\xi$, we may consider only $\b$'s such that their Reeb vector fields $v_\b$ are traversing and $\mathsf{diam}_\mathcal R(\b) =1$. Of course, if $\mathsf{diam}_\mathcal R(\b) = \l >0$ for some contact form $\b$ and its traversing Reeb vector field $v_\b$, then $\mathsf{diam}_\mathcal R(\l^{-1}\b) = 1$ for the traversing Reeb vector field $\l v$ of the contact form $\l^{-1}\b$. Thus, without loss of generality, in  Theorem \ref{th.isoperimetric}, we may assume that $\mathsf{diam}_\mathcal R(\b) =1$.

\begin{theorem}\label{th.isoperimetric} Let $X$ be a compact connected smooth $(2n+1)$-dimensional manifold with boundary, equipped with a contact $1$-form $\b$. Let $v_\b$ be a traversing boundary generic Reeb vector field for $\b$. 

Then we get an ``isoperimetric" inequality
\begin{eqnarray}\label{eq.isoperimetric}
\int_X \b \wedge (d\b)^n \leq \mathsf{diam}_\mathcal R(\b) \cdot \int_{\d_1^+X(v_\b)} (d\b)^n  =:\; \mathsf{diam}_\mathcal R(\b) \cdot vol_{(d\b)^n}(\mathcal T(v_\b)). 
\end{eqnarray}
\indent This enequality can be also written in the ``equatorial" form:
\begin{eqnarray}\label{eq.isoperimetric_A}
\int_X \b \wedge (d\b)^n \leq \mathsf{diam}_\mathcal R(\b) \cdot \int_{\d_2X(v_\b)} \b \wedge (d\b)^{n-1}.
\end{eqnarray}
\end{theorem}

\begin{proof} The argument is a variation on the theme of Santalo formula \cite{S}, as interpreted in \cite{K5}. For any $y \in \d_1^+X(v_\b)$ we consider the $v$-trajectory $\g_y$ through $y$. The $(-v)$-directed projection $p: X \to \d_1^+X(v_\b)$ is a submersion, away from the set $p^{-1}(\d_2X(v_\b))$ of zero measure in $\d_1^+X(v_\b)$. Thus, we can use the Dieudonne generalization of the Fubini formula \cite{D}:
$$\int_X \b \wedge (d\b)^n = \int_{y \in \d_1^+X(v_\b)} \Big(\int_{\g_y} \b \Big) (d\b)^n \leq \mathsf{diam}_\mathcal R(\b) \cdot \int_{\d_1^+X(v_\b)} (d\b)^n,$$
together with Definition \ref{def.Reeb_diam}, to validate (\ref{eq.isoperimetric}).

 By Stokes' theorem, applied to $\d_1^+X(v_\b)$, (\ref{eq.isoperimetric_A}) follows. Note that formula (\ref{eq.isoperimetric_A}) testifies that $\int_{\d_2X(v_\b)} \b \wedge (d\b)^{n-1} > 0$, which supports weakly Conjecture \ref{conj.positive_on_Morse_strata}.

Although these arguments are simple, it is somewhat surprising to realize the $(d\b)^n$-volume of the trajectory space $\mathcal T(v_\b)$ depends only on the form $\b \wedge (\d\b)^{n-1}|_{\d_2X(v_\b)}$ on the ``equator" $\d_2X(v_\b)$. 
\hfill
\end{proof}




\begin{definition}\label{def.average_Reeb_length}
Assume that the Reeb vector field $v_\b$ of a contact form $\b$ is traversing. We introduce the {\sf average length of the Reeb trajectories} by the Fubini type formula
\begin{eqnarray}\label{eq.average_Reeb}
\mathsf{av}_\mathcal R(\b) \;=:\; \frac{\int_X \b \wedge (d\b)^n} {\int_{\d_1^+X(v_\b)} (d\b)^n}.
\end{eqnarray}

By (\ref{eq.isoperimetric_A}), $\mathsf{av}_\mathcal R(\b) \; \leq  \;\mathsf{diam}_\mathcal R(\b).$ \hfill $\diamondsuit$
\end{definition}

\noindent {\bf Example 7.1.} 
Let $\b = dz + x\, dy$ and let $X = \big\{\frac{x^2}{(A/2)^2} + \frac{y^2}{(B/2)^2} +\frac{z^2}{(C/2)^2} \leq 1\big\}$. Then the inequality (\ref{eq.isoperimetric}) can be written as
$$\frac{\pi}{6} ABC \leq C \cdot \big(\frac{\pi}{4} AB\big),$$
or as $\frac{\pi}{6}  \leq \frac{\pi}{4}$.
\hfill $\diamondsuit$
\smallskip

\begin{corollary}\label{cor.beta_and_beta1_share}
For any function $\eta \in C^\infty(\mathcal T(v_\b))$ consider the differential $d\eta$ of its pullback $\Gamma^\ast(\eta)$ under the map $\Gamma: X \to \mathcal T(v_\b)$. Assume that $d\eta(v_\b) > -1$.

Then  the contact form $\b_1 = \b + d\eta$ shares with $\b$ the Reeb diameter $\mathsf{diam}_\mathcal R(\b)$, the volume $\int_X \b \wedge (d\b)^n$, and the volume $vol_{(d\b)^n}(\mathcal T(v_\b))$ of the trajectory space.
\hfill $\diamondsuit$
\end{corollary}

\begin{proof} The ``exact slantings" $\b_1 = \b + d\eta$ of a contact form $\b$ do not change the $\b$-lengths of Reeb trajectories, provided the function $\eta \in C^\infty(\mathcal T(v_\b))$. 

By Lemma~\ref{lem.beta_equiv_beta1}, the contact forms $\b$ and $\b + d\eta$ share the same Reeb vector field: $v_\b = v_{\b + d\eta}$. Thus, for a trajectory $\g$ bounded by points $a$ and $b$, we get $$\int_\g(\b + d\eta) = \eta(b) - \eta(a) + \int_\g \b = \int_\g \b$$ since $\eta(b) = \eta(a)$. As a result, $\mathsf{diam}_\mathcal R(\b + d\eta) = \mathsf{diam}_\mathcal R(\b)$. 

 Using the Fubini formula, we verify that the $(\b+ d\eta)\wedge (d(\b+d\eta))^n$-induced and the $(\b)\wedge (d\b)^n$-induced volumes of $X$ are equal:
$$\int_X (\b+ d\eta)\wedge (d(\b+d\eta))^n = \int_{y \in \d_1^+X(v_\b)} \Big(\int_{\g_y} (\b +d\eta)\Big) (d\b)^n$$ 
$$ = \int_{y \in \d_1^+X(v_\b)} \Big(\int_{\g_y} \b \Big) (d\b)^n + \int_{y \in \d_1^+X(v_\b)} \Big(\int_{\g_y} d\eta \Big) (d\b)^n.$$
Thanks to the property $\mathcal L_{v_\b}\eta = 0$,  $\int_{\g_y} d\eta = 0$ for all $y$, and the latter formula reduces to
$$ \int_{y \in \d_1^+X(v_\b)} \Big(\int_{\g_y} \b \Big) (d\b)^n = \int_X (\b)\wedge (d\b)^n.$$
 
Also, using Lemma \ref{lem.4.4_X}, we get
$$vol_{(d\b)^n}(\mathcal T(v_\b))=: \int_{\d_1^+X(v_\b)} (d\b)^n =  \int_{\d_1^+X(v_{\b+d\eta})} (d\b)^n$$
$$ =  \int_{\d_1^+X(v_{\b+d\eta})} (d(\b+d\eta))^n =: vol_{(d(\b+ d\eta))^n}(\mathcal T(v_{\b+d\eta}))$$

\hfill
\end{proof}




%


\begin{definition}\label{def.Psi_complexity_trajectories}
Let $Y$ be a smooth manifold and let $v$ a non-vanishing vector field on it that admits a Lyapunov function. Consider a compact smooth connected manifold $X$ with boundary, $\dim X = \dim Y$, and a regular imbedding $\Psi: X \hookrightarrow \mathsf{int}(Y)$ such that $v$ is boundary generic relative the hypersurface $\Psi(\d_1X)$. We introduce a natural number $c^\bullet(\Psi, v)$ as one half of the maximal number of times a $v$-trajectory intersects $\Psi(\d_1X)$. \hfill $\diamondsuit$
\end{definition}

In general, $c^\bullet(\Psi, v)$ seems to resist an intrinsic characterization in terms of the pull-back $\Psi^\dagger(v)$ of the vector field $v$ to $X$ and of the combinatorial types of $\Psi^\dagger(v)$-trajectories. We know only one \emph{lower} bound of $c^\bullet(\Psi, v)$ in terms of $\Psi^\dagger(v)$: 
 $$c^\bullet(\Psi^\dagger(v)) =: \max_{\g^\dagger \subset X} \Big\{\frac{1}{2} |\om(\g^\dagger)| \Big\} \leq c^\bullet(\Psi, v),$$
 where $\g^\dagger$ is a $\Psi^\dagger(v)$-trajectory (see (\ref{eq.2.5}) for the definition of $|\om(\g^\dagger)| $).
This estimate follows from the local models of the traversing boundary generic  flows in the vicinity of their trajectories \cite{K1}, \cite{K3}. \smallskip

\begin{remark}
\emph{Even if  the connected source and the target of $\Psi$ both carry \emph{convex} vector fields, the number $c^\bullet(\Psi, v)$ can be arbitrarily big.  Indeed,  just rap a very narrow ($\e \ll 1$) band $Z = I \times [-\e, \e]$ with rounded corners $k$ times around the cylinder $C = S^1 \times [-1, 1]$, where $I$ denotes the interval $[0,1]$ and the vector field $v$ is aligned with the generators of the cylinder.   Next, consider the embedding $$\Psi: X = S^1 \times I \times [-\e, \e]\; \hookrightarrow \;S^1 \times S^1 \times [-1, 1] =Y,$$ obtained by multiplying the wrapping map $\psi : Z \subset C$ with the identity map on the first multiplayer $S^1$. Then $c^\bullet(\Psi, v_Y) = k$.} \hfill $\diamondsuit$
\end{remark}
\smallskip



Consider a traversing $v$-flow on a compact manifold $Y$ and a regular embedding $\Psi: X \to \mathsf{int}(Y)$, where $X$ is a compact manifold of the same dimension as $Y$.  

Assuming that $v$ is boundary generic with respect to the hypersurface $\Psi(\d_1X)$, with the help of $(-v)$-flow, we  generate a {\sf shadow} $\mathsf X^\d(\Psi)$ of $\Psi(X)$ on the boundary of $Y$. It is produced by forming the space $\mathsf X(\Psi)$ of $v$-trajectories through the points of $\Psi(X)$ (equivalently, through the locus $\d_1^+\Psi(X)(v)$) and then taking the intersection  $\mathsf X(\Psi) \cap \d_1^+Y(v)$.  This constructions presumes that $\Psi(X)$ is ``semi-transparent" for the flow (see Fig.~\ref{fig.3_contact_1}, the left diagram). 

\smallskip

Our next theorem resembles  Gromov's non-squeezing theorem \cite{Gr2} in symplectic geometry, being transferred to the contact geometry environment. Non-squeezing results about the contact embeddings of $S^1 \times B_{r_1} \subset S^1 \times \R^{2n}$ into $S^1 \times B_{r_2}  \subset S^1 \times \R^{2n}$, where $B_{r_1}$ and $B_{r_2}$ are the $(2n)$-balls of radii $r_1$ and $r_2$ and $S^1 \times \R^{2n}$ carries the standard contact form, may be found in the paper  \cite{EKP} by Eliashberg, Kim, and  Polterovich, and in its extensions \cite{Chi} by Chiu  and \cite{FSZ} by Frazer, Sandon, and Zhang. In particular, they proved that  contact squeezing is impossible for $1 \leq \pi r_2^2 \leq \pi r_1^2$ \cite{Chi}, \cite{FSZ}. 

Theorem \ref{th.non-squeezing} deals with more general embeddings/immersions $\Psi: X_1 \to X_2$ ($\dim X_1 = 2n+1 = \dim X_2$)  that respect the contact forms: i.e., $\b_1 = \Psi^\ast(\b_2)$ (see Fig. \ref{fig.3_contact_1}). It employs the positive integer $c^\bullet(\Psi, v_{\b_2})$ (see Definition \ref{def.Psi_complexity_trajectories}) that measures how folded is the image $\Psi (\d_1X_1)$ with respest to the $v_{\b_2}$-flow. The ``equatorial" inequality of  Theorem \ref{th.non-squeezing}  may be also viewed as delivering a \emph{lower} bound of the number $c^\bullet(\Psi, v_{\b_2})$ in terms of the appropriate $\b_1$- and $\b_2$-generated symplectic $2n$-volumes. 

\bigskip

\begin{theorem}{\bf (Non-squeezing equatorial property of the Reeb fields)} \label{th.non-squeezing}  

Consider two compact smooth manifolds $X_1, X_2$, equipped with contact forms $\b_1, \b_2$, respectively. We assume that the Reeb vector fields  $v_{\b_1}, v_{\b_2}$ are boundary generic and $v_{\b_2}$ is traversing.  Let $\Psi: X_1 \hookrightarrow X_2$ be a smooth embedding such that $\b_1 = \Psi^\ast(\b_2)$. Then

$$\int_{X_1} \b_1 \wedge (d\b_1)^n \;  \leq \; \int_{X_2} \b_2 \wedge (d\b_2)^n, $$ 

$$\mathsf{diam}_\mathcal R(\b_1) \; \leq \; \mathsf{diam}_\mathcal R(\b_2).$$

 We also get the inequality between the symplectic $2n$-volumes of the two trajectory spaces (or, equivalently, between the contact-like volumes of the tangency locii, the ``equators"): 
$$0 \leq \int_{\d_2X_1(v_{\b_1})} \b_1 \wedge (d\b_1)^{n-1} = vol_{(d\b_1)^n}(\mathcal T(v_{\b_1})) \; \leq \; $$
$$\; \leq \;  c^\bullet(\Psi, v_{\b_2}) \cdot vol_{(d\b_2)^n}(\mathcal T(v_{\b_2})) = c^\bullet(\Psi, v_{\b_2}) \cdot \int_{\d_2X_2(v_{\b_2})} \b_2 \wedge (d\b_2)^{n-1}.$$

\end{theorem}

\begin{figure}[ht]
\centerline{\includegraphics[height=2in,width=4.6in]{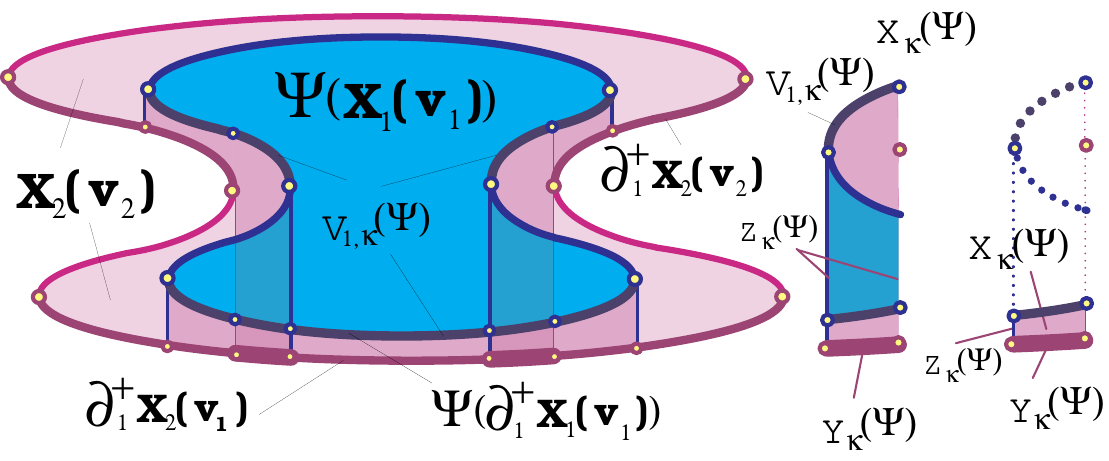}}
\bigskip
\caption{\small{The figure helps with the notations in the proof of Theorem \ref{th.non-squeezing}. In the figure, $v_{\b_1}$ and $v_{\b_2}$ are the constant vertical vector fields 
and $c^\bullet(\Psi, v_{\b_2}) = 2$.}}
\label{fig.3_contact_1}
\end{figure}


\begin{proof} To simplify the notations, put $v_1 =: v_{\b_1}$, $v_2 =: v_{\b_2}$.

 The validation of the first two inequalities is straightforward. 
Using that $\b_1 = \Psi^\ast(\b_2)$, we see that $\Psi_\ast(v_1) = v_2|_{\Psi(X_1)}$. Moreover, the combinatorial tangency pattern (see the discussion preceding formula (\ref{eq.2.4})) of every $v_1$-trajectory $\g_1$ to $\d_1X_1$ coincides with the tangency pattern to $\Psi(\d_1X)$ of the portion $\Psi(\g_1)$ of the unique $v_2$-trajectory $\g_2 \supset \Psi(\g_1)$. Note that the tangency pattern of $\g_2$ to $\Psi(\d_1X_1)$ could be different.\smallskip 

Since $\Psi^\ast(\b_2) = \b_1$, by naturality, we get  $\int_{X_1} \b_1 \wedge (d\b_1)^n =  \int_{\Psi(X_1)} \b_2 \wedge (d\b_2)^n$. Evidently,  $\int_{\Psi(X_1)} \b_2 \wedge (d\b_2)^n \leq \int_{X_2} \b_2 \wedge (d\b_2)^n$. Thus, $\int_{X_1} \b_1 \wedge (d\b_1)^n \leq \int_{X_2} \b_2 \wedge (d\b_2)^n$. \smallskip

 It follows instantly that $\mathsf{diam}_\mathcal R(\b_1) \leq \mathsf{diam}_\mathcal R(\b_2)$, since the $\Psi$-image of any $v_{\b_1}$-trajectory $\g$ is a segment of a $v_{\b_2}$ trajectory $\g'$ that contains $\Psi(\g)$.\smallskip
 
Now we are going to validate the equatorial inequality (see Fig.~\ref{fig.3_contact_1} for all the ingredients of the constructions below). Let $\mathsf W_2(\Psi)$ denote the $2n$-dimensional set comprised of $v_2$-trajectories through the locus $\Psi(\d_2X_1(v_1)) \coprod \d_2X_2(v_2)$. Consider the set (see Fig.~ \ref{fig.3_contact_1}) $$\mathsf V_1^\circ(\Psi) =:\, \Psi\big(\d_1^+X_1(v_1)\big) \setminus \big(\Psi(\d_1^+X_1(v_1)) \cap \mathsf W_2(\Psi)\big)$$
along which the continuous dependence of the $v_2$-flow on the initial conditions in $\Psi(\d_1^+X_1(v_1))$ holds. 
 Let $\mathsf V^\circ_{1,\kappa}(\Psi)$ be a connected component  of the set  $\mathsf V_1^\circ(\Psi)$, and let $\mathsf V_{1, \kappa}(\Psi)$ be the closure of $\mathsf V_{1, \kappa}^\circ(\Psi)$. We denote by $\d^+_1X^\circ_{1, \kappa}(v_1) \subset \d_1^+X_1(v_1)$ the $\Psi$-preimage of $\mathsf V^\circ_{1,\kappa}(\Psi)$.

Keeping in mind that $\Psi_\ast(v_1) = v_2 |_{\Psi(X_1)}$, consider the set $\mathsf X_\kappa^\circ(\Psi)$, formed by the union of portions of the $(-v_2)$-trajectories that originate at the set $\mathsf V_{1, \kappa}^\circ(\Psi)$ and terminate at the points of the set $\d_1^+X_2(v_2)$. 
Let $\mathsf X_\kappa(\Psi)$ be the closure of $\mathsf X_\kappa^\circ(\Psi)$.

Let  $\mathsf Y_\kappa(\Psi) =:\, \mathsf X_\kappa(\Psi) \cap  \d_1^+X_2(v_2)$ be the set where the $(-v_2)$-trajectories through $\mathsf V_{1, \kappa}(\Psi)$ terminate. Then the boundary of the set $\mathsf X_\kappa(\Psi)$ consists of three parts: $\mathsf V_{1, \kappa}(\Psi),\,  \mathsf Y_\kappa(\Psi)$, and the rest, denoted by $\mathsf Z_\kappa(\Psi)$. The $2n$-dimensional Lebesgue measures of mutual intersections of these three parts is zero.  
 
  We notice that $\mathsf Z_\kappa(\Psi)$ consists of portions of $(-v_2)$-trajectories. Since $v_2\, \rfloor\,  d\b_2 =0$, we conclude that $\int_{\mathsf Z_\kappa(\Psi)} (d\b_2)^n  = 0$. 

 
 Now we apply  Stokes' theorem to the zero form $d\big((d\b_2)^n\big)$ on $\mathsf X_\kappa(\Psi)$ to get 
$$0 = \int_{\mathsf X_\kappa(\Psi)} d((d\b_2)^n) = \int_{\mathsf Y_\kappa(\Psi)} (d\b_2)^n - \int_{\mathsf  V_{1,\kappa}(\Psi)} (d\b_2)^n.$$ 
Therefore, using that $\Psi(\d^+_1X^\circ_{1, \kappa}(v_1)) = \mathsf V_{1,\kappa}(\Psi)$ by the definition of the relevant sets, we get
$$ \int_{\mathsf Y_\kappa(\Psi)} (d\b_2)^n = \int_{\mathsf  V_{1,\kappa}(\Psi)} (d\b_2)^n = \int_{\Psi(\d^+_1X^\circ_{1, \kappa}(v_1))} (d\b_2)^n.$$ 
Thus, by  the latter equality,
$$\int_{\d^+_1X^\circ_{1, \kappa}(v_1)} (d\b_1)^n = \int_{\Psi(\d^+_1X^\circ_{1, \kappa}(v_1))} (d\b_2)^n 
= \int_{\mathsf Y_\kappa(\Psi)} (d\b_2)^n \leq  \int_{\d_1^+X_2(v_2))} (d\b_2)^n,$$
since $\mathsf Y_\kappa(\Psi)$ is a part of the locus $\d_1^+X_2(v_2)$ on which, by Lemma \ref{lem.positive_on_d_+}, $(d\b_2)^n \geq 0$. \smallskip

Next, we notice that the domains  $\{\mathsf Y_\kappa(\Psi)\}_\kappa$ may overlap massively in $\d_1^+X_2(v_2))$ and the multiplicity of the overlap is bounded from above by the  maximal number of times a $v_2$-trajectory, passing through the set $\Psi(X_1)$, hits the set $\Psi(\d_1^+X_1(v_1))$. By Definition ~\ref{def.Psi_complexity_trajectories}, this natural number is $c^\bullet(\Psi, v_{\b_2})$.  

Adding the inequalities for all $\kappa$'s and using that   $\{\d_1^+X^\circ_{1, \kappa}(v_1)\}_\kappa$ form a partition of $\d_1^+X_1(v_1)$, modulo the set of $2n$-dimensional measure zero, we get
$$\int_{\d_1^+X_1(v_1)} (d\b_1)^n = \sum_\kappa \int_{\d_1^+X_{1,\kappa}(v_1)} (d\b_1)^n \, \leq \, c^\bullet(\Psi, v_{\b_2}) \cdot\int_{\d_1^+X_2(v_2)} (d\b_2)^n.$$
\smallskip

Finally, in view of Definition \ref{def.vol_of_T(beta)} and by Stokes' Theorem, we get $$\int_{\d_2X_1(v_1)} \b_1 \wedge (d\b_1)^{n-1} = vol_{(d\b_1)^n}(\mathcal T(v_1)),\;\; \;\;\int_{\d_2X_1(v_2)} \b_2 \wedge (d\b_2)^{n-1} = vol_{(d\b_2)^n}(\mathcal T(v_2)).$$
\end{proof}

\begin{example}\label{ex.ellipsoid}
\emph{Let $Y$ be a compact connected smooth $(2n+1)$-manifold with boundary,  equipped with a contact form $\b_Y$ such that its Reeb vector field $v_Y$ is traversing and boundary generic. Let $X$ be the ellipsoid $\big\{\sum_{i=1}^{2n+1} \frac{x_i^2}{a_i^2} \leq 1\big\}$ in the Euclidean space $\mathsf E^{2n+1}$, equipped with the standard contact form}
$$\b = dx_{2n+1} +  \frac{1}{2}\sum_{i=1}^n (x_{i+1}\, dx_i - x_i\, dx_{i+1}).$$
\emph{Then Theorem \ref{th.non-squeezing} claims that, for any embedding $\Psi: X \to Y$ such that $\b = \Psi^\ast(\b_Y)$, the following inequalities are valid:}
\begin{itemize}
\item $$\frac{\pi^{n+\frac{1}{2}}}{\Gamma(n+\frac{3}{2})}\; 
a_1 a_2 \dots a_{2n+1} \; \leq \; \int_Y \b_Y \wedge (d\b_Y)^n,$$
\item $$a_{2n+1} \leq \mathsf{diam}_\mathcal R(\b_Y),$$
\item $$\frac{\pi^{n}}{\Gamma(n + 1)}\; a_1 a_2 \dots a_{2n}\; \leq \; c^\bullet(\Psi, v_{\b_2}) \cdot \int_{\d_1^+Y(v_Y)} (d\b_Y)^n$$ $$ = c^\bullet(\Psi, v_{\b_2}) \cdot \int_{\d_2Y(v_Y)} d\b_Y \wedge (d\b_Y)^{n-1},$$
 \end{itemize}
\emph{where $\Gamma(\sim)$ is the Gamma function.}
\emph{
In particular, for $X$ being a ball of radius $r$, we get the following constraints on the {\sf injectivity radius} $r$ of a contact embedding:
$$\frac{\pi^{n+\frac{1}{2}}}{\Gamma(n+\frac{3}{2})}\; 
r^{2n+1} \leq \int_Y \b_Y \wedge (d\b_Y)^n,$$
$$r \leq \mathsf{diam}_\mathcal R(\b_Y),$$
$$\frac{\pi^{n}}{\Gamma(n + 1)}\; r^{2n}\; \leq \; c^\bullet(\Psi, v_{\b_2})\cdot \int_{\d_2Y(v_Y)} d\b_Y \wedge (d\b_Y)^{n-1}.$$
}
\hfill $\diamondsuit$
\end{example}

\begin{example}\label{ex.sand_clock}
\emph{
Consider a solid $X$ shaped as a sand clock (see Fig.~\ref{fig.3_contact_2}) which is vertically aligned with the $z$-axis in the $xyz$-space $\R^3$. Let $\b = dz + x\, dy$. We consider a $2$-parameter family of  transformations $\{A_{\l\mu}: x \to \l x, y \to \l y, z \to \mu z\}_{\l, \mu \in \R_+}$. We compare two contact $3$-folds $(X, \b)$ and $(A_{\l\mu}(X), \b)$. Then  $c^\bullet(\Psi, v_2) \geq 2$.}

\emph{
For $c^\bullet(\Psi, v_2) = 2$, unless $\l < 2,\, \mu < 1,\, \l^2\mu < 1$, no contact embedding $\Psi$ of $X$ into  $A_{\l\mu}(X)$ exists.  \hfill $\diamondsuit$
}
\end{example}

For any boundary generic Reeb vector field $v_\b$ on a compact $(2n+1)$-manifold $X$, we have $\d_jX(v_\b) = \d(\d_{j-1}^+X(v_\b))$ 
for all $j \geq 1$.  

In what follows, we omit the dependence of the Morse strata on the vector field $v_\b$.

Thus, for an {\sf odd} $j$, by Stokes' theorem, 
$$\int_{\d_jX} (d\b)^{\frac{2n+1-j}{2}} = \int_{\d_{j-1}^+X} d\big((d\b)^{\frac{2n+1-j}{2}}\big) = \int_{\d_{j-1}^+X} \mathbf 0 \; = \; 0.$$
In contrast,  for an {\sf even} $j$, we get $$\int_{\d_jX}\b \wedge (d\b)^{\frac{2n-j}{2}} = \int_{\d_{j-1}^+X} d\big(\b \wedge (d\b)^{\frac{2n-j}{2}}\big) = \int_{\d_{j-1}^+X} (d\b)^{\frac{2n+2-j}{2}}.$$ 
Conjecturally, the last integral is positive (see Conjecture \ref{conj.positive_on_Morse_strata} and Examples 5.1-5.2). 
\smallskip

Motivated by these observations, we introduce some 
{\sf numerical measures of the stratified concavity/convexity} of the Reeb flows $v_\b$ relative to the boundary $\d_1X$:
\begin{eqnarray}\label{eq.j-wrinkle_A}
\kappa_j^+(\b) & =: \; & \int_{\d_j^+X} (d\b)^{\frac{2n+1-j}{2}} \;\;\; (\equiv \kappa_{j+1}(\b)),  \nonumber\\
\kappa_j(\b) & =\; 0 & \text{\; for an odd } j \in [1, 2n-1].
\end{eqnarray}
\begin{eqnarray}
\kappa_j^+(\b) & =: \; &\int_{\d_j^+X} \b \wedge (d\b)^{\frac{2n-j}{2}}, \nonumber  \label{eq.j-wrinkle_AA} \\
\kappa_j(\b) & =: \;& \int_{\d_jX} \b \wedge (d\b)^{\frac{2n-j}{2}} \;\;\; (\equiv \kappa^+_{j-1}(\b))
 \text{\; for an even } j \in [2, 2n], 
 \label{eq.j-wrinkle_BB}
\end{eqnarray}

Of course, these quantities depend on the choice of the contact form $\b$ that represents a given contact structure $\xi$ and whose Reeb vector field $v_\b$ is boundary generic. They measure how ``wrinkled" $\d_1X$ is in relation to the Reeb flow. By the very definition of $\kappa_j^+(\b), \kappa_j(\b)$, they belong to the category of ``boundary data". In general, we do not know whether these quantities are positive for all $j \geq 2$; Conjecture \ref{conj.positive_on_Morse_strata} suggests they are. 
\smallskip

With these $\kappa$-quantities in place, note that the inequality (\ref{eq.isoperimetric}) can be written as $$\kappa^+_{0}(\b) \leq \mathsf{diam}_\mathcal R(\b)\cdot \kappa^+_{1}(\b),$$  and the inequality (\ref{eq.isoperimetric_A})  as $$\kappa^+_{0}(\b) \leq \mathsf{diam}_\mathcal R(\b)\cdot\kappa_{2}(\b),$$ where $\kappa^+_{0}(\b) =: \int_X \b \wedge (d\b)^n.$


\subsection{Invariance of the signed volumes $\{\kappa_j(\b)\}$ under special  deformations of $\b$}

\begin{definition}\label{def.volume-preserving_families}  Let $\mathcal B= \{\b_t\}_{t \in [0, 1]}$ be a family of contact forms on a compact connected smooth $(2n+1)$-manifold $X$ with boundary.  \smallskip

$\bullet$ We say the family $\mathcal B$ is {\sf variationally boundary exact}
if $$\frac{\d}{\d t}\b_t \Big|_{\d_1X} = d\eta_t $$ 
for a $t$-family of smooth functions $\eta_t: \d_1X \to \R$.\smallskip

$\bullet$  We say the family $\mathcal B$ is {\sf variationally boundary $(2n+1-j)$-exact} if, for some smooth family of $(2n-j)$-forms $\{\zeta_t\}_{t \in [0, 1]}$ on $\d_1X$,
\begin{eqnarray}
\Big\{\frac{\d}{\d t}\Big[ \b_t \wedge (d\b_t)^{\frac{2n - j}{2}} \Big] \Big\} \Big |_{\d_1X} = d\zeta_t , \text{ where } j \in [2, 2n] \text{ is even, } \label{eq.boundary_exact_A} \\
\Big\{\frac{\d}{\d t}\Big[(d\b_t)^{\frac{2n+1 - j}{2}} \Big] \Big\} \Big |_{\d_1X} = d\zeta_t, 
 \text{ where  } j \in [1, 2n-1] \text{ is odd }. \label{eq.boundary_exact_B} 
\end{eqnarray}
 Occasionally, we complement these requirements with the additional one:
 \begin{eqnarray} \label{eq.boundary_exact_C}
 \zeta_t\big|_{\d_{j+1}X(v_{\b_t})} = 0. \quad \quad \quad \quad\quad \quad \quad \quad \hfill \diamondsuit
 \end{eqnarray}
\end{definition}

\begin{lemma}\label{lem.variationally_boundary_exact_family} Any variationally boundary exact family $\mathcal B$ of contact forms $\{\b_t\}$ is variationally boundary $(2n+1-j)$-exact for all $j$. \smallskip

If the functions ${\eta_t}$ are such that $\eta_t|_{\d_{j+1}X(v_t)} = 0$, then $\zeta_t|_{\d_{j+1}X(v_t)} = 0$.
\end{lemma}

\begin{proof}   Note that, for any $t$-family differential forms, their exterior differential $d =: d_t$ commutes with the partial derivative operator $\frac{\d}{\d t}$ and with the restriction of forms to $\d_1X$. 

Consider the derivative $$\frac{\d}{\d t}\Big[ (d\b_t)^k \Big] = k\; \frac{\d}{\d t}\Big(d\b_t\Big) \wedge \big[ (d\b_t)^{k-1} \big] =  k\; d \Big(\frac{\d}{\d t}(\b_t)\Big) \wedge \big[ (d\b_t)^{k-1} \big].$$
Let $\eta_t: \d_1X \to \R$ be a smooth family of functions from Definition \ref{def.volume-preserving_families}. Replacing $\frac{\d}{\d t}(\b_t)\big |_{\d_1X}$ with $d\eta_t$ in these formulas and using that $d(d\eta_t) = 0$, we see that $\frac{\d}{\d t}\Big[ (d\b_t)^k \Big]\Big|_{\d_1X} = 0$.  
Similarly, 
$$\frac{\d}{\d t}\Big[ \b_t \wedge (d\b_t)^k \Big] = \frac{\d}{\d t}\Big[ \b_t \Big] \wedge  (d\b_t)^k + \b_t \wedge \frac{\d}{\d t}\Big[ (d\b_t)^k \Big] 
 $$
Since we proved that $\frac{\d}{\d t}\Big[ (d\b_t)^k \Big]\Big|_{\d_1X} = 0$,  the restriction of this formula to $\d_1X$ collapses to $\frac{\d}{\d t}\Big[ \b_t \Big] \wedge  (d\b_t)^k\Big|_{\d_1X}$. 
Replacing $\frac{\d}{\d t}\Big[ \b_t \Big] \Big|_{\d_1X}$ with $d\eta_t$ once more, we get $$\frac{\d}{\d t}\Big[ \b_t \Big] \wedge  (d\b_t)^k\Big|_{\d_1X} = d\eta_t \wedge  (d\b_t)^k\Big|_{\d_1X} = d\Big[\eta_t \cdot  (d\b_t)^k\Big]\Big|_{\d_1X} = d\zeta_t.$$
Thus, (\ref{eq.boundary_exact_A}) is valid. 
If the functions ${\eta_t}$ are such that $\eta_t|_{\d_{j+1}X(v_t)} = 0$, then $\zeta_t|_{\d_{j+1}X(v_t)} = 0$; so the property (\ref{eq.boundary_exact_C}) holds as well. 
\end{proof}
%
%
\begin{proposition} \label{prop.t-independent_kappas}
Let $\mathcal B=\{\b_t\}_{t \in [0, 1]}$ be a variationally boundary $(2n+1-j)$-exact family of contact forms 
on a compact connected smooth $(2n+1)$-manifold $X$. Assume that their Reeb vector fields $v_{\b_t}$ are 
boundary generic for $t = 0, 1$, and for almost all $t \in (0, 1)$. \smallskip

$\bullet$ Then, for $j \equiv 0 \mod 2$, the integrals $\kappa_j(\b_t)$  from (\ref{eq.j-wrinkle_BB}) and $\kappa_{j-1}^+(\b_t)$  from (\ref{eq.j-wrinkle_A}) 
are $t$-independent. If, in addition, the property (\ref{eq.boundary_exact_C}) holds, then $\kappa^+_j(\b_t)$ is also $t$-independent for all $j \equiv 0 \mod 2$.
\smallskip
 
$\bullet$ If $\mathcal B=\{\b_t\}_{t \in [0, 1]}$ is variationally boundary exact family, then $\{\kappa_j(\b_t)\}$ are $t$-independent for all even $j \in [2, 2n]$. If in addition,  (\ref{eq.boundary_exact_C}) holds for the function $\eta_t$, then $\{\kappa^+_j(\b_t)\}$ are $t$-independent for all $j \equiv 0 \mod 2$.
\end{proposition}

\begin{proof} 
Let $\mathsf d$ be the exterior derivative operator on $X \times [0, 1]$ and $d_t$ the exterior derivative operator on a typical slice $X \times \{t\}$. Let $v_t =: v_{\b_t}$.  Consider the following sets: $$\mathcal Y_j =: \bigcup_{t \in [0, 1]}\d_jX(v_t) \times \{t\} \subset X \times [0,1], \text{ and } 
\mathcal Y_j^+ =: \bigcup_{t \in [0, 1]}\d_j^+X(v_t) \times \{t\} \subset X \times [0,1],$$

Although $\mathcal Y_j $ may have singularities, located in some $t$-slices, by the hypothesis, such special $t$'s have measure zero in $[0, 1]$; so one can integrate a smooth differential form against the rectifiable chain $\mathcal Y_j$. 

We start with an even $j$. Put $s =\frac{2n-j}{2}$. Then 
$$\kappa_j(\b_1) -\kappa_j(\b_0) =:  \int_{\d_jX(v_1)} \b_1 \wedge (d_1\b_1)^s - \int_{\d_jX(v_0)} \b_0 \wedge (d_0\b_0)^s 
\stackrel{Stokes}{=}\quad  \int_{\mathcal Y_j}  \mathsf d\big(\b_t \wedge (d_t\b_t)^s\big)$$ 
$$ = \int_{\mathcal Y_j} (d_t \b_t)^{s+1} \;+ \int_{\mathcal Y_j} \frac{\d}{\d t}\Big[\b_t \wedge (d_t\b_t)^s\Big] \wedge dt$$
Each of these two integrals vanishes. Indeed, $dt$ is not present in the $(2s+2)$-form $(d_t \b_t)^{s+1}$. Thus its restriction to the $(2s+2)$-chain $\mathcal Y_j \subset \d_1X \times [0,1]$ vanishes. The second integral is also zero by the variationally boundary $(2n+1-j)$-exact hypothesis: $$\int_{\mathcal Y_j} \frac{\d}{\d t}\big(\b_t \wedge (d_t\b_t)^s\big) \wedge dt \stackrel{Fubini}{=}  
\int_{[0, 1]} \Big(\int_{\d_jX(v_{\b_t})} \frac{\d}{\d t}
\Big[\b_t \wedge (d_t\b_t)^s \Big]\Big) dt$$
$$ =  \int_{[0, 1]} \Big(\int_{\d_jX(v_{\b_t})}
d_t\zeta_t\Big) dt = 0,$$ 
since integrating the exact form $d_t\zeta_t$ over the cycle $\d_jX(v_{\b_t}) \subset \d_1X$ results in zero.
Therefore, $\kappa_j(\b_1) = \kappa_j(\b_0)$ for this particular even  $j  \in [2, 2n]$.
\smallskip

 By Lemma \ref{lem.variationally_boundary_exact_family}, the  second bullet's claim follows for all \emph{even} $j  \in [2, 2n]$. \smallskip

Using Stokes' theorem identity $\{\kappa_j(\b) = \kappa_{j-1}^+(\b)\}$ ($j \equiv 0 \mod 2$), we conclude that, for \emph{odd} $j \in [1, 2n-1]$, the quantities $\{\kappa_{j}^+(\b)\}_j$ are $\mathcal B$-deformation invariants as well. \smallskip

\smallskip

If we add property (\ref{eq.boundary_exact_C}) to property (\ref{eq.boundary_exact_A}), then a similar computation shows that $\kappa^+_j(\b_1) = \kappa^+_j(\b_0)$ for an even $j$. This time, Stokes' theorem is applied to $\mathcal Y_j^+$. Its boundary consists of tree parts: $\d_j^+X(v_0), \d_j^+X(v_1)$, and their complement $\delta\mathcal Y_j^+ = \mathcal Y_{j+1}$. Therefore,
$$\kappa^+_j(\b_1) -\kappa^+_j(\b_0) =  \int_{\mathcal Y^+_j}  \mathsf d\big(\b_t \wedge (d_t\b_t)^s\big) \pm  \int_{\delta\mathcal Y^+_j} \b_t \wedge (d_t\b_t)^s$$
$$ = \int_{\mathcal Y^+_j} (d_t \b_t)^{s+1} \;+ \int_{\mathcal Y^+_j} \frac{\d}{\d t}\Big[\b_t \wedge (d_t\b_t)^s\Big] \wedge dt\; \pm \;  \int_{\delta\mathcal Y^+_j} \b_t \wedge (d_t\b_t)^s.$$
Since $dt$ is not present in the first and third integrals, they vanish. As a result,
$$\kappa^+_j(\b_1) -\kappa^+_j(\b_0) = \int_{[0, 1]} \Big(\int_{\d^+_jX(v_{\b_t})} \frac{\d}{\d t}
\Big[\b_t \wedge (d_t\b_t)^s \Big]\Big) dt$$
$$ =  \int_{[0, 1]} \Big(\int_{\d^+_jX(v_{\b_t})}
d_t\zeta_t\Big) dt\; \stackrel{Stokes}{=}\; \int_{[0, 1]} \Big(\int_{\d_{j+1}X(v_{\b_t})}
\zeta_t\Big) dt = 0$$
by (\ref{eq.boundary_exact_C}).
\hfill
\end{proof}

Consider now the following quantities (see Definition \ref{def.average_Reeb_length}):
\begin{eqnarray}
\mathcal K_j^+(\xi) & =: \; \inf_{\{\b \leadsto \xi\; |\; \mathsf{av}_\mathcal R(\b) =1\}} 
\big\{|\kappa_j^+(\b)|\big\} &  \text{\; for an odd } j \in [1, 2n-1],\label{eq.j-wrinkle_B} \\
\mathcal K_j^+(\xi) & =: \; \inf_{\{\b \leadsto \xi\; |\; \mathsf{av}_\mathcal R(\b) =1\}  } 
\big\{|\kappa_j^+(\b)|\big\} & \text{\; for an even } j \in [2, 2n], \nonumber \\
\mathcal K_j(\xi) & =: \;\inf_{\{\b \leadsto \xi\; |\; \mathsf{av}_\mathcal R(\b) =1\}  } \big\{|\kappa_j(\b)|\big\} & 
 \text{\; for an even } j \in [2, 2n], \label{eq.j-wrinkle_C}
\end{eqnarray}
where the infimums are taken over all contact forms $\{\b\}$ that produce a given contact structure $\xi$ and such that $\mathsf{av}_\mathcal R(\b) =1$ (see Definition \ref{def.average_Reeb_length} ). 
They measure how ``wrinkled"  the Morse stratification $\{\d_j^\pm X(v_\b)\}_j$ could be with respect to a given contact structure $\xi = \ker \b$ on $X$. As a practical matter, the computation of $\{\mathcal K_j^+(\xi)\}_j, \{\mathcal K_j(\xi)\}_j$ could be quite challenging...\smallskip


\begin{figure}[ht]
\centerline{\includegraphics[height=2.3in,width=3.5in]{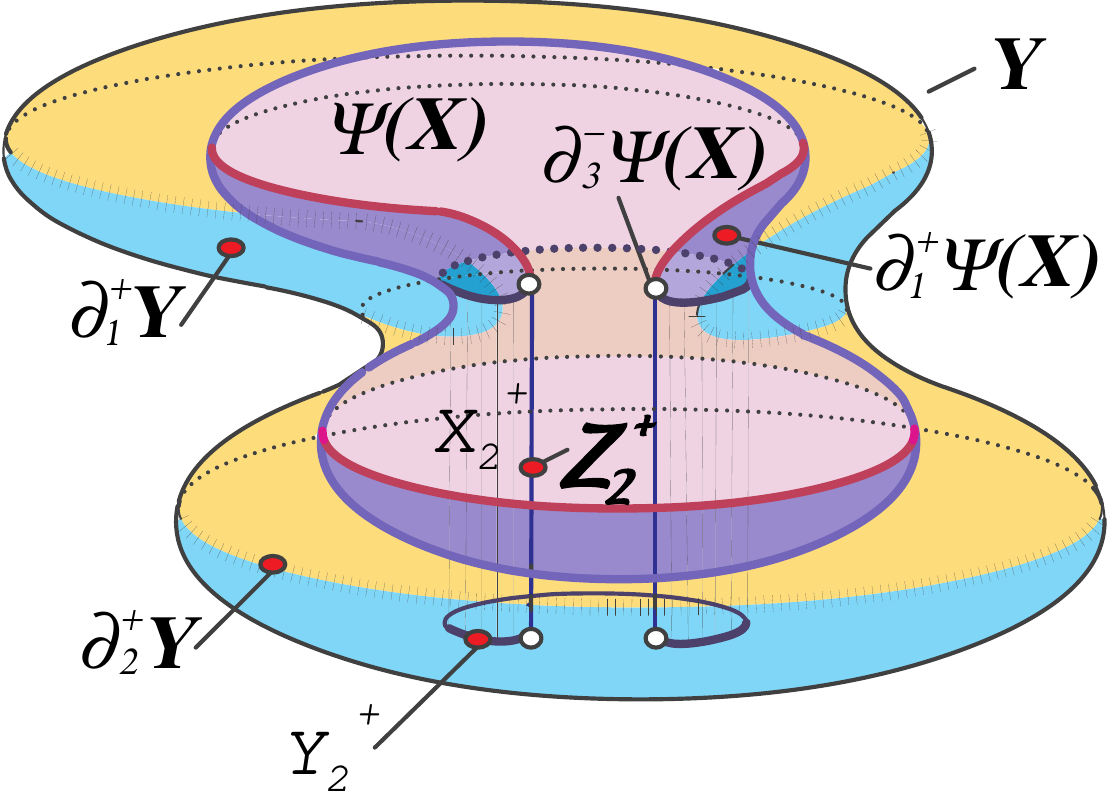}}
\bigskip
\caption{\small{By Theorem \ref{th.detecting_volumes}, $\kappa_1^+(\b_X) = \kappa_2(\b_X) = -\int_{\mathcal Y_{1}^+}d\b_Y$. The figure illustrates the geometry of $3$-folds $X$ and $Y$ and of an embedding $\Psi: X \to Y$ that leads to this formula. Although the relevant for the formula loci are: the surface $\Psi(\d_1^+X)$, the solid $\mathsf X^+_1 \subset Y$,  and the surface $\mathsf Y^+_1 \subset \d_1^+ Y$, they are not shown in the figure due to artistic limitations; instead, we have chosen to show similar loci of one dimension lower: the curve $\Psi(\d_2^+X)$, the surface $\mathsf X^+_2$, and the curve $\mathsf Y^+_2$.}}
\label{fig.3_contact_2}
\end{figure}

\smallskip

 \begin{theorem}{\bf (Holography for the $\kappa_j^+$-volumes\footnote{A priori these quantities may be negative; conjecturally, they are not.} under the contact embeddings)}\label{th.detecting_volumes} 

Let $X$ and $Y$ be compact smooth $(2n+1)$-manifolds, equipped with contact forms $\b_X, \b_Y$ and their Reeb vector fields $v_X$, $v_Y$, respectively. We assume that  $v_X$ and $v_Y$ are boundary generic and $v_Y$ is traversing.   
Let $\Psi: X \hookrightarrow \mathsf{int}(Y)$ be a smooth embedding such that $\b_X = \Psi^\ast(\b_Y)$.
For $j \in [1, 2n+1]$, consider the $(-v_Y)$-guided projections $\Pi_Y : \Psi\big(\d_{j}^+X(v_X)\big) \to \d_1^+Y(v_Y)$ and denote their images by $\mathsf Y^+_j$.\footnote{Here, as before, we assume that $\Psi(X)$ is ``semi-transparent" for the $(-v_Y)$-flow.}
\smallskip

Then the integrals $\{\kappa_j^+(\b_X)\}_j$ from (\ref{eq.j-wrinkle_A}) and $\{\kappa_j(\b_X)\}_j$ from (\ref{eq.j-wrinkle_BB}) can be recovered from the $\Pi_Y$-projections $\{\mathsf Y^+_j\}_j$  and from the knowledge of the contact form $(\b_Y)^{\d _+} =: \b_Y||_{\d_1^+Y(v_Y)}$. 
Specifically, for an odd $ j \in [1, 2n-1]$, 
$$\kappa_j^+(\b_X) = -\int_{\mathsf Y_j^+}(d\b_Y)^{\frac{2n+1-j}{2}}$$
and, for an even $j \in [2, 2n]$,
$$\kappa_j(\b_X) = -\int_{\mathsf Y_{j-1}^+}(d\b_Y)^{\frac{2n+2-j}{2}}.$$
\end{theorem}

\begin{proof} 
We consider the union $\mathsf X_j^+$ of downward $(-v_Y)$-directed  trajectories through  the points of $\Psi(\d^+_j(X)(v_X))$. Note that $\dim(\mathcal X_j^+) = 2n+2-j$. Let $$\mathsf Y_j^+ =:\, \mathsf X_j^+ \cap \d_1^+Y(v_Y)$$---the $\Pi_Y$-image of $\d_{j}^+X(v_X)$---, and let (see Fig. \ref{fig.3_contact_2}) 
$$\mathsf Z_j^+ =:\, \d(\mathsf X_j^+) \setminus \big(\Psi\big(\d_{j}^+X(v_X)\big) \cup \mathsf Y_j^+ \big).$$ 

We claim that, for an odd $j$, $\int_{\mathsf Z_j^+} (d\b_Y)^{\frac{2n+1-j}{2}} = 0$ since $v_Y \in \ker(d\b_Y)$ and $\mathsf Z_j^+$ of dimension $2n +1 -j$ consists of portions of $v_Y$-trajectories. 

In what follows we abbreviate the notation $\d^+_jX(v_X)$ as $\d^+_jX$. 
\smallskip

{\sf (a)} For an odd $ j \in [1, 2n-1]$, we get
$$\kappa_j^+(\b_X) = \int_{\d^+_jX}\big(d\b_X\big)^{\frac{2n+1-j}{2}} =$$
$$=\int_{\d^+_jX}\Psi^\ast\big(d\b_Y\big)^{\frac{2n+1-j}{2}}  = 
 \int_{\Psi\big(\d^+_j(X)\big)}\big(d\b_Y\big)^{\frac{2n+1-j}{2}}.$$

By  Stokes' theorem, applied to $\mathsf X_j^+$, and using that $d\big((d\b_Y)^{\frac{2n+1-j}{2}}\big)=0$ and that  $\int_{\mathsf Z_j^+} (d\b_Y)^{\frac{2n+1-j}{2}} = 0$, we get  
$$\kappa_j^+(\b_X) = \int_{\Psi\big(\d^+_j(X)\big)}(d\b_Y)^{\frac{2n+1-j}{2}} = -\int_{\mathsf Y_j^+}(d\b_Y)^{\frac{2n+1-j}{2}}.
$$
This validates the claim of Theorem~\ref{th.detecting_volumes} for an odd $j$.
\smallskip

\smallskip

{\sf (b)} For an even $j \in [2, 2n]$,  by  Stokes' theorem, we get
$$\kappa_j(\b_X) = \int_{\d_jX}\b_X \wedge (d\b_X)^{\frac{2n-j}{2}} = \int_{\d_{j-1}^+ X}(d\b_X)^{\frac{2n+2-j}{2}} $$
$$ =\int_{\d_{j-1}^+ X}\Psi^\ast\big((d\b_Y)^{\frac{2n+2-j}{2}}\big)  = \int_{\Psi(\d^+_{j-1}(X))} (d\b_Y)^{\frac{2n+2-j}{2}}.
$$

Now we are facing the situation, described in the case {\sf (a)}. Therefore, $$\kappa_j(\b_X) = -\int_{\mathsf  Y_{j-1}^+}(d\b_Y)^{\frac{2n+2-j}{2}},$$ 
and thus is determined by $\mathsf Y_{j-1}^+$, the $\Pi_Y$-projection of the locus $\Psi\big(\d^+_{j-1}(X)(v_X)\big)$.
\end{proof}

\begin{example} 
\emph{
Let $\dim X = \dim Y = 5$. We adopt the notations and hypotheses of Theorem \ref{th.detecting_volumes}. Then $\kappa_1(\b_X) = 0 = \kappa_3(\b_X)$ and 
 $$\kappa_1^+(\b_X) = -\int_{\mathsf Y_1^+}(d\b_Y)^{2} = \kappa_2(\b_X), \quad \kappa_3^+(\b_X) = -\int_{\mathsf Y_3^+} d\b_Y =\kappa_4(\b_X).$$
 }

\emph{In contrast, we do not know whether $\kappa^+_2(\b_X)$ and $\kappa^+_4(\b_X)$ can be recovered in terms of the $\Pi_Y$-generated boundary data residing on $\d_1Y$.}
\hfill $\diamondsuit$
\end{example}

Combining Proposition \ref{prop.t-independent_kappas} with Theorem \ref{th.detecting_volumes} leads to the following claim.

\begin{corollary} 
Let $X$ and $(Y, \b_Y)$ be as in Theorem \ref{th.detecting_volumes}.  Consider a $t$-family  $\{\Psi_t: X \hookrightarrow \mathsf{int}(Y)\}_{t \in [0,1]}$ of orientation-preserving contact embeddings such that, for an even $j \in [2, 2n]$, the $t$-family of forms $\big\{\Psi_t^\ast\big(\b_Y \wedge (d\b_Y)^{\frac{2n-j}{2}}\big)\big\}$ is variationally boundary $(2n+1-j)$-exact in the sense of Definition \ref{def.volume-preserving_families}. 
\smallskip

Then the quantities  $\int_{\mathsf Y_{j-1}^+(t)}(d\b_Y)^{\frac{2n+2-j}{2}} = -\kappa_j(\Psi_t^\ast\big(\b_Y)) = - \kappa^+_{j-1}(\Psi_t^\ast\big(\b_Y))$ are $t$-indepen-dent. \hfill $\diamondsuit$
\end{corollary}



\section{Shadows of Legendrian submanifolds of $(X, \b)$ on the screen $\d_1^+X(v_\b)$} 

We consider a  contact form $\b$ and the cooriented contact distribution $\xi_\b$ on a compact smooth manifold $X$ of dimension $2n+1$. We assume that the Reeb vector field $v_\b$ is traversing and boundary generic.  

We denote by $\mathcal I_x(d\b)$, $x \in X$, a generic {\sf isotropic space} of the symplectic form $d\b|_{\xi_\b(x)}$, that is, a vector subspace of $\xi_\b(x)$ such that $d\b(w_1, w_2) = 0$ for any  $w_1, w_2 \in \mathcal I_x(d\b)$. In fact,  $\dim(\mathcal I_x(d\b)) \leq n$. 

By definition, $L \subset X$ is {\sf sub-Legendrian manifold},  
if $\b|_L \equiv 0$; thus the tangent space $T_x L \subset \ker_x(\b)$ is contained in some isotropic space $\mathcal I_x(d\b) \subset \ker_x(\b)$ for all $x \in L$. If $\dim(X) =n$, $L$ is called a {\sf Legendrian manifold}.
 \smallskip
 

Our next goal is to describe the images $L^\dagger$ of closed sub-Legendrian submanifolds  $L \subset \mathsf{int}(X)$, where $\dim(L) \leq n$, on the screen $\d_1^+X(v_\b)$ under the $(-v_\b)$-flow projection. The main question we address here: ``How to reconstruct $L$ from its shadow $L^\dagger$?". 
Note that these shadows of $L$ do not change when we apply a $t$-family ($t$ being sufficiently small) of local diffeomorphisms $\phi^t$ (as in Definition \ref{def.flow_generated_diff}), generated by the Reeb field $v_\b$. Since $\b$ and $d\b$ are invariant under $\phi^t$, we conclude that $\phi^t(L)$ is also Legendrian.
\smallskip

\begin{remark} 
\emph{
Our setting is reminiscent of the following well-studied classical case: the contact manifold  $X = \R \times T^*M$, where $T^*M$ is a cotangent bundle of a manifold $M$, and the obvious projection $\pi: X \to  T^*M$. The interplay between Legendrian submanifolds $L \subset X$ and their immersed Lagrangian shadows  $L^\dagger = \pi(L)$ is a subject of a deep research (see \cite{AGN}, \cite{EM} as general references). In fact, not any immersed Lagrangian submanifold $L^\dagger \subset M$ is a $\pi$-shadow of some Legendrian submanifold $L \subset X$. 
} 

\emph{
Another classical case of the interplay between Legendrian $L$ and its Lagrangian shadow $L^\dagger$ arise in  the study of prequantization bundles. Consider a smooth principle $S^1$-bundle $\pi: X \to M$, where $M$ is a manifold carrying a symplectic $2$-form $\om$, and a manifold $X$ is carrying a contact  $1$-form $\l$ such that $\pi^*(\om) = d\l$. Here the de Rham cohomology class $\frac{1}{2\pi} [\om] \in H^2(M; \R)$ of $\om$ is nontrivial and integral (belongs to the image of $H^2(M; \Z)$).
}

\emph{
Legendrian submanifolds $L$ of prequantization bundles and their Lagrangian shadows  $L^\dagger = \pi(L)$ in $M$ are again in the focus of an active  research \cite{Sab}. In particular, counting the points of self-intersection of $L^\dagger$ are of great interest. They correspond to segments of $v_\l$-trajectories (of $S^1$-trajectories) that connect pairs of points in $L$. Counting the points of self-intersection $L^\dagger \pitchfork L^\dagger$ and their behavior under contact isotopies of $L$ in all these examples is a fundamental problem \cite{AGN}, \cite{Her}.
}

\emph{ In contrast with these classical settings, in our case, the screen $\d_1^+X(v_\b)$---a vague analog of $M$--- is wrinkled, and the $(-v_\b)$-shadow $L^\dagger$ is cut in pieces due to the concavity of the boundary $\d_1X$ with respect to the Reeb flow. Moreover, the Reeb trajectories on the space of the $S^1$-bundle $\pi: X \to M$ are circles; in our case, they are closed intervals or singletons.
}
\hfill $\diamondsuit$
\end{remark}
\smallskip

By Lemma \ref{lem.positive_on_d_+}, $\mathsf{int}(\d_1^+X(v_\b))$ is an open symplectic manifold with the symplectic form $\om^\dagger =: d\b |_{\d_1^+X(v_\b)}$. Recall that the form $\om^\dagger$ degenerates along the boundary $\d_2X(v_\b) = \d(\d_1^+X(v_\b))$, since there $v_\b \in \ker(d\b) \cap T(\d_1X)$.
\smallskip

We denote by $\mathsf X_2^\updownarrow(v_\b)$ the union of $v_\b$-trajectories that intersect the locus $\d_2X(v_\b)$.

We call the union $\mathsf X_2(v_\b)$ of the downward $(-v_\b)$-trajectories through the locus $\d_2X(v_\b)$ a {\sf waterfall}. Let $\mathsf L(v_\b)$ the union of the downward $(-v_\b)$-trajectories through the points of a sub-Legendrian submanifold $L \subset X$.  Put $L^\dagger =: \mathsf L(v_\b) \cap \d_1^+X(v_\b)$ (see Fig. \ref{fig.3_contact_3}).\smallskip

\begin{definition}\label{def.L_is_in_deneral_position}

$\bullet$ We say that a sub-Legendrian manifold $L \subset \mathsf{int}(X)$ is in a {\sf  weakly general position} to the $v_\b$-flow, if $L \setminus (L \cap \mathsf X_2(v_\b))$ is open and dense in $L$.\smallskip
 
$\bullet$ We say that a sub-Legendrian manifold $L \subset \mathsf{int}(X)$ is in a {\sf  general position} to the $v_\b$-flow, if $L$ is transversal to the locus  $\mathsf X_2^\updownarrow(v_\b)$. 
\hfill $\diamondsuit$
\end{definition}

\begin{figure}[ht]
\centerline{\includegraphics[height=2.3in,width=3.7in]{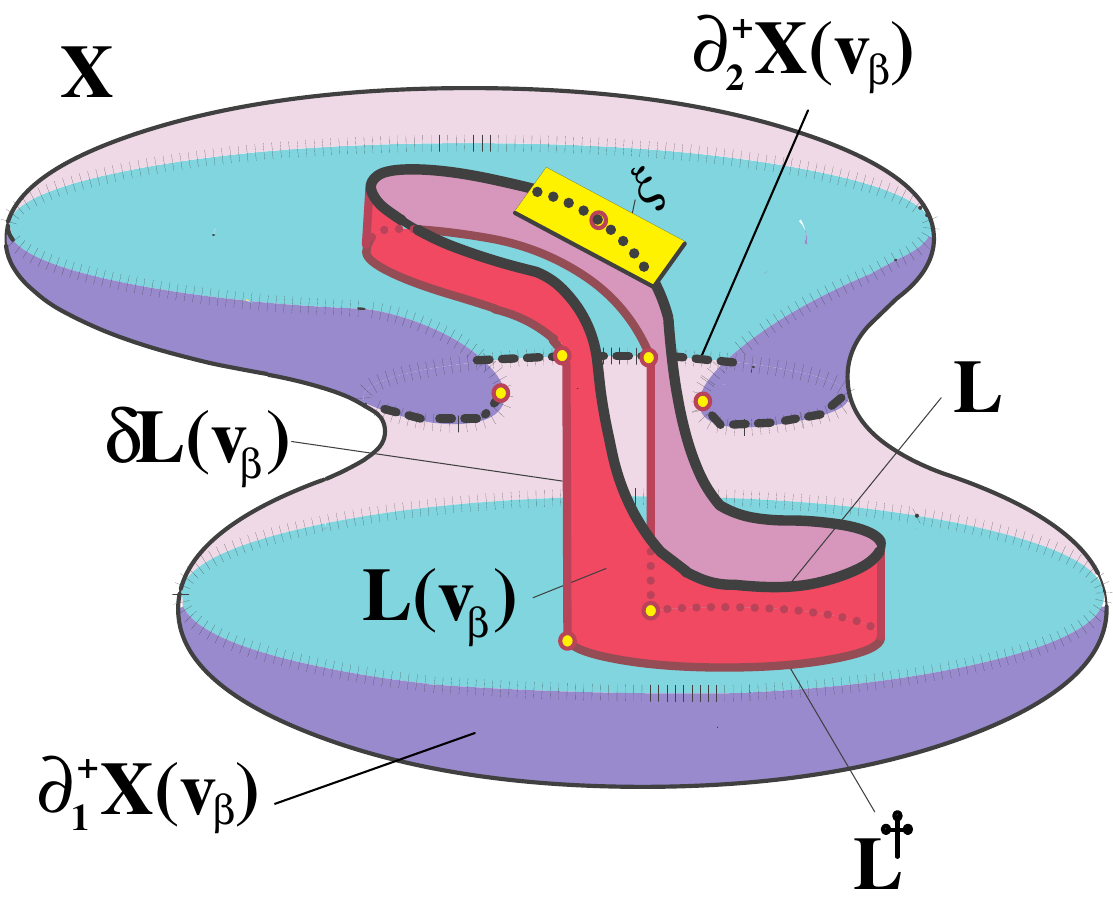}}
\bigskip
\caption{\small{The $(-v_\b)$-generated shadow $L^\dagger \subset \d_1^+X(v_\b)$ of a closed Legendrian manifold $L \subset X$, shown as a bold loop in the interior of the bulk $X$. In the figure, $L^\dagger$ consists of two arcs. The surface $\mathsf L(v_\b)$ consists of $(-v_\b)$-trajectories falling down from $L$.  Note the locus $\delta \mathsf L(v_\b) \subset X$ that consists of $(-v_\b)$-trajectories falling down from the $1$-dimensional locus $L^\dagger \cap \d_2^+X(v_\b)$.}}
\label{fig.3_contact_3}
\end{figure}
 
\begin{lemma}\label{lem.L_to_L} Let $L \subset \mathsf{int}(X)$ be a sub-Legendrian manifold.

The locus $\mathsf{int}(L^\dagger) =: L^\dagger \setminus (L^\dagger \cap \mathsf X_2(v_\b))$ is an open (possibly disconnected) {\sf sub-Lagrangian} submanifold of the open $(2n)$-manifold $$\d_1^+X(v_\b) \setminus \big(\mathsf X_2(v_\b) \cap \d_1^+X(v_\b)\big).$$ That is, for any $x \in \mathsf{int}(L^\dagger)$, the tangent space $T_x(\mathsf{int}(L^\dagger)) \subset \mathsf I_x\big(\om^\dagger\big)$, an isotropic subspace of the symplectic form $\om^\dagger = d\b |_{\mathsf{int}(\d^+_1X(v_\b))}$ at $x$. 
\end{lemma}

\begin{proof} By Lemma \ref{lem.positive_on_d_+}, the $2$-form $\om^\dagger = d\b |_{\d^+_1X(v_\b)}$ is symplectic on $\mathsf{int}(\d_1^+X(v_\b))$. Since $d\b$ is $v_\b$-invariant, for any $x \notin \d_2X(v_\b)$ and for any isotropic space $\mathcal I_x\big(d\b \big) \subset \ker_x(\b)$, its isomorphic image under the differential of the $(-v_\b)$ projection on $\mathsf{int}(\d_1^+X(v_\b))$, is an isotropic space for  the  symplectic form $\om^\dagger$.  Therefore, $\mathsf{int}(L^\dagger)$ is a Lagrangian submanifold of the open manifold $\d_1^+X(v_\b) \setminus (\mathsf X_2(v_\b) \cap \d_1^+X(v_\b))$. \hfill
\end{proof}

\begin{remark}
\emph{
Since $L$ is embedded in $X$ and $v_\b$ is transversal to $\xi$,  $\mathsf{int}(L^\dagger)$ is \emph{immersed} in the open $(2n)$-manifold $\d_1^+X(v_\b) \setminus \big(\mathsf X_2(v_\b) \cap \d_1^+X(v_\b)\big)$. 
}
\hfill $\diamondsuit$
\end{remark}

\begin{lemma}\label{lem.any_L-projection_is_L} Let $\g$ be a Reeb $v_\b$-trajectory through a point $a_\star \in X$, which hits $\d_1^+X(v_\b)$ transversally at a point $b_\star$. Let $L^\dagger$ be a given Lagrangian submanifold (with respect to the symplectic form $\om^\dagger =: d\b|_{\mathsf{int}(\d^+_1X(v_\b))}$)  of $\d_1^+X(v_\b)$, which contains $b_\star$. Then, in the vicinity of $a_\star$, there exists a unique germ of Legendrian submanilold $L \subset X$ that contains $a_\star$ and whose $(-v_\b)$-governed projection is the germ of $L^\dagger$ at $b_\star$.   
\end{lemma}

\begin{proof} In the vicinity of $a_\star$, by the Darboux  theorem,  we can pick  coordinates $(z, x_1, y_1,\; \ldots , \;\hfill\break  x_n, y_n)$ such that:  they vanish at $a_\star$, and $\b = dz + \sum_{i=1}^n x_i \,  dy_i$, and $d\b = \sum_{i=1}^n dx_i \wedge dy_i$.  Then $v_\b = \d_z$. In these coordinates, the germ at $a_\star$ of the Legendrian submanilold $L$ in question is given, in its parametric form, by some  smooth functions $z(\vec t ),\, x_1(\vec t ),\, y_1(\vec t ),  \, \ldots \, ,\hfill\break  x_n(\vec t ),\, y_n(\vec t ),$ where $\vec t = (t_1, \ldots , t_n) \in \R^n$. We assume that the tangent to $L$ vectors $$\tau_i = \Big(\frac{\d z}{\d t_i}, \frac{\d x_1}{\d t_i}, \frac{\d y_1}{\d t_i}, \, \ldots , \frac{\d x_n}{\d t_i}, \frac{\d y_n}{\d t_i}\Big),$$  $1 \leq i \leq n$, are linearly independent in the vicinity of $a_\star$. Put $$\tilde \tau_i = \Big(\frac{\d x_1}{\d t_i}, \frac{\d y_1}{\d t_i}, \, \ldots , \frac{\d x_n}{\d t_i}, \frac{\d y_n}{\d t_i}\Big).$$
We first will prove the assertion of the lemma for the case when $\d_1^+X(v_\b)$ is equal to the hyperplane  $H = \{z = const\}$, where the point $\g \cap H$ resides in the interval of $\g$, bounded by $a_\star$ and $b_\star$. 

The property of $L$ being Legendrian can be expressed by the set of equations: 
\begin{eqnarray}\label{eq.AAX}
\{\b(\tau_i) = 0\}_{1 \leq i \leq n} = \Big\{-\frac{\d z}{\d t_i} = \sum_{q=1}^n x_q \frac{\d y_q}{\d t_i}\Big\}_{1 \leq i \leq n},
\end{eqnarray}
If these equations are solvable, then the system 
\begin{eqnarray}\label{eq.BBX}
\{d\b(\tau_i, \tau_j) = d\b(\tilde\tau_i, \tilde\tau_j) = 0\}_{1 \leq i < j \leq n}, 
\end{eqnarray}
must be solvable as well. 

We can solve the equations in (\ref{eq.AAX}) for $z(\vec t )$, provided that the RHS of equation (\ref{eq.AAX}) is a gradient vector field whose components are $\Big\{\sum_{q=1}^n x_q \frac{\d y_q}{\d t_i}\Big\}_{1 \leq i \leq n}$, or that
\begin{eqnarray}\label{eq.CCX}
\Big\{\frac{\d\tilde\tau_i}{\d t_j} = \frac{\d\tilde\tau_j}{\d t_i}\Big\}_{1 \leq i < j \leq n}.
\end{eqnarray}
Moreover, the solution $z(\vec t )$ of (\ref{eq.AAX}) with the property $z(\vec 0) = a_\star$ will be unique. 

By a direct computation, the  second group of equations (\ref{eq.BBX}) coincides 
with the group of equations in  (\ref{eq.CCX}). In turn, the equations in (\ref{eq.CCX}) claim that $\{d\b(\tilde\tau_i, \tilde\tau_j) = 0\}_{1 \leq i \neq j \leq n}$. In other words, they claim that the vectors $\{\tilde\tau_i\}_{1 \leq i \leq n}$ belong to some $n$-space $\mathcal Iso(d\b|_H) \subset H$,  so that the $n$-manifold $L_H^\dagger$, given by the functions $\{x_1(\vec t ),\, y_1(\vec t ),  \, \ldots \, , x_n(\vec t ),\, y_n(\vec t )\}$ is Lagrangian in the hypersurfuce $H$. The same argument shows that the germ of any Lagrangian submanifold $L_H^\dagger \subset H$, passing through the point $\vec 0 \in H$, is the shadow, under the $z$-projection, of the germ of unique Legendrian $L$, containing the origin of $\R^{2n+1}$.

It remains to prove that any germ of a Lagrangian $L_H^\dagger \subset H$ (with respect to $d\b|_H$) at the point $\g \cap H$, with the help of $(-v_\b)$-flow, projects onto a germ of a Lagrangian $L^\dagger \subset \d_1^+X(v_\b)$ (with respect to $d\b|_{\d_1^+X(v_\b) \setminus \d_2X(v_\b)}$) at the point $b_\star = \g \cap \d_1^+X(v_\b)$. The latter observation follows from the transversality of $\g$ to $\d_1^+X(v_\b)$ at $b_\star$ and the $v_\b$-invariance of the form $d\b$. 
\end{proof}

For a given Lagrangian submanifold $L^\dagger \subset \mathsf{int}(\d_1^+X(v_\b))$, consider the (discontinuous) function 
$\Psi^\b_{L^\dagger}: L^\dagger \to \R_+$, defined by $$\Psi^\b_{L^\dagger}(x) = \int_{[x,\; C_{v_\b}(x)]} \b,$$
where $C_{v_\b}$ is the causality map (see (\ref{def.causal_map})). Note that, for a connected component $L^\ddagger$ of $L^\dagger \subset  \d_1^+X(v_\b) \setminus (\mathsf X_2(v_\b) \cap \d_1^+X(v_\b))$, the function $\Psi^\b_{L^\ddagger} =: \Psi^\b_{L^\dagger}\big |_{L^\ddagger}$ is continuous and smooth.

Put $\mathcal D_{L^\ddagger}^\b =: \sup_{x \in L^\ddagger} \big\{ \Psi^\b_{L^\ddagger}(x)\big\}.$
 One may think of this number as the ``Reeb elevation" of the component $L^\ddagger$. Evidently, $\mathsf{diam}_\mathcal R(\b) \geq \mathcal D^\b_{L^\ddagger}$.
\smallskip

Let us revisit the given Legendrian manifold $L \subset X$ whose $(-v_\b)$-projection is $L^\dagger$. For any $x \in \mathsf{int}(L^\dagger)$, take the first point $y(x) \in \g_x$ that belongs to $L$ and consider the function $$\psi^\b_L(x) =: \int_{[x, y(x)]} \b,$$ where $[x, y(x)] \subset \g_x$ is the segment of the trajectory $\g_x$, bounded by $x$ and $y(x)$. \smallskip

Evidently, $\psi^\b_L(x) \leq \Psi^\b_{L^\ddagger}(x)$ 
for any $x \in L^\ddagger \subset L^\dagger$.

\begin{corollary}\label{cor.LIFTING_L_dagger} Let $\b$ be a contact form whose Reeb vector  field $v_\b$ is traversing and boundary generic. Let $L \subset \mathsf{int}(X)$ be a 
Legendrian submanifold in a weakly general position to the waterfall $\mathsf X_2(v_\b)$. 
Let $L^\dagger$ be its $(-v_\b)$-guided projection of $L$ on $\d_1^+X(v_\b)$.

Consider the connected components $\{L^\ddagger_\kappa\}_\kappa$ of the Lagrangian manifold $\mathsf{int}(L^\dagger)$, immersed in $\d_1^+X(v_\b) \setminus (\mathsf X_2(v_\b) \cap \d_1^+X(v_\b)).$ In each component $L^\ddagger_\kappa$, we pick a base point $b_\kappa$, distinct from the points of self-intersection of $L^\dagger$. \smallskip 

Then the knowledge of the shadow $L^\dagger$, together with the set the base points $\{b_\kappa\}_\kappa$ and the set of values $\{\psi^\b_L(b_\kappa)\}_\kappa$, allows for a  reconstruction of the Legendrian embedding $L \subset  \mathsf{int}(X)$.
\end{corollary}


\begin{proof} The reconstruction of $L$ resembles the process of analytic continuation. 

Let $\pi_{v_\b}: L \to L^\dagger$ be the $(-v_\b)$-guided projection.
Let $L^\dagger_\times \subset L^\dagger$ denotes the subset  consisting of points that have multiple $\pi_{v_\b}$-preimages in $L$, and let $\mathcal L^\dagger_2 =: L^\dagger \cap \mathsf X_2(v_\b)$. By the weakly general position hypothesis, the set $L^\dagger \setminus (L^\dagger_\times \cup \mathcal L^\dagger_2)$ is open and dense in $L^\dagger$.

Knowing the point $b_\kappa \in L^\ddagger_\kappa$ makes it possible, with the help of the value $\psi^\b_L(b_\kappa)$, to locate a unique point $a_\kappa \in L \cap \g_{b_\kappa}$.  There is a unique germ $L_\kappa$ of $L$ in the vicinity of $a_\kappa$ that projects onto the germ of $L^\ddagger_\kappa$ at $b_\kappa$. 

Let $U^\dagger_\kappa$ be an open neighborhood of $b_\kappa$ in $L^\dagger$ over which the reconstruction of $L$ is achieved. 
Take a point $x \in L^\dagger$ that belongs to the boundary of $U^\dagger_\kappa$ and such that $x \notin L^\dagger_\times \cup \mathcal  L^\dagger_2$.  Then by Lemma \ref{lem.any_L-projection_is_L}, there exists a unique lift $L^\bullet$ to $X$ of $L^\dagger$ over some open neighborhood $U_x \subset L^\dagger$ of $x$, such that the point $y= \g_x \cap L$  belongs to $L^\bullet$.  Since, over $U_x \cap U^\dagger_\kappa$, both $L^\bullet$ and $L$ are the lifts of $L^\dagger$ that share the point $y$, again  by Lemma \ref{lem.any_L-projection_is_L}, we conclude that $L^\bullet = L$ over $U_x$. As a result, the reconstruction of $L$ over $U_x \cup U^\dagger_\kappa$ is available. Therefore, the maximal open set in $L^\dagger$, over which the lift of $L^\dagger$, subject to the constraint ``$a_\kappa$ belongs to the lift", is possible is a connected component $L^\ddagger$ of the open set $L^\dagger \setminus (L^\dagger_\times \cup \mathcal L^\dagger_2)$.  Since $L$ is in weakly general position with respect to the waterfall, the union of such  lifts, taken over all $\kappa$'s, is a dense  subset $L \setminus (L \cap \pi_{v_\b}^{-1}(\mathcal L^\dagger_2 \cup L^\dagger_\times))$ of $L$.  
Therefore, taking the closure of this union, produces the given $L \subset X$.
\end{proof}

\begin{remark}
\emph{Note that the very existence of Legendrian submanifold $L \subset X$ insured that the assembly instructions in the proof of Corollary \ref{cor.LIFTING_L_dagger} were consistent. }

\emph{ Even when the Reeb vector field $v_\b$ is convex (so the waterfall $\mathsf X_2(v_\b) = \d_2^-X(v_\b)$) and the $v_\b$-trajectories $\g_x$ depend continuously on $x \in \d_1^+X(v_\b)$, we do not know whether any closed Lagrangian submanifold $L^\dagger \subset \d_1^+X(v_\b)$ is the $\pi_{v_\b}$-shadow of a Legendrian manifold $L \subset X$. Corollary \ref{cor.LIFTING_L_dagger} claims only that, if such $L$ exists, then choosing a point $a_\star \in L$ makes it unique.  It seems that to tackle this question we need to be able to estimate from above the function $\Psi^\b_{L^\dagger}(x)$ (or at least the Reeb elevation $\mathcal D_{L^\dagger}^\b$ of $L^\dagger$) in terms of the form $\b$ in the vicinity of $L^\dagger \subset \d_1^+X(v_\b)$.
}
\hfill $\diamondsuit$
\end{remark}

\begin{proposition}\label{prop.holography_for_Legendrians} Let $\b$ be a  contact form on $X$, and $v_\b$ its traversing boundary generic Reeb vector field. Assume that Property $\mathsf A$ from Definition \ref{def.property_A} is valid. Let $\phi: (L, \d L) \hookrightarrow (X, \d_1 X)$ be a Legendrian embedding of a compact manifold $L$.  \smallskip

Then the following boundary-confined data: 
\begin{itemize}
\item the causality map $C_{v_\b}: \d_1^+X(v_\b) \to \d_1^-X(v_\b)$,
\item  the restriction $(f^\bullet)^\d: \d_1 X \to \R$ of the Lyapunov function $f^\bullet: X \to \R$ such that $df^\bullet(v_\b)=1$, 
\item the restriction $\b^\d = \b |_{\d_1X}$ of the contact form $\b$, 
\item the Lagrangian immersion  $\mathsf{int}(L^\dagger) \subset \d_1^+X(v_\b) \setminus (\mathsf X_2(v_\b) \cap \d_1^+X(v_\b))$ with respect to the symplectic form $d\b|_{\mathsf{int}(\d_1^+X(v_\b))}$,
\item the set of base points $\{b_\kappa\}_\kappa$ and the elevation values $\{\psi^\b_L(b_\kappa)\}_\kappa$ as in Corollary \ref{cor.LIFTING_L_dagger}, 
\end{itemize}
allow for a reconstruction of the Legendrian embedding $\phi: L \hookrightarrow X$, up to a diffeomorphism of $X$ that is the identity on $\d_1 X$.
\end{proposition}

\begin{proof} 
By Corollary \ref{cor.main_reconstruction}, the hypotheses in the first three bullets allow for a reconstruction of $(X, v_\b, f^\bullet, \b)$ up to a diffeomorphism that is the identity on $\d_1X$. With this quadruple in place, 
 the base points $\{b_\kappa \in \mathsf{int}(L^\dagger)\}_\kappa$ and the values $\{\psi^\b_L(b_\kappa)\}_\kappa$, determine uniquely a point $y(b_\kappa) \in L$ on each trajectory $\g_{b_\kappa}$. By Corollary \ref{cor.LIFTING_L_dagger}, the assertion follows.
\end{proof}


\begin{lemma}\label{lem.zero_volumes} Let $X$ be a compact connected smooth $2n+1$-dimensional manifold, $\b$ a contact form, and $v_\b$ its  Reeb vector field, which is traversing and boundary generic. 

Let $L \subset X$ be a sub-Legendrian (isotropic) $(2k+1)$-dimensional closed submanifold of $X$, where $2k+1 \leq n$. We assume that $L$ is in a 
general position with respect to the waterfall $\mathsf X_2(v_\b)$ (see Definition \ref{def.L_is_in_deneral_position}).  We denote by $\mathsf L(v_\b)$ the union of all $(-v_\b)$-trajectories originating in $L$, and by $\delta \mathsf L(v_\b)$  the union of $(-v_\b)$-trajectories originating in $L$ and hitting the locus $\d_2^+X(v_\b)$. The $L$-shadow $L^\dagger$ is defined as $\mathsf L(v_\b) \cap \d_1^+X(v_\b)$ (see Fig. \ref{fig.3_contact_3}). \smallskip


Then, for any $k > 0$, we get  
$\int_{\mathsf L(v_\b)}  (d\b)^{k+1}  =0$ and $\int_{L^\dagger} \b \wedge (d\b)^{k} =0.$

In contrast, for $k=0$, we still get 
 $\int_{\mathsf L(v_\b)}  d\b  = 0, \; \text{ but }  
 \;  \int_{L^\dagger} \b = - \int_{\delta \mathsf L(v_\b)} \b\; \leq \; 0.$
\end{lemma}

\begin{proof} Since $T_x L  \subset \ker\b$ 
we see that  $\b \wedge (d\b)^{k}|_L = 0$. Thus, $\int_L \b \wedge (d\b)^{k} = 0$ for all $k \in [0,  \lfloor (n-1)/2\rfloor]$.\smallskip

Note that, at a generic point $x \in \mathsf  L(v_\b)$, the $(2k+2)$-dimensional  tangent space $T_x(\mathsf L(v_\b))$  is spanned by $v_\b(x)$ and the image of the tangent space $T_y L$ under $(-v_\b)$-flow $\phi^t$ that takes $y \in L$ to $x$. Then,  $(d\b)^{k+1} |_{T_x(\mathsf L(v_\b))} = 0$, since $v_\b \in \ker(d\b)$. 
 Therefore, $\int_{\mathsf L(v_\b)} (d\b)^{k+1} = 0$.

  Let $\delta \mathsf L(v_\b)$ be the set of downward trajectories originating in $L$ and intersecting $\d_2^+X(v_\b)$  (see Fig. \ref{fig.3_contact_3}).
 At a generic point $x \in \delta \mathsf L(v_\b)$, the tangent space $T_x(\delta \mathsf L(v_\b))$  is spanned by $v_\b(x)$ and a $2k$-dimensional isotropic subspace $\mathcal K_x$ of  $T_x(\delta \mathsf L(v_\b)) \cap \xi_\b(x)$. Indeed, since $\mathsf X_2^\updownarrow(v_\b))$ intersects $L$ transversally, we may regard $\mathcal K_x$ as the image of the isotropic tangent space $T_y(L \cap \mathsf X_2^\updownarrow(v_\b))$ under the $(-v_\b)$-flow that takes $y$ to $x$ and preserves $d\b$.

 
 Thus, for $k>0$, $\b \wedge (d\b)^{k} |_{T_x(\delta \mathsf  L(v_\b))} = 0$ since $v_\b \in \ker(d\b)$ and $d\b|_{\mathcal K_x} = 0$. 
 Therefore, $\int_{\delta \mathsf L(v_\b)} \b \wedge (d\b)^{k} = 0$.
 
  The boundary $\d(\mathsf L(v_\b)) = \delta \mathsf L(v_\b) \cup L^\dagger \cup -L$. In fact, $L^\dagger$, away from the set $\delta \mathsf L(v_\b) \cap \d_1^+X(v_\b)$ of measure zero, is an immersed submanifold of $\d_1^+X(v_\b)$. 
\smallskip 
 
Recall that $d(\b \wedge (d\b)^{k}) = (d\b)^{k+1}$. Therefore, for $k > 0$, by the Stokes' theorem and using that $\int_{L} \b \wedge (d\b)^{k} = 0$,  $\int_{\delta \mathsf L(v_\b)} \b \wedge (d\b)^{k} = 0$, and $\int_{\mathsf L(v_\b)} (d\b)^{k+1} = 0$, we conclude that $\int_{L^\dagger} \b \wedge (d\b)^{k} = 0$ as well.
\smallskip

In the special case $k =0$, using a similar analysis and Stokes' theorem, we get $$\int_{L^\dagger} \b = - \int_{\delta \mathsf L(v_\b)} \b\; \leq \; 0 $$
since $\b(v_\b) = 1$.
\end{proof}

\begin{corollary}{\bf (A holographic property of Legendrian links in relation to a  concave boundary)}\label{cor.sub_Legendrian_links} 

Let $\b$ be a contact form whose Reeb vector field $v_\b$ is traversing  and boundary generic on $X$. Let  a Legendrian link\footnote{a collection of embedded disjoint loops tangent to $\xi_\b$} $L \hookrightarrow \mathsf{int}(X)$ be in a  general position with respect to the waterfall $\mathsf X_2(v_\b)$. Let $L^\dagger \subset \d_1^+X(v_\b)$ be the image of $L$ under the $(-v_\b)$-flow guided projection $X \to \d_1^+X(v_\b)$. \smallskip

Then $\int_{L^\dagger} \b < 0$, if and only if, there exists a down-flow trajectory $\g_\star$ that originates at $L$ and hits the concavity locus $\d_2^+X(v_\b)$. 

On the other hand, if $\int_{L^\dagger} \b \geq 0$, then $\int_{L^\dagger} \b = 0$, and there is no such trajectory $\g_\star$. 
In such a case, the  $L^\dagger$ is a collection of closed immersed curves.
\end{corollary}

\begin{proof} 

Since $\b(v_\b) =1$, $\int_{\delta \mathsf L(v_\b)} \b > 0$ when $\delta \mathsf L(v_\b)$ has a positive $1$-dimensional measure. This is equivalent to the existence of downward trajectory that connects $L$ to $\d_2^+X(v_\b)$. 

By Lemma \ref{lem.zero_volumes}, in the case $k=0$, we get that $\int_{\delta \mathsf L(v_\b)} \b > 0$ if and only if $\int_{L^\dagger} \b < 0$.
\end{proof}

{\it Acknowledgments:} I am grateful  to Yakov Eliashberg and John Etnyre for very enlightening  conversations. My pleasant duty is to thank Gunther Uhlmann for the much needed encouragement.  

Finally, I would like to thank the referee for exceptionally useful recommendations, which improved  a great deal the accuracy and overall quality of my presentation.

\end{document}